\documentclass[a4paper,10pt]{article}
\usepackage{a4wide}

\usepackage{amssymb,amsmath}
\usepackage{amsthm,sectsty}
\usepackage{graphicx}
\usepackage[english]{babel}
\usepackage{relsize}
\usepackage{color}
\usepackage{enumitem}
\usepackage{hyperref}
\usepackage{color}

\def\rr{\mathbb{R}} % real numbers
\def\sF{\mathcal{F}} % sigma-algebra
\def\Pr{\mathbb{P}} % P for probability
\def\EX{\mathbb{E}}	% expectation
\def\dif{\mathrm{d}} % differential operator (not partial!)
\def\landau{\mathcal{O}} % big O - Landau symbol
\def\part{\mathcal{T}_{\mathbf{h}}^N} % grid
\def\e{\textnormal{e}} % exponential
\def\eps{\epsilon}
\def\ntilde{\widetilde}
\def\nhat{\widehat}
\def\nbar{\overline}

\def\pindex{\mathcal{P}}

\def\cX{X}
\def\cE{\mathcal{E}}

\newcommand{\tn}{\textnormal}
\newcommand{\trace}{\textnormal{Tr}\,}
\def\longrightharpoonup{\relbar\joinrel\rightharpoonup} 

\newcommand{\assref}[2]{(#1\ref{#2})}

\newcounter{ass}
\theoremstyle{definition}

\newtheorem{remark}{Remark}[section]
\newtheorem{assumptions}[ass]{Assumptions}
\theoremstyle{plain}
\newtheorem{theorem}{Theorem}[section]
\newtheorem{lemma}{Lemma}[section]
\newtheorem{proposition}{Proposition}[section]
\newtheorem{corollary}{Corollary}[section]
\sectionfont{\large}
\subsectionfont{\normalsize}
\subsubsectionfont{\normalsize}

\numberwithin{equation}{section}

\newenvironment{enumerateII}
{\setlength{\leftmargini}{2.0em}
\begin{enumerate}[align=left]
  \setlength{\labelwidth}{1.0em}
  \setlength{\labelsep}{1.0em} 
  \setlength{\itemsep}{-0pt}

  }
{\end{enumerate}}

\newenvironment{enumerateI}
{\setlength{\leftmargini}{2.0em}
\begin{enumerate}[align=left]
  \setlength{\labelwidth}{0.5em}
  \setlength{\labelsep}{1.5em}
  \setlength{\itemsep}{-0pt}
  \setlength{\parsep}{0pt}  
  
  }
{\end{enumerate}}

\newenvironment{enumerate_ass}
{\setlength{\leftmargini}{2.5em}
\begin{enumerate}[align=left]
  \setlength{\labelwidth}{1.5em}
  \setlength{\labelsep}{1.0em}
  \setlength{\itemsep}{-0pt}
  \setlength{\parsep}{0pt}  
  
  }
{\end{enumerate}}

\title{\Large Spatio-Temporal Hybrid (PDMP) Models: Central Limit Theorem and Langevin Approximation for Global Fluctuations.\\ Application to Electrophysiology.  \footnote{This work has been supported by the Agence Nationale de la Recherche through the ANR Project MANDy ``Mathematical Analysis of Neuronal Dynamics'' ANR-09-BLAN-0008.}}
\author{\normalsize\textsc{Martin G.~Riedler}\\
\small\emph{Institute for Stochastics, Johannes Kepler University}\\
\small\texttt{martin.riedler@jku.at}\\[2ex]
\normalsize\textsc{Mich\`ele~Thieullen}\\
\small\emph{Universit\'e Pierre et Marie Curie}\\
\small\texttt{michele.thieullen@upmc.fr}
}

\date{}

\begin{document}
 
\maketitle 

\begin{center}
\begin{minipage}{0.8\textwidth}
\emph{Abstract:} In the present work we derive a Central Limit Theorem for sequences of Hilbert-valued Piecewise Deterministic Markov process models and their global fluctuations around their deterministic limit identified by the Law of Large Numbers. We provide a version of the limiting fluctuations processes in the form of a distribution valued stochastic partial differential equation which can be the starting point for further theoretical and numerical analysis. We also present applications of our results to two examples of hybrid models of spatially extended excitable membranes: compartmental-type neuron models and neural fields models. These models are fundamental in neuroscience modelling both for theory and numerics. 

 \medskip


\medskip

\end{minipage}
\end{center}

\section{Introduction}

The present work studies the global fluctuations of hybrid processes also called Piecewise Deterministic Markov Processes (PDMPs) on Hilbert spaces around a macroscopic limit. That is, we present a limit theorem for the rescaled fluctuations of PDMPs which admit a deterministic limit in the sense of a weak law or large numbers. We prove that under the suitable convergence of the initial conditions the rescaled fluctuations converge to a generalised Ornstein-Uhlenbeck process given by the mild solution of an Hilbert-valued stochastic evolution equation. PDMPs generalise continuous time Markov chains and possess a wide range of applications. The interest in these processes results from the fact that they combine continuous deterministic evolution with random instantaneous events which entails their specification as hybrid processes. Additionally, PDMPs arise when non-homogenous continuous time Markov chains are turned into homogeneous Markov processes by the state-space extension \cite{YinZhang}. Hence, our results also provide limit theorems for non-homogeneous Markov chains. The central limit theorem in the present study builds up on previous work \cite{RTW} where we established a Law of Large Numbers for PDMPs and a limit theorem for their internal fluctuations. The theorem in this study thus completes the `trias' of limit theorems as have been previously considered for time homogeneous Markov chains \cite{Kurtz1,Kurtz2}, PDMPs in finite dimensional Euclidean space \cite{Wainrib1} and reaction-diffusion models \cite{Blount1,Kotelenez1}. 

\medskip

In recent studies, PDMPs proved particularly useful for modelling in neuroscience and physiology under bottom-up (\cite{Austin,BuckwarRiedler,BR_neural_fields,Wainrib1}) or multi-scale approaches (\cite{GenadotThieullen1,GenadotThieullen2, Wainrib2}). Indeed mathematical models for biological real-life processes are constructed on largely varying temporal and spatial scales. Depending on the perspective what seems to be a regular behaviour on one level of description arises as the emergent effect averaged over a large number of individual dynamics on a lower level at best described in a stochastic way. Examples in neuroscience are travelling waves on excitable membranes which arise due to the gating dynamics of huge numbers of individual ion channels immersed in the membrane (cf.~\cite{Koch}) or macroscopic models of brain activity where the activity is due to a huge network of interconnected individuals neurons (cf.~\cite{Bressloff1}). However, when the macroscopic averaged pattern is observed then clearly not all stochasticity is lost. Even the higher level dynamics show random fluctuations albeit often typically small. Nevertheless, these small fluctuations can be of fundamental functional importance to systems characterised by complex non-linear responses which is the often the case in real-life systems. Hence, small fluctuations can cause tremendous changes in functionality, in the positive as in the negative way. For instance previously mentioned ion channel fluctuations can cause propagation failure of nerve signals and thus noise places limitations on the size of nerve fibres (cf.~\cite{FaisalLaughlinWhite}). In many well known examples, fluctuations are shown to be inversely proportional to the square root of the system size. In the present spatio-temporal framework however the difficulty is precisely to define what the system size is. This is an important point since determining the noise scaling constitutes one of the practically relevant applications of the kind of limit theorems we prove in the present study. Another application is the provision of a limit model of the fluctuations process that is supposed to be more tractable than the original process. This procedure is usually called diffusion or Langevin approximation. For the models in this study, this tractable model takes the form of a stochastic partial differential equation. However, as being distribution valued, it forbids a direct approach by standard approximation methods and new interpretation has to be provided first. We believe that the stochastic partial differential equation which we provide as a version of the limiting fluctuations processes can be the starting point for further theoretical and numerical analysis. This mathematically challenging and practically highly relevant point can be addressed on the basis of the present work.

\medskip

\bigskip

The remainder of the paper is organised as follows. In Sections \ref{sec_previous_results} and \ref{sec_law_of_large_numbers} we review the class of PDMPs used in this study as well as results we obtained previously which are relevant here. In particular, we present a law of large numbers for PDMPs providing a macroscopic `average' model around which the fluctuations are studied. In Section \ref{section_fluctuation_limit} we state the Central Limit Theorem which is the main result of this study and discuss its sufficient condition. The proof is presented in Section \ref{section_proofs}. In Section \ref{section_applications} we apply the limit theorem to PDMP models of excitable membranes and neural fields. The Appendix contains technical proofs of auxiliary results for the establishment of our main theorem.

\bigskip

\begin{remark} [Notation] We gather here notations regarding Hilbert spaces. We use $\ast$ for the dual of an Hilbert space as well as the adjoint of a linear operator. Round brackets $(\cdot,\cdot)_H$ denote inner product in the Hilbert space $H$. Angle brackets $\langle\cdot,\cdot\rangle_H$ denote the duality pairing in $H$, i.e., the application of a linear, bounded functional in $H^\ast$ (first argument) to an element of $H$ (second argument). Square brackets $[\cdot]$ are used to denote the elements where partial derivatives are evaluated at. Given two Hilbert spaces $X,H$ the notation $X\hookrightarrow H$ (continuous embedding) means that $X\subset H$ and the embedding operator is continuous: $\|u\|_H\leq C \|u\|_X$ for all $u\in X$, for some finite constant $C$. An evolution triplet features $X\hookrightarrow H\hookrightarrow X^\ast$, where $X$ and $H$ are Hilbert spaces; the inner product in $H$ is identified with the duality pairing in $X$:  for all $\phi\in H$ and $u\in X$, $\langle \phi,u\rangle_X=(\phi,u)_H$. These embeddings must be distinguished from the canonical identification of the Hilbert space $X$ with its dual, therefore we use for this a distinguished notation: $\iota_X$ for the canonical embedding of an element on $X$ into the dual $X^\ast$ and $\iota^{-1}_X$ for the Riesz Representation of an element in $X^\ast$ by an element in $X$. Accordingly, $\langle \iota_H(x),y\rangle_X=(x,y)_X$ for all $y\in X$ and $\iota^{-1}_X(\phi)$ is the unique element in $X$ such that $\langle\phi,y\rangle_X=(\iota^{-1}_X(\phi),y)_X$ for all $y\in X$.
\end{remark}

\section{Spatio-Temporal Piecewise Deterministic Processes}\label{sec_previous_results}

We give a brief definition of hybrid models / PDMPs relevant for the present study and refer to \cite{Davis2,Jacobsen,RiedlerPhD} for a more detailed discussion. Let $(\Omega,\sF,(\sF_t)_{t\geq 0}, \Pr)$ be a filtered probability space satisfying the usual conditions, $X$ and $H$ denote separable Hilbert spaces forming an evolution triplet $X\subset H\subset X^\ast$ and $K$ be an at most countable set. In this study all spaces and sets are equipped with their Borel-$\sigma$-fields and measurability always means Borel-measurability. Then a PDMP $(U_t,\Theta_t)_{t\geq 0}$ is a c\`adl\`ag strong Markov process taking values in $H\times K$ which is uniquely defined by the quadruple $(A,B,\Lambda,\mu)$ in the following sense:

\begin{enumerateII}

\item The operators $A: X\times K\to X^\ast$, which is linear in its $X$-argument, and $B:X\times K\to X^\ast$ are such that the abstract evolution equations
\begin{equation}\label{abstract_evolution_equation_fam}
 \dot u = A(\theta) u + B(u,\theta) \qquad\forall\,\theta\in K
\end{equation}
are well-posed in the weak sense for any initial condition $u_0\in H$, i.e., there exists a unique weak solution $u$ to \eqref{abstract_evolution_equation_fam} such that $u\in L^2((0,T),X)\cap H^1((0,T),X^\ast)\subset C([0,T],H)$ for all $T>0$ with $u(0)=u_0$ in $H$. We denote by $\phi_t(u_0,\theta)$ the value of the unique solution corresponding to the parameter $\theta$ at time $t$ started in $u_0$. Then the PDMP $(U_t,\Theta_t)_{t\geq 0}$ satisfies $U_t=\phi_{t-\tau_k}(U_{\tau_k})$ for $t\in[\tau_k,\tau_{k+1})$, where the random variables $\tau_k$, $k\in\mathbb{N}$, denote the jump-times of the PDMP with $\tau_0=0$.

\item The measurable map $\Lambda: H\times K\to \rr_+$ is locally path integrable along the solutions of \eqref{abstract_evolution_equation_fam}, i.e,
\begin{equation*}
\int_0^T \Lambda(\phi_s(u,\theta),\theta)\,\dif s\ <\ \infty \qquad\forall\,(u,\theta)\in H\times K,\,T>0\,, 
\end{equation*}
and defines the distribution of the jump times via
\begin{equation*}
 \Pr\bigl[\tau_{k+1}\geq t\,\big|\,\sF_{\tau_k}
\bigr]=\exp\Bigl(-\int_0^{t-\tau_k}\Lambda(\phi_{s}(U_{\tau_k},\Theta_{\tau_k}),\Theta_{\tau_k})\,\dif s\Bigr)\qquad\forall\, t\geq\tau_k\,.
\end{equation*}
Hence, the process $\Lambda(U_t,\Theta_t)$ states the random instantaneous jump rate of the PDMP.
\item Finally, the Markov kernel $\mu$ from $H\times K$ into $K$ defines the conditional distribution of the post jump values, i.e.,
\begin{equation*}
\Pr\bigl[(U_{\tau_k},\Theta_{\tau_k})\in A\,\big|\bigr(U_{\tau_k-},\Theta_{\tau_{k-1}})] \, =\, \Bigl(\delta_{U_{\tau_k-}}\times\mu\bigl((U_{\tau_k-},\Theta_{\tau_{k-1}}),\cdot\bigr)\Bigr)(A)
\end{equation*}
for all Borel sets $A$ in $H\times K$, with $\mu\bigl((U_{\tau_k-},\Theta_{\tau_{k-1}}),\{U_{\tau_k-}\}\times K\bigr)=1$. This implies that the components $U_t$ and $\Theta_t$ possess continuous and piecewise constant paths almost surely, respectively.
\end{enumerateII}

In this study we assume that all PDMPs are regular, i.e., $\lim_{k\to\infty}\tau_k=\infty$ a.s., and thus possess only finitely many jumps in any finite time interval almost surely. This assumption is particularly satisfied if the jump rate $\lambda$ is bounded. The domain of the extended generator of the process and its form for certain important functions is given in the next theorem, which is proven in \cite{BuckwarRiedler,RiedlerPhD}.

% REMARK:
% \marginpar{For the proof of the Theorem only the properties of being a weak solution and the embeddings of evolution triplets are used and no properties of the operators.}

\begin{theorem}\label{PDMP_gen_theorem}\begin{enumerateI}

\item A bounded, measurable function $f: H\times K\to\rr$ is in the domain of the extended generator of a PDMP if the mapping $t\mapsto f(U_t,\Theta_t)$ is absolutely continuous almost surely and the mapping $(\xi,s,\omega)\mapsto f(U_{s-},\xi)-f(U_{s-}(\omega),\Theta_{s-}(\omega))$ is integrable in the mean with respect to the random measure $\Lambda(U_{s-},\Theta_{s-})\mu\bigl((U_{s-},\Theta_{s-}),\dif\xi\bigr)\dif s$. 

\item Moreover, if in addition $f$ is continuously Fr\'echet-differentiable with respect to its first argument and such that the Riesz Representation $f_u\in H$ of the Fr\'echet derivative satisfies $f_u(u,\theta)\in X$ for $u\in X$ and is a locally bounded composition operator in $L^2((0,T),X)$, then the extended generator $\mathcal{A}f$ is given by
\begin{equation}\label{infindimGen}
\mathcal{A}f(u,\theta)=\bigl\langle A(\theta)\,u+B(\theta,u), f_u(u,\theta)\bigr\rangle_X+\Lambda(u,\theta)\int_K \Bigl(f(u,\xi)-f(u,\theta)\Bigr)\,\mu\bigl((u,\theta),\dif \xi\bigr)\,.
\end{equation}
\end{enumerateI}
\end{theorem}

An immediate consequence of Theorem \ref{PDMP_gen_theorem} is \emph{Dynkin's formula} for PDMPs. If the stochastic integral of $f$ with respect to the associated fundamental martingale measure is a martingale, then it holds that
\begin{equation}\label{Dynkin_formula}
\EX[f(Y_t)|Y_0]=f(Y_0)+\EX\Bigl[\int_0^t\mathcal{A}f(Y_s)\dif s\Big|Y_0\Bigr]\,.
\end{equation}

\begin{remark}\label{exponential_function_props} Important for our study are the functions $f$ of the following type 
\begin{equation}\label{def_of_a_special_f}
f: H\times K \to \mathbb{C}: (u,\theta)\mapsto\exp\Bigl(i\,\langle \psi,u \rangle_H+i\,\langle\Phi,z(\theta)\rangle_E\Bigr),
\end{equation}
where $E$ is a separable Hilbert space, $i$ is the complex element, $\psi\in X\hookrightarrow H^\ast$, $z:K\to E$ is a measurable map and $\Phi\in E^\ast$. We show in Appendix \ref{appendix_function_example} that the function \eqref{def_of_a_special_f} indeed satisfies the conditions of Theorem \ref{PDMP_gen_theorem} (b).
\end{remark}

\subsection{A particular class of PDMPs}

Motivated by applications, we subsequently work with specific sequences of PDMPs that we now describe. In applications, mainly the PDMP's continuous component $U$ is the variable of interest which may be experimentally observed and one is usually only interested in this part of the information the piecewise constant component $\Theta$ provides that is sufficient to define the dynamics of the continuous component. In the following we consider PDMPs indexed by $n\in\mathbb{N}$ defined by the quadruples $(A^n,B^n,\Lambda^n,\mu^n)$ on the probability spaces $(\Omega^n,\sF^n,(\sF^n_t)_{t\geq 0}, \Pr^n)$, where we assume that there exists a separable Hilbert space\footnote{In the study \cite{RTW} the functions $z^n$, called coordinate functions, are actually vector-valued with components in a Hilbert space and then $E$ corresponds to the direct product of the spaces which is a Hilbert space itself, hence no generality is lost presently.} $E$ and a sequence of measurable functions $z^n:K^n\to E$, such that 
\begin{equation*}
A^n(\theta^n)=A(z^n(\theta^n)),\qquad B^n(u,\theta)=B(u,z^n(\theta^n))\qquad\forall\,\theta^n\in K^n, n\in\mathbb{N}\,,
\end{equation*}
for some operators $A:E\to L(X,X^\ast)$ and $B:X\times E\to X^\ast$. That is, the continuous dynamics depend on the piecewise constant component only via the functions $z^n$. 
%
%That is, let $E'$ be a Hilbert space then $E=\bigtimes_{i=1}^m E'$ is a Hilbert space with inner product $(x,y)_E=\sum_{i=1}^m (x_i,y_i)_{E'}$ and thus norm $\|x\|_E^2=\sum_{i=1}^m \|x_i\|_{E'}^2$ and dual space $E^\ast=\bigtimes_{i=1}^m E^{'\ast}$.
%
%
%
Thus, actually, the limit theorems we derive for are limit theorems for the sequence $(U^n_t,z^n(\Theta^n_t))_{t\geq 0}$ (see \cite{RTW}). 

\bigskip

Towards applications, we are particularly interested in excitable membranes and neural fields. The aim is to model on the one hand a spatially extended neuronal membrane with a finite number of channels immersed and on the other hand a finite populations of neurons in a brain region. Indeed this corresponds to biologically realistic situations and to experimental conditions as well since in an imagery treatment only finite number of neurons can be observed. For these two cases, using the spatial character of the models, we consider partitions of a full spatial domain $D\subset\rr^d$ into subdomains $D^{k,n}$, $k=1,\ldots,p(n)$. In the case of an excitable membrane equipped with ion channels this is a partition of the membrane patch. Whereas in the case of neural fields it is a partition of the brain area of interest. Then the piecewise constant component $\theta^n=(\theta^{1,n},\ldots,\theta^{p(n),n})\in K^n\subseteq\mathbb{N}^{p(n)}$ counts the number of open ion channels (or the number of active neurons) in the subdomains $D^{k,n}$. The total number of channels (or neurons) in the subdomain $D^{k,n}$ is denoted by $l(k,n)<\infty$. In these two cases the coordinate functions $z^n$ are given by
\begin{equation}\label{def_neuron_coordinate}
z^n(\theta^n)=\sum_{k=1}^{p(n)} \frac{\theta^{k,n}}{l(k,n)}\,\mathbb{I}_{D^{k,n}} \in L^2(D)\,,
\end{equation}
which gives at each point in space the fraction of open ion channels (active neurons) in the respective domain. Due to the spatially piecewise constant nature of the coordinate functions in we call these PDMPs {\it Compartmental Models}.

\subsection{Quadratic variation associated with PDMPs}

Before proceeding further, we now discuss a fundamental fact regarding the aim of deriving limit theorems in our framework: the decomposition
\begin{equation}\label{proof_lln_decomposition}
z^n(\Theta^n_t)=z^n(\Theta^n_0)+\int_0^t \bigr(\mathcal{A} z^n\bigl)(U^n_t,\Theta^n_t) \dif t +M^n_t
\end{equation}
of the piecewise constant component of the PDMP. The integral in the right hand side of \eqref{proof_lln_decomposition} exists in the sense of Bochner and the process $M^n$ is a c\`adl\`ag, square-integrable $E$-valued martingale under Assumptions \ref{assumptions_LLN} displayed in the next section. As \eqref{proof_lln_decomposition} can be considered a stochastic evolution equation driven by the martingale $M^n$, one can think of these martingales containing the inherent stochasticity in the process. To prove a central limit theorem one has to investigate more closely the structure of the martingale $M^n$ in (\ref{proof_lln_decomposition}) and establish a non-trivial limit under suitable rescaling.

\bigskip

The quadratic variation $G^n\in L(\cE,\cE^\ast)$ associated with the PDMPs $(U^n,\Theta^n)$ is the mapping $G^n$ defined via a bilinear form on $\cE$ given by
\begin{eqnarray}\label{definition_quad_var_operator}
\lefteqn{(\Phi,\Psi)\mapsto\langle G^n(U^n_t,\Theta^n_t)\,\Phi,\Psi \rangle_\cE\ =}\nonumber\\[2ex]
&&\phantom{xxx}\Lambda^n(U^n_t,\Theta^n_t)\int_{K^n}\Bigl(\langle z^n(\xi)-z^n(\Theta^n_t),\Phi \rangle_\cE\Bigr)\Bigl(\langle z^n(\xi)-z^n(\Theta^n_t),\Psi \rangle_\cE\Bigr)\,\mu^n\bigl((U^n_t,\Theta^n_t),\dif\xi\bigr).\phantom{xxxx}
\end{eqnarray}
Then it holds that the operators $G^n(U^n,z^n)$ are symmetric, positive and, anticipating the condition \eqref{ass_C_generator_diff_bound_1} below, of trace class. Furthermore, for the martingale part in \eqref{proof_lln_decomposition} it holds
\begin{eqnarray}
\EX^n\|M^n_t\|_{\cE^\ast}^2&=&\EX^n\int_0^t\trace G^n(U^n_t,\Theta^n_t)\,\dif t\\[2ex]
&=&\EX^n\Bigl[\int_0^t\Lambda^n(U^n_t,\Theta^n_t)\int_{K^n}\| z^n(\xi)-z^n(\Theta^n_t)\|_{\cE^\ast}^2\,\mu^n\bigl((U^n_t,\Theta^n_t),\dif\xi\bigr)\,\dif t\Bigr]\,.\nonumber
\end{eqnarray}
For a more detailed discussion of quadratic variations associated with PDMPs and proofs of the cited results we refer to \cite[Section 3]{RTW}.

\section{Law of Large Numbers}\label{sec_law_of_large_numbers}

In this section we review the Law of Large Numbers obtained in \cite{RTW} which establishes the convergence of the PDMP sequence to a deterministic trajectory. Then we present a simple corollary relevant for the present study.  \medskip

In our framework the deterministic limit solves an evolution problem as follows. Let $F:X\times E\to E$ be an operator such that the deterministic abstract evolution system
\begin{equation}\label{det_limit_process}
\left.\begin{array}{rcl}\dot u & = & A(p)u+B(u,p),\\[2ex] \dot p &=& F(u,p)
\end{array}\right.
\end{equation}
is well-posed in the following sense: for suitable initial condition $(u_0,p_0)\in H\times E$ the system \eqref{det_limit_process} possesses a unique global weak solution such that for all $T>0$ the component $u$ is in $H^1((0,T),X^\ast)\cap L^2((0,T),X)\subset C([0,T],H)$ and $p$ is in $ H^1((0,T),E)\subset C([0,T],E)$. We now describe the set of assumptions (Assumptions \ref{assumptions_LLN}) under which we proved the Law of Large Numbers.

\begin{assumptions}\label{assumptions_LLN} We assume the following conditions are satisfied:
\begin{enumerate_ass}
\item\label{ass_A_lipschitz_conds} The operators $A,\,B$ and $F$ satisfy almost surely Lipschitz-type conditions uniformly in $n$ along the deterministic and stochastic solutions, i.e., for all $T>0$ there exist constants $L_1,\, L_2$, independent of $n$, such that almost surely
\begin{eqnarray*}
\lefteqn{\int_0^T\langle A(z^n(\Theta^n_t))\,U^n_t-A(p(t))\,u(t),U^n_t-u(t)\rangle_X
+\langle B(U^n,z^n(\Theta^n))-B(u,p),U^n-u\rangle_X\,\dif t}\nonumber\\ &&\phantom{xxxxxxxxxxxxxxxxxxxxx}\leq\ L_1\int_0^T\|U^n_t-u(t)\|^2_H+
\|z^n(\Theta^n_t)-p(t)\|_E^2\,\dif t\,.\phantom{xxxxxx}
\end{eqnarray*}
and
\begin{equation*}
\|F(u,p)-F(U^n,z^n(\Theta^n))\|_{L^2((0,T),E)}\leq L_2 \|(u,p)-(U^n,z^n(\Theta^n))\|_{L^2((0,T),H\times E)}
\end{equation*}
Note that the constants $L_1,\, L_2$ may depend on $T$ and $(u,p)$.
\item\label{ass_A_IC_convergence} The initial conditions converge in probability, i.e., for all $\eps>0$
\begin{equation*}
\lim_{n\to\infty}\Pr^n\Bigl[\|U^n_0-u_0\|_H^2+\|z^n(\Theta_0^n)-p_0\|_E^2>\eps\,\Bigr]=0\,.
\end{equation*}
\item\label{ass_A_fluid_limit_ass} The fluid limit assumption holds in probability, i.e., for all $T,\eps >0$
\begin{equation*}
\lim_{n\to\infty}\Pr^n\Bigl[\int_0^T\big\|\bigl[\mathcal{A}^n\langle\,\cdot\,,z^n(\cdot)\rangle_E\bigr](U^n_s,\Theta^n_s) -  F(U_s^n,z^n(\Theta^n_s))\big\|_E^2\,\dif s > \eps\Bigr] =0.
\end{equation*}
\item\label{ass_A_second_moments} The second moments of the jump sizes vanish in the limit, i.e., for all $T>0$
\begin{equation*}
\lim_{n\to\infty}\, \EX^n\,\int_0^T\Bigl[\Lambda^n(U^n_s,\Theta^n_s)\int_{K_n} \|z^n(\xi)-z^n(\Theta^n_s)\|_E^2\,\mu^n\bigl((U^n_s,\Theta^n_s),\dif\xi\bigr)\Bigr]\,\dif s\,=\,0\,.
\end{equation*}
\end{enumerate_ass}
\end{assumptions}

\begin{theorem}[Law of Large Numbers {\cite[Thm.~4.1]{RTW}}]\label{theorem_lln} Under the \tn{Assumptions \ref{assumptions_LLN}} the sequence $(U^n_t,z^n(\Theta^n_t))_{t\geq 0}$ converges uniformly on compacts in probability (u.c.p.) to the solution $(u(t),p(t))_{t\geq 0}$ of \eqref{det_limit_process}, that is, for all $T,\,\eps>0$
\begin{equation}\label{the_LLN_property}
\lim_{n\to\infty}\Pr^n\Bigl[\sup\nolimits_{t\in[0,T]}\Bigl(\|U^n_t-u(t)\|_H^2+\|z^n(\Theta^n_t)-p(t)\|_E^2\Bigr)>\eps\Bigr]=0\,.
\end{equation}
\end{theorem}

Heuristically, the proof of Theorem \ref{theorem_lln} (law of large numbers) is done by inserting into the decomposition \eqref{proof_lln_decomposition} the limits for the initial condition and the integral given in assumptions \assref{A}{ass_A_IC_convergence} and \assref{A}{ass_A_fluid_limit_ass} and noting that the martingale $M^n$ in \eqref{proof_lln_decomposition} converges to zero due to assumption \assref{A}{ass_A_second_moments}. The resulting equation is just the equation for $p$ in \eqref{det_limit_process}. 

\begin{remark} The u.c.p.~convergence \eqref{the_LLN_property} immediately extends to any normed spaces the Hilbert spaces $H$ and $E$ are continuously embedded in. %\footnote{As the state space of the process is a separable metric space being an $D([0,T],H\times E)$--valued random variable and being a stochastic process on $[0,T]$ with paths in $H\times E$ is equivalent \cite[p.~128]{EthierKurtz}} %Furthermore, Theorem \ref{theorem_lln} still holds if $X$ is a not a Hilbert space but only a reflexive Banach space.\marginpar{It might be also possible if $E$ is only a Banach space}
Moreover u.c.p.~convergence implies that the processes also converge weakly on the space of c\`adl\`ag functions $D(\rr_+,H\times E)$. % The local uniform norm is stronger than the Skorokhod norm (JacodShiryaev p.292) and applying markov inequality we obtain that convergence in probability on the Skrokhof space and thus weak convergence is implied by the metric introduced by the the u.c.p. convergence.
\end{remark}

Conversely to the extension of the law of large numbers to spaces with weaker norms, it can also - albeit not immediately - be extended to spaces of higher spatial and/or temporal regularity, provided that the continuous component of the PDMP sequence and the deterministic limit are sufficiently regular. The extension of the results works along the lines of standard estimation techniques from the analysis of (linear) evolution equation (see, e.g., \cite{Evans}). (For a specific model, a more direct method was used in \cite{Austin}.) A first extension, which we employ later in this study, is stated in the following corollary.

\begin{corollary} Let $T>0$ be fixed and assume that the operator $A$ is independent of $p$ and satisfies the energy estimate
\begin{equation*}
\gamma_1\|u\|_X^2\leq-\langle A\,u,u\rangle_X+\gamma_2\|u\|_H^2 \quad \forall\, u\in X
\end{equation*}
for constants $\gamma_1>0$ and $\gamma_2\geq 0$ and that the mapping $B: H\times E\to X^\ast$ is continuous. Then for all $T,\eps \geq 0$,
\begin{eqnarray*}
\lim_{n\to\infty}\Pr^n\bigl[\|U^n-u\|_{L^2((0,T),X)}>\eps\bigr] =0\,.
\end{eqnarray*}
\end{corollary}

\begin{proof} By definition\footnote{It holds in between jumps for almost all $t$ and as the process is regular, there are almost surely only finitely many jumps. Thus for each $\omega$ the possible set, it does not hold, is a finite union of Lebesgue null sets, hence a null set.} of a weak solution, the following holds almost surely for almost all $t\in[0,T]$,
\begin{eqnarray*}
\lefteqn{\langle\dot U^n_t-\dot u(t), U^n_t-u(t) \rangle_X}\\[1ex]
&=&\langle A\,U^n_t-A\,u(t),U^n_t-u(t)\rangle_X+\langle B(U^n_t,z^n(\Theta_t^n))-B(u(t),p(t)),U^n_t-u(t)\rangle_X\,.
\end{eqnarray*}
The energy estimate yields
\begin{equation*}
\gamma_1\|U^n_t-u(t)\|_X^2\,\leq\,\gamma_2\|U^n_t-u(t)\|_H^2+\|B(U^n_t,z^n(\Theta^n_t))-B(u(t),p(t))\|_{X^\ast}\|U^n_t-u(t)\|_X
\end{equation*}
and due to Young's inequality we get for all $\eps>0$
\begin{equation*}
(\gamma_1-\eps)\,\|U^n_t-u(t)\|_X^2\,\leq\,\gamma_2\|U^n_t-u(t)\|_H^2+\tfrac{1}{\eps}\,\|B(U^n_t,z^n(\Theta^n_t))-B(u(t),p(t))\|_{X^\ast}^2\,.
\end{equation*}
Choosing $\eps$ small enough such that the left hand side of this inequality is positive and integrating, we find that 
\begin{equation*}
\|U^n-u\|_{L^2((0,T),X)}^2\,\leq K\,\|U^n-u\|_{L^2((0,T),H)}^2+K\,\|B(U^n,z^n(\Theta^n))-B(u,p)\|_{L^1((0,T),X^\ast)}^2,
\end{equation*}
for some constant $K$. Here the right hand side is a random variable that converges to zero in probability due to Theorem \ref{theorem_lln} and the Continuous Mapping Theorem. The corollary now follows immediately.
\end{proof}

From Theorem \ref{theorem_lln} we see that all stochasticity of the PDMP sequence is lost in the limit since this theorem provides a connection of the sequence to a deterministic, macroscopic equation. At this point two types of fluctuations come to mind: one can either consider only the internal fluctuations of the piecewise constant component $z^n(\Theta^n)$ or the global fluctuations PDMP sequence. In \cite{RTW} the internal fluctuations were characterised via a central limit theorem for the martingale part of the PDMPs, i.e., the process $M^n$ in \eqref{proof_lln_decomposition}. In this study we consider the global fluctuations of the PDMP sequence around its deterministic limit.

%%%%%%%%%%%%%%%%%%%%%%%%%%%%%%%%%%
%%%%%%%%%%%%%%%%%%%%%%%%%%%%%%%%%%
%%
%% Fluctuations around equilibrium
%%
%%%%%%%%%%%%%%%%%%%%%%%%%%%%%%%%%%
%%%%%%%%%%%%%%%%%%%%%%%%%%%%%%%%%%

\section{Central Limit Theorem for global fluctuations}\label{section_fluctuation_limit}

In this paper we investigate the fluctuations of a PDMP sequence around its deterministic limit given by the law of large numbers of Theorem \ref{theorem_lln}. Hence we are naturally concerned with the differences $U^n-u$ and $z^n(\Theta^n)-p$ which we have to rescale in order to obtain a non-trivial limit. We thus introduce a real, positive sequence $\alpha_n$, $n\in\mathbb{N}$, necessarily diverging, and make the stochastic process given by
\begin{equation}\label{def_fluctuations_process}
Z^n_t:= \sqrt{\alpha_n}\bigl(U^n_t-u(t),z^n(\Theta^n_t)-p(t)\bigr), \ t\geq 0,
\end{equation}
the central object of interest in our study. Such a rescaled fluctuations process is classical in the study of central limit theorems. The main result we presently establish is that the sequence of processes \eqref{def_fluctuations_process} admits a converging subsequence with limit supported on the space of continuous functions and with a further strengthening the limit is a unique diffusion process, see Theorem \ref{fluctuation_limit_theorem}.\medskip

It is important to note that for central limit theorems in infinite-dimension it is in general not possible to prove the convergence in the norm of the Hilbert space $H\times E$ for which the Law of Large Numbers holds, nor does the limit exist in this space. However, we can obtain convergence in weaker norms: in the literature limits for spatio-temporal models are often considered in the Schwartz space of distributions. We consider the limits in duals of subspaces of $H$ and $E$.
A natural choice for the subspace of $H$ is the space $X$ in the evolution triplet. The convergence will take place in its dual $X^\ast$. In order to keep the presentation simple we restrict ourselves to this choice.
For the limit of the coordinate functions, we introduce a Hilbert space $\cE\hookrightarrow E$ usually obtained by the completion with respect to a different inner product defined on a subspace of $E$. Clearly, together with the duals these spaces form an evolution triplet with $E\subset\cE^\ast$ and the norms satisfy $\|z\|_{\cE^\ast}\leq C\|z\|_E$ for all $z\in E$ for some constant $C$. In the application to neuroscience models $E:=L^2(D)$ and $\cE^\ast$ is given by the dual to some Sobolev space $H^s(D)$. For a more detailed discussion see Remark \ref{Hilbert_scales}. 
When establishing tightness of the sequence \eqref{def_fluctuations_process} we assume for technical matters that there exists a third Hilbert space $V$ satisfying the chain of continuous embeddings
\begin{equation*}
 \cE\hookrightarrow V\hookrightarrow E\hookrightarrow V^\ast\hookrightarrow \cE^\ast.
\end{equation*}

In applications these embeddings of spaces may collapse, i.e., some or all the spaces coincide. This is indeed the case with finite-dimensional Hilbert spaces where all the above spaces coincide with an Euclidean space. This property renders limit theorems in finite dimensions much simpler.

\subsection{The main result on global fluctuations}

We now proceed to our main result. We will establish the Central Limit Theorem under the following set of assumptions that we call Assumptions C. 

\setcounter{ass}{002}
\begin{assumptions}\label{assumptions_CLT} 

\begin{enumerate_ass}

\item\label{ass_C_lipschitz_conds} The operators $A,\,B$ and $F$ satisfy Lipschitz-type conditions, uniformly in $n$, along the stochastic and deterministic solutions: for all $T>0$ 
\begin{eqnarray*}
\lefteqn{\int_0^T\langle A(z^n(\Theta^n_t))\,U^n_t-A(p(t))\,u(t),U^n_t-u(t)\rangle_X
+\langle B(U^n,z^n(\Theta^n))-B(u,p),U^n-u\rangle_X\,\dif t}\nonumber\\ &&\phantom{xxxxxxxxxxxxxxxxxxxxx}\leq\ L_3\int_0^T\|U^n_t-u(t)\|^2_H+
\|z^n(\Theta^n_t)-p(t)\|_{V^\ast}^2\,\dif t,\phantom{xxxxxx}
\end{eqnarray*}

\begin{equation*}
\|F(u,p)-F(U^n,z^n(\Theta^n))\|_{L^2((0,T),V^\ast)}\leq L_4 \|(u,p)-(U^n,z^n(\Theta^n))\|_{L^2((0,T),H\times V^\ast)},
\end{equation*}

\begin{equation*}
\|F(u,p)-F(U^n,z^n(\Theta^n))\|_{L^2((0,T),\cE^\ast)}\leq L_5 \|(u,p)-(U^n,z^n(\Theta^n))\|_{L^2((0,T),H\times \cE^\ast)}
\end{equation*}
almost surely for some constants $L_3,\, L_4$ and $L_5$ which do not depend on $n$ but may depend on $T$ or $(u,p)$.

\item\label{ass_C_IC_bounds} The sequence of initial conditions satisfies
%
% IMPORTANT REMARK
%
%\footnote{Using here the stronger space $H$ for the continuous component is no problem as we can choose the deterministic initial condition $U^n_0=u_0$ for all $n$ or small deterministic perturbations $U^n_0=u_0+\landau(\sqrt{\alpha_n}^{-1})$}
\begin{equation*}
\sup_{n\in\mathbb{N}}\sqrt{\alpha_n}\,\EX^n\|U^n_0-u_0\|_H+\sup_{n\in\mathbb{N}}\sqrt{\alpha_n}\,\EX^n\|z^n(\Theta^n_0)-p_0\|_{V^\ast} < \infty\,.
\end{equation*}

\item\label{ass_C_generator_diff_bound} For all $T>0$,
\begin{equation}\tag{C3.1}\label{ass_C_generator_diff_bound_1}
\sup_{n\in\mathbb{N}}\alpha_n\,\EX^n\int_0^T\Bigl[\Lambda^n(U^n_s,\Theta^n_s)\int_{K_n}\|z^n(\xi)-z^n(\Theta^n_s)\|_{V^\ast}^2\,\mu^n\bigl((U^n_s,\Theta^n_s),\dif\xi\bigr)\,\dif s\Bigr]\,<\,\infty,
\end{equation}
and
\begin{equation}\tag{C3.2}\label{ass_C_generator_diff_bound_2}
\sup_{n\in\mathbb{N}}\alpha_n\int_0^T\EX^n\big\|\bigl[\mathcal{A}^n\langle z^n_j(\cdot),\,\cdot\,\rangle_V\bigr](U^n_s,\Theta^n_s) -  F(U^n_s,z^n(\Theta^n_s))\big\|_{V^\ast}^2\,\dif s \,<\, \infty\,.
\end{equation}

\item\label{ass_C_tightness} There exists an orthonormal basis $(\varphi_k)_{k\in\mathbb{N}}$ of $\cX\times\cE$ such that for all $\delta>0$ and almost all $t\geq 0$
\begin{equation*}
\lim_{m\to\infty} \sup_{n\in\mathbb{N}}\Pr^n\Bigl[\sum\nolimits_{k>m}|\langle Z^n_t,\varphi_k\rangle_{\cX\times\cE}|^2>\delta\Bigr]=0 \,.
\end{equation*} 

\item\label{ass_C_convergences} For all $\Phi\in\cE$ and all $T>0$,
\begin{equation}\tag{C5.1}\label{ass_C_covariance_convergence}
\lim_{n\to\infty} \int_0^T\EX^n\big|\bigl\langle G(u(s),p(s))\,\Phi,\Phi\bigr\rangle_\cE-\alpha_n\bigl\langle G^n(U^n_s,\Theta^n_s)\,\Phi,\Phi\bigr\rangle_\cE\big|\,\dif s\, =\, 0\,,
\end{equation}
where $G(u(s),p(s))\in L(\cE,\cE^\ast)$, $s\geq 0$, is an integrable family of symmetric, positive, trace class operators, and
\begin{equation*}\tag{C5.2}\label{ass_C_generator_convergence}
\lim_{n\to\infty}\sqrt{\alpha_n}\int_0^T\EX^n\big|\bigl[\mathcal{A}\langle z^n(\cdot),\Phi\rangle_\cE\bigr](U^n_s,\Theta^n_s)-\bigl\langle F(U^n_s,z^n(\Theta^n_s)),\Phi\bigr\rangle_\cE\big|\,\dif s =0\,.
\end{equation*}

\item\label{ass_B_large_fluctuations} For all $\Phi\in\cE$ and all $T,\,c>0$,
\begin{equation}\tag{C6.1}\label{ass_B_large_fluctuations_1}
\lim_{n\to\infty}\,\EX^n\Bigl[\int_0^T\Lambda^n(U^n_s,\Theta^n_s)\int_{\sqrt{\alpha_n}\,|\langle z^n(\xi)-z^n(\Theta_s^n),\Phi\rangle_\cE|>c}\,\mu^n\bigl((U^n_s,\Theta^n_s),\dif\xi\bigr)\,\dif s\Bigr]\,=\, 0
\end{equation}
and the jump heights of the rescaled martingales are almost surely uniformly bounded, i.e., 
\begin{equation}\tag{C6.2}\label{clt_martingale_cond_1}
\sup_{n\in\mathbb{N}}\sup_{t\geq 0}\,\sqrt{\alpha_n}\,\|z^n(\Theta^n_t)-z^n(\Theta^n_{t-})\|_{\cE^\ast}<\infty \tn{ a.s.}
\end{equation}

\item\label{ass_C_differentiability} The operators $A:X\times E\to X^\ast$, $B:X\times E\to X^\ast$ and $F:X\times E\to E$ are partially continuously Gateaux differentiable along the solution $(u,p)$ of the deterministic system and for almost all $s\geq 0$, the Gateaux derivatives evaluated at $(u(s),p(s))$ possess a bounded, Lipschitz continuous extension to $\cX^\ast$, respectively $\cE^\ast$, which we denote by $\ntilde{\phantom{x}}$. For instance, if $B_u[u,p]\in L(X,X^\ast)$ denotes the partial Gateaux derivative of $u\mapsto B(u,p)$ evaluated at $(u,p)\in X\times E$, then $\ntilde B_u[u,p]$ is its extension to $\in L(\cX^\ast,\cX^\ast)$ such that 
\begin{equation*}
\ntilde B_u[u,p] h = B_u[u,p] h  \quad\forall\, h\in X\,. 
\end{equation*}

\item\label{ass_C_uniform_bound} For some $p>2$ and for all $T>0$,
\begin{equation*}
\sup_{n\in\mathbb{N},\,t\leq T}\EX^n\|U^n_t-u(t)\|^p_X+\sup_{n\in\mathbb{N},\,t\leq T}\EX^n\|z^n(\Theta^n_t)-p(t)\|^p_E \,<\,\infty\,.
\end{equation*}

\end{enumerate_ass}
\end{assumptions}

We now state the main result of this study. Remember that
\begin{equation}
Z^n_t:= \sqrt{\alpha_n}\bigl(U^n_t-u(t),z^n(\Theta^n_t)-p(t)\bigr), \ t\geq 0\,.
\end{equation}

\begin{theorem}[Central Limit Theorem for the global fluctuations]\label{fluctuation_limit_theorem} We assume that the \tn{Assumptions \ref{assumptions_LLN}} and  \tn{\ref{assumptions_CLT}} are satisfied. Then the laws of the processes $(Z^n_t)_{t\geq 0}$ are tight in $D(\rr_+,\cX^\ast\times\cE^\ast)$ and any accumulation point $\Pr^\ast$ is supported on the space of continuous functions, i.e.,
\begin{equation}\label{support_of_the_limit_result}
\Pr^\ast\bigl(C(\rr_+,\cX^\ast\times\cE^\ast)\bigr)=1\,. 
\end{equation}
Furthermore, the limit is unique, i.e., the processes $(Z^n_t)_{t\geq 0}$ converge in distribution to an $\cX^\ast\times\cE^\ast$--valued process with characteristic function $\varphi: \rr_+\times \cX\times \cE\to\rr$ satisfying the partial differential equation
\begin{eqnarray}\label{the_complex_equality_in_theorem}
\frac{\partial\varphi}{\partial t}&=& -\frac{1}{2}\bigl\langle G(u(s),p(s))\Phi,\,\Phi\bigr\rangle_\mathcal{E}\,\varphi+ \Bigl \langle\frac{\partial\varphi}{\partial \psi},\ntilde A^\ast(p(s))\,\psi+\ntilde B_u^\ast[u(s),p(s)]\,\psi+\ntilde F_u^\ast[u(s),p(s)]\,\Phi\Bigr\rangle_\cX\nonumber \\[2ex]
&&\mbox{} + \Bigl\langle\frac{\partial\varphi}{\partial\Phi},\ntilde A_p^\ast[u(s),p(s)]\psi+\ntilde B_p^\ast[u(s),p(s)]\,\psi+\ntilde F_p^\ast[u(s),p(s)]\,\Phi\Bigr\rangle_\cE
\end{eqnarray}
where $\varphi$ and its partial derivatives are evaluated at $(t,\psi,\Phi)$, with initial condition $\varphi(0,\psi,\Phi)=\lim_{n\to\infty} \EX^n \e^{i\langle Z_0^n,(\psi,\Phi)\rangle_{\cX\times\cE}}$. Moreover, if the initial conditions $Z^n_0$ converge to a Gaussian random variable on $\cX^\ast\times\cE^\ast$, then the limit is a Gaussian process.
\end{theorem}

\begin{remark}  It is easy to see that Theorem \ref{fluctuation_limit_theorem} is an extension to previous results considering PDMPs on Euclidean spaces. In finite dimensions linear operators are matrices with the adjoint operator given by the matrix transpose and thus equation \eqref{the_complex_equality_in_theorem} reduces to the equation reported in \cite{Wainrib1}. Furthermore, reducing to piecewise constant finite dimensional processes, i.e., $A=B=0$, we obtain the equation as in Kurtz \cite{Kurtz2}. 

We stated that in previous work we presented a limit for the internal fluctuations of the PDMPs (cf.~\cite[Thm.~5.1]{RTW}). The present Assumptions \ref{assumptions_CLT} are stronger than the assumptions needed for that limit theorem (if we assume that the embedding of $\cE\in V$ is of Hilbert-Schmidt type as in applications is always the case). Hence under the Assumptions \ref{assumptions_CLT} it follows that the $\cE^\ast$--valued martingales $(\sqrt{\alpha_n}\,M_t^n)_{t\geq 0}$ converge weakly to an $\cE^\ast$--valued centred, continuous Gaussian process with independent increments characterised by the covariance operator
\begin{equation}
C(t)=\int_0^t G(u(s),p(s))\circ\iota^{-1}_\cE\,\dif s \qquad\forall\,t \geq 0\,. 
\end{equation}
Finally, in \cite{RiedlerPhD,RTW} a second version of the martingale central limit theorem with slightly different conditions was considered, i.e., a marginally stronger condition instead of \eqref{ass_B_large_fluctuations_1} and a marginally weaker condition instead of \eqref{clt_martingale_cond_1}. We note that Theorem \ref{fluctuation_limit_theorem} remains valid if the Assumptions \ref{assumptions_CLT} are changed accordingly.

\end{remark}

It is important to note that if the limit is a Gaussian process, it may be identified with the mild solution of a linear stochastic partial differential equation with additive noise, i.e., a generalised infinite dimensional Ornstein-Uhlenbeck process.

\begin{corollary}\label{corollary_spde_representation} Assume that \tn{Theorem \ref{fluctuation_limit_theorem}} holds and the initial conditions $Z^n_0$ converge to a Gaussian limit $(U_0,P_0)\in \cX^\ast\times\cE^\ast$. Then the limiting process is a version of the mild solution to the linear, non-autonomous stochastic evolution equation
\begin{equation}\label{general_spde}
\left.\begin{array}{rcl}
\dif U_t &=& \Bigl[\bigl(\ntilde A(p(t))+\ntilde B_u[u(t),p(t)]\bigr) U_t + \ntilde F_u[u(t),p(t)] P_t\Bigr]\,\dif t\\[2ex] 
\dif P_t &=& \Bigl[\bigl(\ntilde A_p[p(t)] +\ntilde B_p[u(t),p(t)]\bigr) U_t + \ntilde F_p[u(t),p(t)] P_t\Bigr]\,\dif t + g(t)\,\dif W^Q_t 
\end{array}\right.
\end{equation}
with initial condition $(U_0,P_0)$, where $W^Q$ is a Wiener process on a Hilbert space $H'$ with covariance $Q\in L(H')$ and $g(t)\in L(H',\cE^\ast)$ such that $g(t)Qg^\ast(t)=G(u(t),p(t))\circ\iota^{-1}_\cE$.
\end{corollary}

\begin{remark} It is always possible to find operators $g$ in \eqref{general_spde}. For instance, choose $U=\cE^\ast$ and $Q=I$ being the identity. Then $g(t)$ is the unique square root of the trace class operator $G(u(t),p(t))\circ\iota^{-1}_\cE$.
\end{remark}

\subsection{Discussion of the assumptions}\label{discussion_of_assumptions}

Before proceeding to the proof of Theorem \ref{fluctuation_limit_theorem}, we discuss the set of Assumptions C that we have chosen to rely on.

\subsubsection{Proving tightness under Assumptions \ref{assumptions_CLT}}

In the sequel we prove convergence in distribution of a sequence of c\`adl\`ag processes. This is a well established question in Probability Theory for which sufficient conditions as well as the plan of the proof are well identified: one first proves tightness of the given sequence and then shows that the set of accumulation points is a singleton using some characterisation. Tightness of probability measures on the space of c\`adl\`ag functions taking values in $\cX^\ast\times\cE^\ast$ is equivalent to tightness of the sequences of marginals (in a dense subset of $\rr_+$) and convergence in probability for the mode of continuity. These properties are classically obtained by checking the conditions \assref{T}{ass_T1}--\assref{T}{ass_T3} below. The first property is implied by \assref{T}{ass_T1} and \assref{T}{ass_T2} (which are also necessary), the second one is implied by condition \assref{T}{ass_T3} due to theorem \cite[Thm.~8.6, Chap.~3]{EthierKurtz}.

\setcounter{ass}{020}
\begin{enumerate_ass}
\item\label{ass_T1} There exists an increasing sequence of finite dimensional subspaces $S_k$, $k\in\mathbb{N}$ of $\cX^\ast\times\cE^\ast$ such that $\lim_{k\to\infty} \pi_k x =x$ for all $x\in \cX^\ast\times\cE^\ast$, where $\pi_k$ denotes the projection onto $S_k$, and for every $\eps,\,\delta>0$ and all $t$ out of a dense subset of $\rr_+$, there exists $S_m$ out of $(S_k)_{k\in\mathbb{N}}$ satisfying 
\begin{equation*}
\sup_{n\in\mathbb{N}} \Pr^n\bigl[\,\inf\nolimits_{y\in S_m}\|Z^n_t-y\|_{\cX^\ast\times\cE^\ast}>\delta\bigr]\leq\eps\,. 
\end{equation*}

\item\label{ass_T2} For all $\eps>0$ and all $t$ out of a dense subset of $\rr_+$, there exists a $\delta>0$ such that
\begin{equation*}
\sup_{n\in\mathbb{N}}\,\Pr^n[\,\|Z_t^n\|_{\cX^\ast\times\cE^\ast}>\delta]\leq \eps\,. 
\end{equation*}

\item\label{ass_T3} For every $T>0$, all $\Delta t\in(0,1)$ all $n\in\mathbb{N}$, there exist non-negative random variables $\delta^n(\Delta t)$ such that for all $0\leq t<T$ and $0<s\leq\Delta t$
\begin{equation}\label{inequality_T3}
\EX^n\Bigl[\big\|Z^n_{t+s}-Z^n_t\big\|_{\cX^\ast\times\cE^\ast}^2\,\Big|\,\sF^n_t\Bigr] \,\leq\,\EX^n\bigl[\delta^n(\Delta t)\big|\,\sF^n_t\Bigr]
\end{equation}
and
\begin{equation}\label{conditaion_on_delta_T3}
\lim_{\Delta t\to 0} \limsup_{n\to\infty} \EX^n[\delta^n(\Delta t)]=0\,. 
\end{equation}
\end{enumerate_ass}
\setcounter{ass}{003}

However in the present paper we found it more natural to work with the set of Assumptions C. We now show that Assumptions C imply \assref{T}{ass_T1}--\assref{T}{ass_T3}: 

To this end we first note that if $S_n$ are the spaces spanned by an increasing set of elements of the orthonormal basis obtained from the Riesz Representation of the orthonormal basis $(\varphi_k)_{k\in\mathbb{N}}$ in $\cX\times\cE$ then, assumption \assref{C}{ass_C_tightness} is just a reformulation of condition \assref{T}{ass_T1}. To proceed we need the following lemma, which is also of later use.

\begin{lemma}\label{lemma_uniform_integrable} Assume that \tn{\assref{C}{ass_C_lipschitz_conds}} holds. Then there exist constants $K_1,\,K_2$ independent of $n\in\mathbb{N}$ and $t\leq T$ such that for all $n\in\mathbb{N}$ and $t\leq T$,
\begin{eqnarray}\label{proof_tightness_gronwall_est}
\lefteqn{\alpha_n\,\EX^n\Bigl[\|U^n_t-u(t)\|_H^2+\|z^n(\Theta^n_t)-p(t)\|_{V^\ast}^2\Bigr]}\nonumber\\[1ex]
&\leq& \Bigl(\alpha_n\,\EX^n\|U^n_0-u_0\|_{H}^2+\alpha_n\,\EX^n\|z^n(\Theta^n_0)-p_0\|^2_{V^\ast} +\alpha_n\,\EX^n\sup_{s\in[0,t]}\|M^n_s\|_{V^\ast}^2\Bigr)\,K_1\e^{K_2t}\\
&&\mbox{}+ \Bigl(\alpha_n\,\EX^n\int_0^t\bigl\|\mathcal{A}^n\bigl(\langle\cdot,z^n(\cdot)\rangle_E\bigr)(U^n_s,\Theta^n_s) - F(U^n_s,z^n(\Theta^n_s))\bigr\|_{V^\ast}^2\,\dif s \Bigr)\,t K_1\e^{K_2t}\,.\nonumber
\end{eqnarray}
Assume moreover  that \tn{\assref{C}{ass_C_IC_bounds}} and \tn{\assref{C}{ass_C_generator_diff_bound}} also hold. Then the family of random variables
\begin{equation}\label{uniform_integrability_family}
\Bigl\{ \sqrt{\alpha_n}\|U^n_t-u(t)\|_{H}+\sqrt{\alpha_n}\|z^n(\Theta^n_t)-p(t)\|_{V^\ast};\, t\in [0,T], n\in\mathbb{N}\Bigr\}
\end{equation}
is uniformly integrable for every $T>0$.
\end{lemma}

We postpone the proof of Lemma \ref{lemma_uniform_integrable} to the next section and prove now that \assref{T}{ass_T2} and \assref{T}{ass_T3} hold under Assumptions C. 
We see that Lemma \ref{lemma_uniform_integrable} combined with the continuous embeddings $H\hookrightarrow\cX^\ast$ and $V^\ast\hookrightarrow\cE^\ast$, and the inequality $\sqrt{(a^2+b^2)}\leq |a|+|b|$, implies that also the set $\{\|Z^n_t\|_{\cX^\ast\times\cE^\ast};\, t\in[0,T],n\in\mathbb{N}\}$ is uniformly integrable. Thus there exists an increasing, non-negative, convex function $g:[0,\infty)\to\rr_+$ such that $\lim_{x\to\infty} g(x)/x=\infty$ and
\begin{equation}
 \sup_{t\in[0,T],\,n\in\mathbb{N}} \EX^n\bigl[g(\|Z^n_t\|_{\cX^\ast\times\cE^\ast})\bigr]\,<\infty,
\end{equation}
\cite[Appendix, Prop.~2.1]{EthierKurtz}. Now, using Chebyshev's inequality we obtain
\begin{equation*}
\Pr^n[\,\|Z_t^n\|_{\cX^\ast\times\cE^\ast}>\delta]\leq\frac{1}{g(\delta)}\,\EX^n\bigl[g(\|Z^n_t\|_{\cX^\ast\times\cE^\ast})\bigr]\,.
\end{equation*}
Taking the supremum over all $n\in\mathbb{N}$ in the right hand side of this inequality and choosing $\delta$ sufficiently large yields \assref{T}{ass_T2}.\medskip

Moreover, using the norm of $H\times\cE^\ast$ yields
\begin{eqnarray*}
\lefteqn{\big\|Z^n_{t+s}-Z^n_t\big\|_{H\times\cE^\ast}^2}\\[2ex]
& =& \ \alpha_n\,\|U^n_{t+s}-u(t+s)-U^n_t+u(t)\|_H^2+\alpha_n\,\|z^n(\Theta^n_{t+s})-p(t+s)-z^n(\Theta^n_t)-p(t)\|_{\cE^\ast}^2.
\end{eqnarray*}
Next the estimation techniques used to prove \cite[Thm.~4.1]{RTW} yield the inequalities
\begin{eqnarray}\label{Kurtz_aldous_est_1}
\lefteqn{\|U^n_{t+s}-u(t+s)-U^n_t+u(t)\|_H^2}\\[1ex]
&\leq& 4L_3\int_t^{t+s}\|U^n_{r}-u(r)-U^n_t+u(t)\|_H^2+\|z^n(\Theta^n_{r})-p(r)-z^n(\Theta^n_t)-p(t)\|_{V^\ast}^2\,\dif r\nonumber\\[1ex]
&&\mbox{} + 4sL_3\,\Bigl(\|U^n_t-u(t)\|_H^2+\|z^n(\Theta^n_t)-p(t)\|_{V^\ast}^2\Bigr),\nonumber
\end{eqnarray}
and
\begin{eqnarray}\label{Kurtz_aldous_est_2}
\lefteqn{\|z^n(\Theta^n_{t+s})-p(t+s)-z^n(\Theta^n_t)-p(t)\|_{\cE^\ast}^2}\\[1ex]
&\leq& 6L_5\int_t^{t+s}\|U^n_{r}-u(r)-U^n_t+u(t)\|_H^2+\|z^n(\Theta^n_{r})-p(r)-z^n(\Theta^n_t)-p(t)\|_{\cE^\ast}^2\,\dif r\nonumber\\[1ex]
&&\mbox{} +3\,\|M^n_{t+s}-M^n_t\|_{\cE^\ast}^2 + 3sL_5\Bigl(\|U^n_t-u(t)\|_H^2+\|z^n(\Theta^n_t)-p(t)\|_{\cE^\ast}^2\Bigr)\nonumber\\[1ex]
&&\mbox{} + 3s\int_t^{t+s}\bigl\|\mathcal{A}^n\bigl(\langle z^n(\cdot),\cdot\rangle_\cE\bigr)(U^n_r,\Theta^n_{r}) - F(z^n(\Theta^n_r),U^n_r)\bigr\|_{\cE^\ast}^2\,\dif r,\nonumber
\end{eqnarray}
where $L_3$ and $L_5$ are the Lipschitz constants in assumption \assref{C}{ass_C_lipschitz_conds}. Remember that $G^n$ is the quadratic variation defined in \eqref{definition_quad_var_operator}. Then, for fixed $t\geq 0$, the process
\begin{equation*}
s\mapsto\|M^n_{t+s}-M^n_t\|_{\cE^\ast}^2-\int_t^{t+s}\textnormal{Tr}\,G^n(U^n_r,\Theta^n_r)\,\dif r
\end{equation*}
is a martingale with respect to $(\sF^n_{t+s})_{s\geq 0}$ satisfying
\begin{equation*}
\EX^n\bigl[\|M^n_{t+s}-M^n_t\|_{\cE^\ast}^2\big|\sF^n_t\bigr]\,=\,\EX^n\Bigl[\int_t^{t+s}\textnormal{Tr}\,G^n(U^n_r,\Theta^n_r)\,\dif r\Big|\sF^n_t\Bigr]\,.
\end{equation*}
This can be proven using the martingale results developed in \cite[Section 3]{RTW}. Now, a combination of the estimates \eqref{Kurtz_aldous_est_1} and \eqref{Kurtz_aldous_est_2}, employing the continuous embeddings $\cE\hookrightarrow V$ and the conditional Fubini Theorem, see, e.g., \cite[Thm.~1.1.7]{Applebaum}, for interchanging conditional expectation and integration, yields
\begin{eqnarray*}
\lefteqn{\EX^n\bigl[\alpha_n^{-1}\|Z^n_{t+s}-Z^n_t\|_{H\times\cE^\ast}^2\big|\sF^n_t\bigr]}\\[1ex]
&\leq& K_1\,\int_t^{t+s}\EX^n\bigl[\alpha_n^{-1}\|Z^n_{r}-Z^n_r\|_{H\times V^\ast}^2\big|\sF^n_t\bigr]\,\dif r  + \EX^n\Bigl[\int_t^{t+s}\textnormal{Tr}\,G^n(U^n_r,\Theta^n_r)\,\dif r\Big|\sF^n_t\Bigr]\\[1ex]
&&\mbox{} + s\, K_1\Bigl(\EX^n\|U^n_t-u(t)\|_H^2+\EX^n\|z^n(\Theta^n_t)-p(t)\|_{V^\ast}^2\Bigr)\\[1ex]
&&\mbox{} + s\,K_1\EX^n\Bigl[\int_t^{t+s}\bigl\|\mathcal{A}^n\bigl(\langle z^n(\cdot),\cdot\rangle_\cE\bigr)(U^n_r,\Theta^n_{r}) - F(z^n(\Theta^n_r),U^n_r)\bigr\|_{\cE^\ast}^2\,\dif r\Big|\sF^n_t\Bigr]
\end{eqnarray*}
for a suitable finite constant $K_1$ independent of $n\in\mathbb{N}$. We then apply Gronwall's lemma to this estimate and multiply the resulting inequality by $\alpha_n$. Thus, as $s\leq \Delta t$, we obtain the estimate
\begin{eqnarray*}
\lefteqn{\EX^n\bigl[\|Z^n_{t+s}-Z^n_t\|_{H\times\cE^\ast}^2\big|\sF^n_t\bigr]}\\[1ex]
&\leq& \alpha_n\,\EX^n\Bigl[\int_t^{t+\Delta t}\textnormal{Tr}\,G^n(U^n_r,\Theta^n_r)\,\dif r\Big|\sF^n_t\Bigr]\e^{K_1\Delta t}\\[1ex]
&&\mbox{} +\Delta t\,K_1\Bigl(\alpha_n\|U^n_t-u(t)\|_H^2+\alpha_n\|z^n(\Theta^n_t)-p(t)\|_{V^\ast}^2\Bigr)\e^{K_1\Delta t}\\[1ex]
&&\mbox{} + \Delta t\,K_1\,\alpha_n\,\EX^n\Bigl[\int_t^{t+s}\bigl\|\mathcal{A}^n\bigl(\langle z^n(\cdot),\cdot\rangle_\cE\bigr)(U^n_r,\Theta^n_{r}) - F(z^n(\Theta^n_r),U^n_r)\bigr\|_{\cE^\ast}^2\,\dif r\Big|\sF^n_t\Bigr]\e^{K_1\Delta t}\,.
\end{eqnarray*}
Let us define the random variable $\delta^n(\Delta t)$ in \eqref{inequality_T3} as being the right hand side of this inequality. This choice is relevant since the left hand side is an upper bound to $\EX^n\bigl[\|Z^n_{t+s}-Z^n_t\|^2_{\cX^\ast\times\cE^\ast}\big|\sF^n_t\bigr]$ due to the continuous embeddings $H\hookrightarrow \cX^\ast$. Then this random variable satisfies
\begin{eqnarray}\label{final_tightness_proof_estimate}
\limsup_{n\to\infty} \EX^n[\delta^n(\Delta t)]& \leq& K_2\Delta t+K_2\, \limsup_{n\to\infty}\EX^n\Bigl[\alpha_n\int_t^{t+\Delta t}\textnormal{Tr}\,G^n(U^n_r,\Theta^n_r)\,\dif r\Bigr]\nonumber\\[1ex]
& \leq& K_2\Delta t+K_2\int_t^{t+\Delta t} \textnormal{Tr}\,G(u(s),p(s))\,\dif s
\end{eqnarray}
for a suitable constant $K_2$ independent of $n$. Here the first estimate is due to condition \assref{C}{ass_C_generator_diff_bound} and Lemma \eqref{lemma_uniform_integrable}. The second one follows from the convergence result in Proposition \ref{app_trace_con_prop} in the appendix. Since the upper bound \eqref{final_tightness_proof_estimate} converges to zero when $\Delta t\to 0$, condition \assref{T}{ass_T3} is satisfied.

\subsubsection{Norms and spaces}

In our assumptions we have to choose the right norms and spaces carefully. In applications the choices for the correct spaces $\cX^\ast$ and $\cE^\ast$ are usually determined by the need to establish condition \assref{C}{ass_C_tightness}. If the embeddings $\cX\hookrightarrow H$ and $\cE\hookrightarrow V$ are of Hilbert-Schmidt type then \assref{C}{ass_C_tightness} is satisfied under the conditions \assref{C}{ass_C_lipschitz_conds} -- \assref{C}{ass_C_generator_diff_bound}, see Remark \ref{embedding_remark}. Thus, in applications one usually chooses first the space $V^\ast$ small enough such that conditions \assref{C}{ass_C_lipschitz_conds}--\assref{C}{ass_C_generator_diff_bound} hold. Then the spaces $\cX$ and $\cE$ are chosen large enough such that the embeddings $\cX\hookrightarrow H$ and $\cE\hookrightarrow V$ are of Hilbert-Schmidt type. Choosing largest possible spaces $V,\cX,\cE$ one obtains the best possible regularity of the limiting process.
Secondly, condition \eqref{ass_C_generator_diff_bound_1} implies a condition akin to \assref{A}{ass_A_second_moments} and thus the processes $(\sqrt{\alpha_n}\,M^n_t)_{t\geq 0}$ are square-integrable $V^\ast$-- and $\cE^\ast$--valued c\`adl\`ag martingales under this assumption. Moreover the assumption \assref{C}{ass_C_generator_diff_bound} implies the same statements for the space $\cE^\ast$.
Concerning assumption \assref{C}{ass_C_differentiability}, we note that the required extensions exist if the embeddings are dense which is usually the case in applications. 
Also, in applications \assref{C}{ass_C_uniform_bound} is either a consequence of the Law of Large Numbers or follows directly from the path properties of the PDMPs.

\section{Proof of the Theorem \ref{fluctuation_limit_theorem} (Central Limit Theorem)}\label{section_proofs}

We split the proof of Theorem \ref{fluctuation_limit_theorem} into several parts. In the first part, Section \ref{proof_section_tightness}, we establish tightness of the sequence of laws of $Z^n$ in \eqref{def_fluctuations_process} and thus existence of a weakly converging subsequence. The short second part, Section \ref{proof_section_support}, shows that any weak limit of a converging subsequence is a continuous stochastic process. The central part of the proof is the third one, Section \ref{proof_section_characterisation}. Therein we first prove that the characteristic function of the limit satisfies the partial differential equation \eqref{the_complex_equality_in_theorem}. To this end we have adapted a proof from \cite{Kurtz2,Wainrib1} to the present infinite-dimensional setting. The main difficulties here are technical in nature by identification and careful treatment of corresponding infinite dimensional concepts that substantially simplify in Euclidean space, that is, the use of general functional analytic operator theory instead of linear algebra for matrices and vectors is necessary. We show that equation \eqref{the_complex_equality_in_theorem} possesses a solution and the solution characterises a diffusion process.

\subsection{Tightness}\label{proof_section_tightness}

We now prove Lemma \ref{lemma_uniform_integrable}. Remember that in the preceding section this lemma was the key to prove that conditions \assref{T}{ass_T1}--\assref{T}{ass_T3}, which ensure tightness, holds true under Assumptions C. Hence, tightness will follow.
 
\medskip

\begin{proof}[Proof of Lemma \ref{lemma_uniform_integrable}.] Due to the de la Vall\'ee-Poussin Theorem \cite[Appendix, Prop.~2.2]{EthierKurtz} it is sufficient to prove that the family \eqref{uniform_integrability_family} possesses uniformly bounded second moments. Using the first two Lipschitz conditions in \assref{C}{ass_C_lipschitz_conds} and Gronwall's inequality yields the almost sure inequality
\begin{eqnarray*}
\|U^n_t-u(t)\|_H^2+\|z^n(\Theta^n_t)-p(t)\|_{V^\ast}^2
\!\!&\!\!\leq&\!\! \Bigl(\|U^n_0-u_0\|_{H}^2+\|z^n(\Theta^n_0)-p_0\|^2_{V^\ast} +\sup_{s\in[0,t]}\|M^n_s\|_{V^\ast}^2\Bigr)\,K_1\e^{K_2t}\\
&&\hspace{-35pt}\mbox{}+ \Bigl(\int_0^t\bigl\|\mathcal{A}^n\bigl(\langle\cdot,z^n(\cdot)\rangle_V\bigr)(U^n_s,\Theta^n_s) - F(z^n(\Theta^n_s),U^n_s)\bigr\|_{V^\ast}^2\,\dif s \Bigr)\,K_1\e^{K_2t},
\end{eqnarray*}
where the constants $K_1,\,K_2$ are independent of $n\in\mathbb{N}$. Applying Cauchy-Schwarz inequality to the last term, taking expectations and multiplying by $\alpha_n$ yields \eqref{proof_tightness_gronwall_est}. Let us show now that each expectation in the right hand side of  \eqref{proof_tightness_gronwall_est} is bounded over all $n\in\mathbb{N}$ and $t\leq T$: Indeed the terms containing the initial conditions are bounded due to assumption \assref{C}{ass_C_IC_bounds}. The boundedness of the martingale term holds due to Doob's inequality and condition \eqref{ass_C_generator_diff_bound_1} as $\|M^n\|_{V^\ast}^2$ is a submartingale.
The last term is bounded by condition \eqref{ass_C_generator_diff_bound_2}. This proves the lemma.
\end{proof}

The results so far have the following consequence which will be useful in the remainder of the proof. As we now have established tightness there exists a weak limit $Z_t$ to a subsequence of the sequence $Z^n_t$ in $\cX^\ast\times\cE^\ast$. Without loss of generality we identify in the following $Z^n$ with a converging subsequence. Then, the Continuous Mapping Theorem and the established uniform integrability of $Z^n$ imply, due to \cite[Appendix, Prop.~2.3]{EthierKurtz}, the convergence of expectations
\begin{equation}\label{integrability_of_limit}
\lim_{n\to\infty}\EX^n\|Z^n_t\|_{\cX^\ast\times\cE^\ast}\ = \ \EX \|Z_t\|_{\cX^\ast\times\cE^\ast} <\infty \qquad\forall\, t\geq 0\,.
\end{equation}

\begin{remark}\label{embedding_remark} Before closing this section let us mention another consequence of Lemma \ref{lemma_uniform_integrable}. Let $\varphi_k=(\varphi_k^\cX,\varphi_k^\cE)$, $k\in\mathbb{N}$, be an orthonormal basis of $\cX\times\cE$. The Markov inequality yields
\begin{equation*}
\Pr^n\Bigl[\sum_{k>m}|\langle Z^n_t,\varphi_k\rangle_{\cX\times\cE}|^2>\delta\Bigr]\ \leq\ \frac{1}{\delta}\EX^n\Bigl[\sum\nolimits_{k>m}\alpha_n\,\langle U^n_t-u(t),\varphi^\cX_k\rangle_{\cX}^2+\alpha_n\,\langle z^n(\Theta^n_t)-p(t),\varphi^\cE_k\rangle^2_{\cE}\Bigr] 
\end{equation*}
and employing the properties of the evolution triplet we further estimate
\begin{eqnarray*}
\lefteqn{\Pr^n\Bigl[\sum_{k>m}|\langle Z^n_t,\varphi_k\rangle_{\cX\times\cE}|^2>\delta\Bigr]}\\
&&\phantom{xxxxxxx} \leq\ \frac{1}{\delta}\EX^n\Bigl[\sum\nolimits_{k>m}\bigl(\|\varphi^\cX_k\|_H^2+\|\varphi^\cE_k\|_V^2\bigr)\,\alpha_n\,\Bigl(\|U^n_t-u(t)\|_H^2+\alpha_n\,\|z^n(\Theta^n_t)-p(t)\|^2_{V^\ast}\Bigr)\Bigr]\,.
\end{eqnarray*}
Then the additional assumption that the embeddings $\cX\hookrightarrow H$ and $\cE\hookrightarrow V$ are of Hilbert-Schmidt type implies that the sum $\sum\nolimits_{k>m}\bigl(\|\varphi^\cX_k\|_H^2+\|\varphi^\cE_k\|_V^2\bigr)$ is finite and vanishes for $m\to\infty$. Thus
\begin{eqnarray*}
\lefteqn{\Pr^n\Bigl[\sum_{k>m}|\langle Z^n_t,\varphi_k\rangle_{\cX\times\cE}|^2>\delta\Bigr]}\\
&&\leq\ \frac{1}{\delta}\Bigl(\sum\nolimits_{k>m}\bigl(\|\varphi^\cX_k\|_H^2+\|\varphi^\cE_k\|_V^2\bigr) \Bigr)\sup\nolimits_{n\in\mathbb{N},\,t\leq T}\Bigl(\alpha_n\,\EX^n\bigl[\|U^n_t-u(t)\|_H^2+\|z^n(\Theta^n_t)-p(t)\|^2_{V^\ast}\bigr]\Bigr)\,.
\end{eqnarray*}
Lemma \ref{lemma_uniform_integrable} ensures that the supremum is bounded. Hence in this case \assref{C}{ass_C_tightness} follows from \assref{C}{ass_C_lipschitz_conds} -- \assref{C}{ass_C_generator_diff_bound}.
\end{remark}

\subsection{Support of the limit \eqref{support_of_the_limit_result}}\label{proof_section_support}

Let us show that any weak limit to a subsequence of the sequence $Z^n$ is a continuous process. This result is an immediate consequence of the Martingale Central Limit Theorem \cite[Thm.~5.1]{RTW} and its method of proof. Therein it was shown that the rescaled martingales converge to a continuous process and for this result it is sufficient that
\begin{equation}\label{vanishing_martingale_jumps}
 \lim_{n\to\infty}\EX\sum_{s\in(0,t]}\alpha_n\|\Delta M_t^n\|^2_{\cE^\ast}=\EX\sum_{s\in(0,t]}\|\Delta M_t\|^2_{\cE^\ast} = 0
\end{equation}
holds for all $t>0$, where $\Delta_t M^n = M^n_t-M^n_{t-}$ and $M$ denotes the weak limit of the martingales $\sqrt{\alpha_n} M^n$. The property \eqref{vanishing_martingale_jumps} was established in \cite{RTW} under assumptions implied by the Assumptions \ref{assumptions_CLT}. An analogous result immediately follows in the present case for the processes $Z^n$. Note that
\begin{equation*}
\|\Delta_t Z^n\|^2_{\cX^\ast\times\cE^\ast}\,=\,\alpha_n\|\Delta_t(U^n-u)\|^2_{\cX^\ast}+\alpha_n\|\Delta_t(z^n(\Theta^n)-p)\|_{\cE^\ast}^2\,=\, 0 +\alpha_n\|\Delta_t M^n\|_{\cE^\ast}^2\,.
\end{equation*}
Thus 
\begin{equation*}
\lim_{n\to\infty}\EX\sum_{s\in(0,t]}\|\Delta_tZ^n\|^2_{\cX^\ast\times\cE^\ast}=0
\end{equation*}
follows immediately from \eqref{vanishing_martingale_jumps}. Then we infer from \cite[Chap.~3, Thm.~1.2]{EthierKurtz} that the limit possesses almost surely continuous paths.\medskip

\subsection{Characterisation of the limit}\label{proof_section_characterisation}

In the course of the proof we have so far obtained that there exists a weak limit to a subsequence of the processes $Z^n$, which is an a.s.~continuous process on $\cX^\ast\times\cE^\ast$. In this final part we now characterise this limit as a particular diffusion process and show that it is unique. We use a characterisation by convergence of the finite-dimensional distributions. 
Let us introduce some notation. We denote by $(Z_t)_{t\geq 0}$ a version of the limiting process defined on some probability space $(\Omega,\sF,\Pr)$ and $\EX$ denotes the expectation with respect to the measure $\Pr$. For given $\psi\in\cX\subset H$ and $\Phi\in\cE$ we denote by $h$ the bounded, continuous mapping
\begin{equation}\label{def_of_function_h}
h: H\times\cE^\ast\to\mathbb{C}: (x,y)\to \exp\bigl(i\langle x,\psi\rangle_\cX+i\langle y,\Phi\rangle_{\mathcal{E}}\bigr) 
\end{equation}
for which we also use the shorter notation $h(z)=h(x,y)$ for $z=(x,y)$. Moreover, due to the chain rule $h$ is partially Fr\'echet differentiable with derivatives
\begin{equation*}
\frac{\partial h}{\partial x}(x,y)= i h(x,y)\,\psi\, \in\, L(H,\mathbb{C})\triangleq H,\quad \frac{\partial h}{\partial y}(x,y)= i h(x,y)\,\Phi\, \in\, L(\cE^\ast,\mathbb{C})\triangleq\cE\,.
\end{equation*}
That is, the derivatives lie in the corresponding dual spaces, i.e., $H^\ast=H$ and $\cE$, respectively, considered as complex Hilbert spaces. Conversely, $h$ can be considered as a function of $\psi$ and $\Phi$ for given $x$ and $y$ in which case we obtain the partial derivatives
\begin{equation*}
\frac{\partial h}{\partial \psi}(x,y)= i h(x,y)\,x\, \in\, L(\cX,\mathbb{C})\triangleq\cX^\ast,\quad \frac{\partial h}{\partial \Phi}(x,y)= i h(x,y)\,y\, \in\, L(\cE,\mathbb{C})\triangleq\cE^\ast\,.
\end{equation*}
By the use of this notation the characteristic function $\varphi^n$ of the random variable $Z^n_t\in\cX^\ast\times\cE^\ast$ satisfies
\begin{equation*}
\varphi^n(t,\psi,\Phi)=\EX^n h(Z^n_t)=\EX^n h\Bigl(\sqrt{\alpha_n}(U^n_t-u(t)),\sqrt{\alpha_n}(z^n(\Theta^n_t)-p(t))\Bigr)\,. 
\end{equation*}
Note here that $U^n_t-u(t)$ is considered a random variable in $\cX^\ast$, however, it actually takes values in $H$ which is embedded in $\cX^\ast$ in the sense of evolution triplets.
We show in Appendix \eqref{appendix_function_example} that $h$ defined in \eqref{def_of_function_h} is in the domain of the extended generator of the PDMP $(U^n_t,\Theta^n_t,u(t),p(t))_{t\geq 0}$.\footnote{It is clear that the process $(U^n_t,\Theta^n_t,u(t),p(t))_{t\geq 0}$ is a PDMP as the components $(U^n_t,\Theta^n_t)_{t\geq 0}$ form a PDMP and also $(u(t),p(t))_{t\geq 0}$ are a (degenerate) continuous Markov process as being the solution of a deterministic abstract evolution equation.} 
Thus Dynkin's formula \eqref{Dynkin_formula} yields \footnote{To be precise, here Dynkin's formula is applied to the map $$ (U,\Theta,u,p)\mapsto \exp\Big(i\sqrt{\alpha_n}\langle U-u,\psi\rangle_{V}+i\sqrt{\alpha_n}\langle\Phi,z^n(\Theta)-p\rangle_{\mathcal{E}}\Bigr) $$ which is in the domain of the extended generator of the PDMP $(U^n_t,\Theta^n_t,u(t),p(t))_{t\geq 0}$ satisfying the special conditions in Theorem \ref{PDMP_gen_theorem}.}
\begin{eqnarray*}
\lefteqn{\varphi^n(t,\psi,\Phi)-\EX^n h\Bigl(\sqrt{\alpha_n}(U^n_0-u_0),\sqrt{\alpha_n}(z^n(\Theta^n_0)-p_0)\Bigr)}\\[2ex]
&\hspace{-30pt}=&\hspace{-20pt}\int_0^t\EX^n\Bigl[\bigl\langle A(z^n(\Theta^n_s))\,U^n_s+B(U^n_s,z^n(\Theta^n_s))-A(p(s))\,u(s)-B(u(s),p(s)),\psi\bigr\rangle_\cX\,i\sqrt{\alpha_n}\, h(Z^n_s)\\[2ex]
&&\hspace{-20pt}\phantom{xxxxxxx}\mbox{}- \bigl\langle F(p(s),u(s)),\Phi\bigr\rangle_\cE\,i\sqrt{\alpha_n}\,h(Z^n_s)\\[2ex]
&&\hspace{-20pt}\phantom{xxxxxxx}\mbox{}+ \Lambda^n(U^n_s,\Theta^n_s)\int_{K_n}\Bigl(h\bigl(\sqrt{\alpha_n}(U^n_s-u(s)),\sqrt{\alpha_n}(z^n(\xi)-p(s))\bigr)\\[2ex]
&&\hspace{-20pt}\phantom{xxxxxxxxxxxxxxxxxxxx}\mbox{}- h\bigl(\sqrt{\alpha_n}(U^n_s-u(s)),\sqrt{\alpha_n}(z^n(\Theta^n_s)-p(s))\bigr)\Bigr)\,\mu^n\bigl((U^n_s,\Theta^n_s),\dif\xi\bigr)
\Bigr]\,\dif s\\[2ex]
&=:& \int_0^t\EX^n\bigl[ H_1(s)+H_2(s)\bigr]\,\dif s,
\end{eqnarray*}
where the terms $H_1$ and $H_2$ denote the continuous and jump part of the generator, respectively. Introducing the function
\begin{equation*}
K(u)=e^{iu}-1-iu+\tfrac{1}{2}u^2,
\end{equation*}
for which $K(u)=\landau(u^3)$ holds, we obtain
\begin{eqnarray*}
H_2(s)&=&\Lambda^n(U^n_s,\Theta^n_s)h(Z^n_s)\int_{K_n}\Bigl[\exp\Bigl(\sqrt{\alpha_n}\langle z^n(\xi)-z^n(\Theta^n_s),\Phi\rangle_\cE\Bigr)-1\Bigr]\,\mu^n\bigl((U^n_s,\Theta^n_s),\dif\xi\bigr)\\[2ex]
&=& \Lambda^n(U^n_s,\Theta^n_s)h(Z^n_s)\int_{K_n}\Bigl[i\sqrt{\alpha_n}\langle z^n(\xi)-z^n(\Theta^n_s),\Phi\rangle_\cE\\[2ex]
&&\mbox{}-\frac{\alpha_n}{2}\langle z^n(\xi)-z^n(\Theta^n_s),\Phi\rangle_\cE^2+K\Bigl(\sqrt{\alpha_n}\langle z^n(\xi)-z^n(\Theta^n_s),\Phi\rangle_\cE\Bigr)\Bigr]\,\mu^n\bigl((U^n_s,\Theta^n_s),\dif\xi\bigr)\,.
\end{eqnarray*}
Next note that (see \cite[Sec.~3]{RTW} for details),
\begin{equation*}
\Lambda^n(U^n_s,\Theta^n_s)\int_{K_n}\Bigl( \langle z^n(\xi),\Phi\rangle_\cE-\langle z^n(\Theta^n_t),\Phi\rangle_\cE\Bigr)\,\mu^n\bigl((U^n_s,\Theta^n_s),\dif\xi\bigr)= \bigl[\mathcal{A}\langle z^n(\cdot),\Phi\rangle_\cE\bigr](U^n_s,\Theta^n_s)\,,
\end{equation*}
where $\mathcal{A}\langle z^n(\cdot),\Phi\rangle_\cE$ denotes the application of the generator of the PDMP to the map $(U,\Theta,u,p)\mapsto \langle z^n(\Theta),\Phi\rangle_\cE$. \medskip

Therefore we obtain that $\varphi^n(t,\psi,\Phi)-\varphi^n(0,\psi,\Phi) $ is the sum of five integral terms that we denote by (i) to (v). Next we make these terms explicit and we discuss their respective limits for $n\to\infty$. We start with (i), (iv) and (v) which are simpler before proceeding to the more demanding terms (ii) and (iii).

%%
%%
%%  TERM 1
%%
%%

\subsubsection{The term (i)}

The term (i) is equal to 
\begin{equation*}
-\frac{1}{2}\int_0^t\EX^n\Bigl[\alpha_n\,\Lambda^n(U^n_s,\Theta^n_s)\int_{K_n}\langle z^n(\xi)-z^n(\Theta^n_s),\Phi\rangle_\cE^2\, \mu^n\bigl((U^n_s,\Theta^n_s),\dif\xi\bigr)\,h(Z^n_s)\Bigr]\,\dif s
\end{equation*}
which can also be written as
\begin{equation*}
-\frac{1}{2}\,\alpha_n\int_0^t\EX^n\Bigl[h(Z^n_s)\,\langle G^n(U^n_s,\Theta^n_s)\Phi,\Phi\rangle_\cE\Bigr]\,\dif s\,.
\end{equation*}
This term converges to
\begin{equation*}
-\frac{1}{2}\int_0^t\langle G(u(s),p(s))\Phi,\Phi\rangle_\mathcal{E}\,\EX\bigl[h(Z_s)\bigr]\,\dif s, 
\end{equation*}
where $\EX\bigl[h(Z_s)\bigr]=\varphi(s,\psi,\Phi)$ is the characteristic function of the limit process, due to
\begin{eqnarray*}
\lefteqn{\Bigl|\int_0^t\alpha_n\EX^n\Bigl[h(Z^n_s)\,\langle G^n(U^n_s,\Theta^n_s)\Phi,\Phi\rangle_\cE\Bigr]\,\dif s-\int_0^t\langle G(u(s),p(s))\Phi,\Phi\rangle_\cE\,\EX\bigl[h(Z_s)\bigr]\,\dif s\Bigr|}\\
&\leq& \int_0^t\EX^n\Bigl[|h(Z^n_s)|\,\big|\alpha_n\langle G^n(U^n_s,\Theta^n_s)\Phi,\Phi\rangle_\cE-\langle G(u(s),p(s))\Phi,\Phi\rangle_\cE\big|\Bigr]\,\dif s\\
&&\mbox{}+ \Bigl|\int_0^t\langle G(u(s),p(s))\Phi,\Phi\rangle_\cE\,\EX\bigl[h(Z^n_s)\bigr]\,\dif s-\int_0^t\langle G(u(s),p(s))\Phi,\Phi\rangle_\cE\,\EX\bigl[h(Z_s)\bigr]\,\dif s\,\Bigr|\,.
\end{eqnarray*}
Indeed, the first integral in the right hand side converges to zero due to $|h(Z^n_s)|\leq 1$ and condition \eqref{ass_C_covariance_convergence}. The second integral converges to zero as well due to the Bounded Convergence Theorem as $\EX[h(Z_s^n)]\to \EX[h(Z_s)]$ for all $s\geq 0$ by definition of the weak convergence.

%%
%%
%%  TERM 4
%%
%%

\subsubsection{The term (iv)}

The term (iv) is equal to
\begin{equation*}
\int_0^t i\sqrt{\alpha_n}\,\EX^n\Bigl[\Bigl(\bigl[\mathcal{A}\langle z^n(\cdot),\Phi\rangle_\cE\bigr](U^n_s,\Theta^n_s)-\bigl\langle F(U^n_s,z^n(\Theta^n_s)),\Phi\bigr\rangle_\cE\Bigr)\,h(Z^n_s)\Bigr]\,\dif s\,.
\end{equation*}
It vanishes in the limit due to condition \eqref{ass_C_generator_convergence} and the following estimate:
\begin{eqnarray*}
\tn{(iv)}&\leq&\int_0^t \sqrt{\alpha_n}\,\EX^n\Bigl[\Big|\bigl[\mathcal{A}\langle z^n(\cdot),\Phi\rangle_\cE\bigr](U^n_s,\Theta^n_s)-\bigl\langle F(U^n_s,z^n(\Theta^n_s)),\Phi\bigr\rangle_\cE\Big|\,|h(Z_s^n)|\Bigr]\,\dif s\\
&\leq& \sqrt{\alpha_n}\int_0^t\EX^n\big|\bigl[\mathcal{A}\langle z^n(\cdot),\Phi\rangle_\cE\bigr](U^n_s,\Theta^n_s)-\bigl\langle F(U^n_s,z^n(\Theta^n_s)),\Phi\bigr\rangle_\cE\big|\,\dif s\,.
\end{eqnarray*}

%%
%%
%%  TERM 5
%%
%%

\subsubsection{The term (v)}

The term (v) is equal to
\begin{equation*}
\int_0^t\EX^n\Bigl[\Lambda^n(U^n_s,\Theta^n_s)\int_{K_n}K\Bigl(\sqrt{\alpha_n}\langle z^n(\xi)-z^n(\Theta^n_s),\Phi\rangle_\cE\Bigr)\,\mu^n\bigl((U^n_s,\Theta^n_s),\dif\xi\bigr)\,h(Z^n_s)\Bigr]\,\dif s\,.
\end{equation*}
In order to study its limit, note first that assumption \eqref{ass_B_large_fluctuations_1} implies the existence of a sequence $\beta_n$, $n\in\mathbb{N}$, converging to zero\footnote{For example, we can construct such a sequence in the following way: Let $a_n(\beta)$ denote the sequence elements in the limit of \eqref{ass_C_covariance_convergence} which satisfy $\lim_{n\to\infty} a_n(\beta)=0$ for all $\beta>0$. The aim is to construct a sequence $\beta_n$ satisfying (A) $\lim_{n\to\infty}\beta_n=0$ and (B) $\lim_{n\to\infty} a_n(\beta_n)=0$.

We first define a subsequence of $\beta_n$ by $\beta_{n_m}=m^{-1}$ for $m\in\mathbb{N}$ where the strictly increasing values of the indices $n_m$ has to be defined. For $n_m< n < n_{m+1}$ we set $\beta_n=m^{-1}$. This sequence converges to zero and hence the condition (A) on the sequence is satisfied. Choosing cleverly the values of the indices $n_m$ allows for also the condition (B) to be satisfied.

We define these values recursively: Choose $n_1$ to be the smallest integer $n'$ such that $a_{n}(1)<1$ for all $n\geq n'$ which exists due to \eqref{ass_C_covariance_convergence}. Then for the finitely many $n<n_1$ we can set $\beta_n$ arbitrary, say, equal to 1, as these sequence elements are insignificant for the limit. Now we have $m=1$. Next set for $n_{m+1}$ some integer $n''$ larger than $n_m$ which satisfies $a_{n}((m+1)^{-1})\leq (m+1)^{-1}$ for all $n\geq n''$. For the $n$ in between $n_m$ and $n_{m+1}$ we set $\beta_n=m^{-1}$ and it thus holds that $a_n(\beta_n)\leq \beta_n$ for all these $n$. The sequence $a_n(\beta_n)$ we construct in this way is monotonically decreasing and convergent to zero.} such that
\begin{equation}\label{conseqeuence_similar_to_Kurtz}
\lim_{n\to\infty}\,\EX^n\Bigl[\int_0^t\Lambda^n(U^n_s,\Theta^n_s)\int_{\sqrt{\alpha_n}\,|\langle\Phi, z^n(\xi)-z^n(\Theta_s^n)\rangle_\mathcal{E}|>\beta_n}\,\mu^n\bigl((U^n_s,\Theta^n_s),\dif\xi\bigr)\,\dif s\Bigr]\,=\, 0\,.
\end{equation}

We then obtain for (v) the estimate
\begin{eqnarray*}
\lefteqn{\int_0^t\EX^n\Bigl[\Lambda^n(U^n_s,\Theta^n_s)|h(Z^n_s)|\int_{K_n} \big|K\Bigl(\sqrt{\alpha_n}\langle z^n(\xi)-z^n(\Theta^n_s),\Phi\rangle_\cE\Bigr)\big|\,\mu^n\bigl((U^n_s,\Theta^n_s)\Bigr]\,\dif s}\\
&\leq& \nbar k \,\int_0^t\EX^n\Bigl[\Lambda^n(U^n_s,\Theta^n_s)\int_{\sqrt{\alpha_n}\,|\langle z^n(\xi)-z^n(\Theta_s^n),\Phi\rangle_\cE|>\beta_n}\alpha_n\,|\langle z^n(\xi)-z^n(\Theta_s^n),\Phi\rangle_\cE|^2\,\mu^n\bigl((U^n_s,\Theta^n_s)\Bigr]\,\dif s\\
&&\mbox{} + |k(\beta_n)|\,\alpha_n\int_0^t\EX^n\Bigl[\Lambda^n(U^n_s,\Theta^n_s)\int_{\sqrt{\alpha_n}\,|\langle z^n(\xi)-z^n(\Theta_s^n),\Phi\rangle_\cE|\leq \beta_n}|\langle z^n(\xi)-z^n(\Theta_s^n),\Phi\rangle_\cE|^2\,\mu^n\bigl((U^n_s,\Theta^n_s)\Bigr]\,\dif s,
\end{eqnarray*}
where $k(u):=K(u)/u^2$ is bounded by some constant $\nbar k$. The first term in the right hand side of this inequality converges to zero due to conditions \eqref{clt_martingale_cond_1} and \eqref{conseqeuence_similar_to_Kurtz}, as does the second term due to $|k(\beta_n)|=\landau(\beta_n)$ in combination with condition \eqref{ass_C_generator_diff_bound_1}. Therefore, the term (v) vanishes for $n\to\infty$.

%%
%%
%%  TERM 2 & TERM 3
%%
%%

\subsubsection{The terms (ii) and (iii)}

The terms (ii) and (iii) are given by 
\begin{equation*}
\int_0^t i \sqrt{\alpha_n}\,\EX^n\Bigl[\bigl\langle A(z^n(\Theta^n_s))\,U^n_s+B(U^n_s,z^n(\Theta^n_s))-A(p(s))\,u(s)-B(u(s),p(s)),\psi\bigr\rangle_\cX\,h(Z^n_s)\Bigr]\,\dif s,
\end{equation*}
which is further split into parts containing the operators $A$ and $B$, and 
\begin{equation*}
\int_0^t i\sqrt{\alpha_n}\,\EX^n\Bigl[\bigl\langle F(U^n_s,z^n(\Theta^n_s))-F(u(s),p(s)),\Phi\bigr\rangle_\cE\,h(Z^n_s)\Bigr]\,\dif s,
\end{equation*}
respectively. Let us start with the term (ii) containing and prove the convergence
\begin{eqnarray}\label{exp_result_B}
\lefteqn{\lim_{n\to\infty}\int_0^t i \sqrt{\alpha_n}\,\EX^n\Bigl[\bigl\langle B(U^n_s,z^n(\Theta^n_s))-B(u(s),p(s)),\psi\bigr\rangle_\cX\,h(Z^n_s)\Bigr]\,\dif s}\\[2ex] 
&&\phantom{xxxxxxxxx}=\ \int_0^t \Bigl\langle\frac{\partial \varphi}{\partial\psi},\ntilde B_u^\ast[u(s),p(s)]\psi\Bigr\rangle_\cX + \Bigl\langle \frac{\partial\varphi}{\partial\Phi},\ntilde B_p^\ast[u(s),p(s)]\psi\Bigr\rangle_\cE\,\dif s,\nonumber
\end{eqnarray}
where the partial derivatives of $\varphi$ are evaluated at $(s,\psi,\Phi)$. Note that the operators $\ntilde B_u^\ast[u(s),p(s)]:\cX\to\cX$ and $\ntilde B_p^\ast[u(s),p(s)]:\cX\to\cE$ are the adjoint operators to the extensions of the partial derivatives, see \assref{C}{ass_C_differentiability}. 

The derivation of the limit for the operators $A$ and $F$ work completely analogously and the same line of argument yields the convergence to a limit given by analogous formulae as for the term concerning $B$, i.e.,
\begin{eqnarray}\label{exp_result_A}
\lefteqn{\lim_{n\to\infty}\int_0^ti \sqrt{\alpha_n}\,\EX^n\Bigl[\bigl\langle A(z^n(\Theta^n_s))\,U^n_s-A(p(s))\,u(s),\psi\bigr\rangle_\cX\,h(Z^n_s)\Bigr]\,\dif s }\nonumber\\[2ex]
&&\phantom{xxxxxxxxxxxxxxxxxx}=\  \int_0^t \Bigl\langle\frac{\partial\varphi}{\partial\psi},\ntilde A^\ast(p(s))\psi\Bigr\rangle_\cX+\Bigl\langle\frac{\partial \varphi}{\partial\Phi},\ntilde A_p^\ast[u(s),p(s)]\,\psi\Bigr\rangle_\cE\,\dif s\phantom{xxxxxxx}
\end{eqnarray}
and 
\begin{eqnarray}\label{exp_result_F}
\lefteqn{\lim_{n\to\infty}\int_0^t i \sqrt{\alpha_n}\,\EX^n\Bigl[\bigl\langle F(U^n_s,z^n(\Theta^n_s))-F(u(s),p(s)),\Phi\bigr\rangle_\cE\,h(Z^n_s)\Bigr]\,\dif s}\\[2ex] 
&&\phantom{xxxxxxxxx}=\ \int_0^t \Bigl\langle\frac{\partial \varphi}{\partial\psi},\ntilde F_u^\ast[u(s),p(s)]\Phi\Bigr\rangle_\cX + \Bigl\langle \frac{\partial\varphi}{\partial\Phi},\ntilde F_p^\ast[u(s),p(s)]\Phi\Bigr\rangle_\cE\,\dif s,\nonumber
\end{eqnarray}
where the partial derivatives of $\varphi$ are evaluated at $(s,\psi,\Phi)$. Here the operators marked $\phantom{x}^\ast$ are the adjoints to the extensions $\ntilde A(p(s))\in L(\cX^\ast,\cX^\ast)$, $\ntilde A_p[u(s),p(s)]\in L(\cE^\ast,\cX^\ast)$, $\ntilde F_u[u(s),p(s)] \in L(\cX^\ast,\cE^\ast)$ and $\ntilde F_p[u(s),p(s)]\in L(\cE^\ast,\cE^\ast)$ of the partial derivatives.

\medskip

We now prove in detail \eqref{exp_result_B}. We first derive the pointwise limits for (almost) all $s\in [0,t]$ of the integrands in \eqref{exp_result_B} and then use the Dominated Convergence Theorem to infer the convergence of the integrals. Expanding the mapping $B(U^n_s,z^n(\Theta^n_s))$ around $(u(s),p(s))$  we obtain
\begin{eqnarray}\label{expansion_in_B}
\lefteqn{\langle B(U^n_s,z^n(\Theta^n_s))-B(u(s),p(s)),\psi\bigr\rangle_\cX \ =}\\[1ex]
&&\phantom{xxxxxx}\mbox{} + \langle B_u[u(s),p(s)](U^n_s-u(s)),\psi\rangle_\cX + \langle B_p[u(s),p(s)](z^n(\Theta^n_s)-p(s)),\psi\rangle_\cX +\eps_{B,n}\,.\nonumber
\end{eqnarray}

We thus obtain -- omitting the remainder term $\eps_{B,n} $-- the equalities
\begin{eqnarray*}
i\sqrt{\alpha_n}\bigl\langle B_u[u(s),p(s)](U^n_s-u(s)),\psi\bigr\rangle_\cX\, h(Z^n_s) &=& i\sqrt{\alpha_n}\bigl\langle \ntilde B_u^\ast[u(s),p(s)](U^n_s-u(s)),\psi\bigr\rangle_\cX\, h(Z^n_s)\\[2ex]
&=& i\sqrt{\alpha_n}\bigl\langle U^n_s-u(s),\ntilde B_u^\ast[u(s),p(s)]\psi\bigr\rangle_\cX\, h(Z^n_s)\\[2ex]
&=& \Bigl\langle\frac{\partial h(Z^n_s)}{\partial\psi},\ntilde B_u^\ast[u(s),p(s)]\psi\Bigr\rangle_\cX\in\mathbb{C} 
\end{eqnarray*}
and analogously
\begin{eqnarray*}
i\sqrt{\alpha_n}\bigl\langle B_p[u(s),p(s)](z^n(\Theta^n_s)-p(s)),\psi\bigr\rangle_\cX\, h(Z^n_s)
&=& \Bigl\langle\frac{\partial h(Z^n_s)}{\partial\Phi},\ntilde B_p^\ast[u(s),p(s)]\psi\Bigr\rangle_\cE\in\mathbb{C}. 
\end{eqnarray*} 
Here the operator $\ntilde B_u^\ast[u(s),p(s)]\in L(\cX,\cX)$ is the adjoint of $\ntilde B_u[u(s),p(s)]\in L(\cX^\ast,\cX^\ast)$ and $\ntilde B_p^\ast[u(s),p(s)]\in L(\cX,\cE)$ is the adjoint of $\ntilde B_p[u(s),p(s)]\in L(\cE^\ast,\cX^\ast)$. Then the Continuous Mapping Theorem and the continuity of the derivatives of $h$ imply that
\begin{equation*}
 \Bigl\langle\frac{\partial h(Z^n_s)}{\partial\psi},\ntilde B_u^\ast[u(s),p(s)]\psi\Bigr\rangle_\cX\quad \stackrel{d}{\longrightharpoonup}\quad \Bigl\langle\frac{\partial h(Z_s)}{\partial\psi},\ntilde B_u^\ast[u(s),p(s)]\psi\Bigr\rangle_\cX
\end{equation*}
and
\begin{equation*}
 \Bigl\langle\frac{\partial h(Z^n_s)}{\partial\Phi},\ntilde B_p^\ast[u(s),p(s)]\psi\Bigr\rangle_\cE \quad \stackrel{d}{\longrightharpoonup}\quad \Bigl\langle\frac{\partial h(Z_s)}{\partial\Phi},\ntilde B_p^\ast[u(s),p(s)]\psi\Bigr\rangle_\cE
\end{equation*}
where $ \stackrel{d}{\longrightharpoonup}$ denotes convergence in distribution. Furthermore, the boundedness of $h$, the uniform integrability of $Z^n$ due to Lemma \ref{lemma_uniform_integrable} and the continuity of the partial derivatives imply that the family of random variables given by the real and imaginary parts of the random variables
\begin{equation*}
\Bigl\langle\frac{\partial h(Z^n_s)}{\partial\psi},\ntilde B_u^\ast[u(s),p(s)]\psi\Bigr\rangle_X\,, \Bigl\langle\frac{\partial h(Z^n_s)}{\partial\Phi},\ntilde B_p^\ast[u(s),p(s)]\psi\Bigr\rangle_\cE\quad\forall\, s\leq t,\,n\in\mathbb{N}
\end{equation*}
are uniformly integrable. Thus \cite[Appendix, Prop.~2.3]{EthierKurtz} yields that
\begin{eqnarray*}
 \lim_{n\to\infty}\EX^n\Bigl\langle\frac{\partial h(Z^n_s)}{\partial\psi},\ntilde B_u^\ast[u(s),p(s)]\psi\Bigr\rangle_\cX &=& \EX\Bigl\langle\frac{\partial h(Z_s)}{\partial\psi},\ntilde B_u^\ast[u(s),p(s)]\psi\Bigr\rangle_\cX,\\[2ex]
\lim_{n\to\infty}\EX^n\Bigl\langle\frac{\partial h(Z^n_s)}{\partial\Phi},\ntilde B_p^\ast[u(s),p(s)]\psi\Bigr\rangle_\cE &=& \EX\Bigl\langle\frac{\partial h(Z_s)}{\partial\Phi},\ntilde B_p^\ast[u(s),p(s)]\psi\Bigr\rangle_\cE.
\end{eqnarray*}
Finally, the Dominated Convergence Theorem (see Appendix \ref{app_dom_conv_appl} for a detailed argument) allows to interchange taking the expectation and the Fr\'echet derivative due to \eqref{integrability_of_limit} and the boundedness of the exponential function along the imaginary axis. This yields
\begin{eqnarray*}
 \EX\Bigl\langle\frac{\partial h(Z_s)}{\partial\psi},\ntilde B_u^\ast[u(s),p(s)]\psi\Bigr\rangle_\cX\ = \ \Bigl\langle\EX\frac{\partial h(Z_s)}{\partial\psi},\ntilde B_u^\ast[u(s),p(s)]\psi\Bigr\rangle_\cX&=& \Bigl\langle\frac{\partial \varphi}{\partial\psi},\ntilde B_u^\ast[u(s),p(s)]\psi\Bigr\rangle_\cX,
\end{eqnarray*}
where the characteristic function of the limiting process $\varphi$ is evaluated at $(s,\psi,\Phi)$. This holds for (almost) all $s\in [0,t]$. Moreover, the integrands in the left hand side of \eqref{exp_result_B} are bounded in $s$ and $n$, thus dominated convergence yields that the integrals converge and \eqref{exp_result_B} holds. We have deferred the detailed estimates of this last step to Appendix \ref{app_dom_conv_appl}. 

\begin{remark}\label{remainders} Note that in parts (ii) and (iii) it was tacitly assumed that the remainder terms $\eps_{A,n}$, $\eps_{B,n}$ and $\eps_{F,n}$ in the above expansions vanish for $n\to\infty$. An argument why this holds under the Assumptions \ref{assumptions_CLT} is detailed in Appendix \ref{app_vanishing_remainder}. 
\end{remark}

\subsubsection{Combination of the estimates}

We now study $\varphi^n(t,\psi,\Phi)-\varphi^n(0,\psi,\Phi)$ which equals the sum of the previously discussed five terms. On the one hand, thanks to the convergence in distribution of $Z^n_s$ to a limit, the sequence $\varphi^n(t,\psi,\Phi)$ converges to $\varphi(t,\psi,\Phi)$ for all $t\geq 0$ and on the other hand the five terms converge to their limits respectively. Hence, we have obtained that the characteristic function $\varphi:\rr\times \cX\times \cE \to C: (t,\psi,\Phi)\to\varphi(t,\psi,\Phi)$ of the limit $Z$ satisfies the complex partial differential equation
\begin{eqnarray}\label{the_complex_equality}
\frac{\partial\varphi}{\partial t}&=& -\frac{1}{2}\bigl\langle G(u(s),p(s))\Phi,\Phi\bigr\rangle_\cE\,\varphi+ \Bigl\langle\frac{\partial\varphi}{\partial \psi}, \ntilde A^\ast(p(s))\,\psi+\ntilde B_u^\ast[u(s),p(s)]\,\psi+\ntilde F_u^\ast[u(s),p(s)]\,\Phi\Bigr\rangle_\cX\nonumber \\[2ex]
&&\mbox{} + \Bigl\langle\frac{\partial\varphi}{\partial\Phi},\ntilde A_p^\ast[u(s),p(s)]\,\psi+\ntilde B_p^\ast[u(s),p(s)]\,\psi+\ntilde F_p^\ast[u(s),p(s)]\,\Phi\Bigr\rangle_\cE
\end{eqnarray}
with initial condition $\varphi(0,\psi,\Phi)$, where $\varphi$ and its partial derivatives $\frac{\partial\varphi}{\partial t},\,\frac{\partial\varphi}{\partial\psi}$ and $\frac{\partial\varphi}{\partial\Phi}$ are evaluated at $(t,\psi,\Phi)$ and hence are elements of $\mathbb{C},\, L(\rr,\mathbb{C})=\mathbb{C}$, $L(\cX,\mathbb{C})$ and $L(\cE,\mathbb{C})$, respectively. But this equation is precisely \eqref{the_complex_equality_in_theorem}. Furthermore, as the equation \eqref{the_complex_equality} is linear in $\varphi$ uniqueness of solutions to given initial conditions follows from the uniqueness of solutions started at $\varphi(0,\psi,\Phi)=0$. However, the uniqueness of this solution is easily obtained by adapting the method of characteristics to partial differential equations in Hilbert spaces. We can thus uniquely characterise the one-dimensional distributions of the limit via the characteristic functions given by the unique solution to \eqref{the_complex_equality} started at the limit of the initial distributions. At the end of this section, see Section \ref{section_proof_conclusion}, we prove that \eqref{the_complex_equality} uniquely characterises all finite-dimensional distributions. Let us consider first as an intermediate step a particular family of solutions to \eqref{the_complex_equality}.

%%%%%%%%%%%%%%%%%%%%%%%%%%%%%%%%%%%%%%%%%%%
%%%%%%%%%%%%%%%%%%%%%%%%%%%%%%%%%%%%%%%%%%%
%%
%% Existence of a solution to the char eq.
%%
%%%%%%%%%%%%%%%%%%%%%%%%%%%%%%%%%%%%%%%%%%%
%%%%%%%%%%%%%%%%%%%%%%%%%%%%%%%%%%%%%%%%%%%

\subsection{Gaussian solution to the characteristic equation \eqref{the_complex_equality_in_theorem}}

In this part of the proof we show that the limit possesses one-dimensional Gaussian marginals if the initial condition is Gaussian. Thus we recall that if a process $Z$ is Gaussian, then its one-dimensional characteristic functions satisfy
\begin{equation}\label{Gaussian_process_characteristic}
\varphi(t,\psi,\Phi)=\exp\Bigl(i\,\langle m(t),(\psi,\Phi) \rangle_{\cX\times\cE}-\tfrac{1}{2}\langle C(t)(\psi,\Phi),(\psi,\Phi) \rangle_{\cX\times\cE}\Bigr)
\end{equation}
for families of $m(\cdot)\in \cX^\ast\times\cE^\ast$ and self-adjoint, non-negative trace class operators $C(t):\cX\times\cE\to\cX^\ast\times\cE^\ast$. That is, \eqref{Gaussian_process_characteristic} states that the one-dimensional distributions are Gaussian with mean $m(\cdot)$ and covariance $C(\cdot)$ \cite{DaPratoZabczyk}. 
We prove in the following that equation \eqref{the_complex_equality_in_theorem} possesses a solution of the form \eqref{Gaussian_process_characteristic}. Note that we assume the initial condition is Gaussian, hence there exists a non-negative, self-adjoint trace class operator $\mathcal{R}_0$ and an element $m_0$ such that \eqref{Gaussian_process_characteristic} holds for $t=0$ with $C(0)=\mathcal{R}_0$ and $m(0)=m_0$. From here on we consider the spaces $\cX,\,\cE$ to be real Hilbert spaces.

\medskip
In order to obtain a solution of the form \eqref{Gaussian_process_characteristic} for the equation \eqref{the_complex_equality_in_theorem} (or \eqref{the_complex_equality}) we consider the ansatz
\begin{eqnarray}\label{diffusion_char_funct_ansatz}
\varphi(t,\psi,\Phi)&=&\exp\Bigl\{i\,\langle m_1(t),\psi\rangle_\cX+i\,\langle m_2(t),\Phi\rangle_\cE\\[1ex]
&&\phantom{xxxxx}-\frac{1}{2}\Bigl(\langle R_{11}(t)\psi,\psi\rangle_\cX+\langle R_{12}(t)\Phi,\psi\rangle_\cX+\langle R_{21}(t)\psi,\Phi\rangle_\cE+\langle R_{22}(t)\Phi,\Phi\rangle_\cE\Bigr)\Bigr\} \nonumber
\end{eqnarray}
where $m_1(t)\in\cX^\ast$ and $m_2(t)\in\cE^\ast$ and the trace class operators $R_{11}(t)\in L_1(\cX,\cX^\ast)$, $R_{22}(t)\in L_1(\cE,\cE^\ast)$, $R_{12}(t)\in L_1(\cE,\cX^\ast)$ and $R_{21}(t)\in L_1(\cX,\cE^\ast)$ are Fr\'echet differentiable with respect to $t$. Note, that the adjoint of a trace class operator is again of trace class. 
% Springer online reference
Moreover, we can use these operators to define a trace class operator $\mathcal{R}\in L_1\bigl(\cX\times\cE,\cX^\ast\times\cE^\ast\bigr)$ such that $S(t):=\bigl\langle\mathcal{R}(t)\,(\psi_2,\Phi_2),(\psi_1,\Phi_1)\bigr\rangle_{\cX\times\cE}$ satisfies
\begin{eqnarray}
S(t)&=& \langle R_{11}(t)\psi_2,\psi_1\rangle_\cX+\langle R_{22}(t)\Phi_2,\Phi_1\rangle_\cE\nonumber\\[2ex]&&\qquad\qquad\mbox{}+\langle R_{12}(t)\Phi_2,\psi_1\rangle_\cX+\langle R_{21}(t)\psi_2,\Phi_1\rangle_\cE\phantom{xxxxxxxx}\label{first_definition_full_exponent}\\[2ex]
&=& \langle R_{11}^\ast(t)\psi_1,\psi_2\rangle_\cX+\langle R_{22}^\ast(t)\Phi_1,\Phi_2\rangle_\cE\nonumber\\[2ex]&&\qquad\qquad\mbox{}+\langle R_{12}^\ast(t)\psi_1,\Phi_2\rangle_\cE+\langle R_{21}^\ast(t)\Phi_1,\psi_2\rangle_\cX\,.\label{second_definition_full_exponent}\phantom{xxxxxxxx}
\end{eqnarray}
The operator $\mathcal{R}$ is self-adjoint if the operators $R_{11}$ and $R_{22}$ are and if $R_{12}=R^\ast_{21}$. However, these conditions are \emph{not} assumed a-priori. Nevertheless, using \eqref{first_definition_full_exponent} we find that the second part in the  exponent in \eqref{diffusion_char_funct_ansatz} equals $-\tfrac{1}{2}\bigl\langle\mathcal{R}(t)\,(\psi,\Phi),(\psi,\Phi)\bigr\rangle_{\cX\times\cE}$\,. The strategy is now to show that there exists a solution to \eqref{the_complex_equality_in_theorem} of the form \eqref{diffusion_char_funct_ansatz} and that $\mathcal{R}(t)$ is a suitable Gaussian covariance operator $C(t)$, i.e., $\mathcal{R}(t)$ is non-negative and self-adjoint.

\medskip

By the chain rule we obtain the partial derivatives of $\varphi$ with respect to $t, t,\psi$ and $\Phi$. Using the equality of the exponent in \eqref{diffusion_char_funct_ansatz} to \eqref{second_definition_full_exponent} we insert these partial derivatives into equation \eqref{the_complex_equality} and compare the coefficients. We obtain for the mean values

\begin{subequations}\label{op_valued_real_equations}
\begin{eqnarray}
\bigl\langle\frac{\partial m_1(t)}{\partial t},\psi\bigr\rangle_\cX&=& \bigl\langle m_1(t),\ntilde A^\ast(p(s))\,\psi+\ntilde B_u^\ast[u(s),p(s)]\,\psi\bigr\rangle_\cX\,,\nonumber\\[1ex]
&&\mbox{}+ \bigl\langle m_2(t),\ntilde A_p^\ast[u(s),p(s)]\,\psi+\ntilde B_p^\ast[u(s),p(s)]\,\psi\bigr\rangle_\cE\\[2ex]
\bigl\langle\frac{\partial m_2(t)}{\partial t},\Phi\bigr\rangle_\cX&=& \bigl\langle m_1(t),\ntilde F_u^\ast[u(s),p(s)]\,\Phi\Bigr\rangle_\cX+\bigl\langle m_2(t),\ntilde F_p^\ast[u(s),p(s)]\,\Phi\bigr\rangle_\cE
\end{eqnarray}
and for the covariance operators
\begin{eqnarray}
\bigl\langle\frac{\partial R_{11}(t)}{\partial t}\psi,\psi\bigr\rangle_\cX&=& \bigl\langle\bigl(R_{11}(t)+R^\ast_{11}(t)\bigr)\psi, \ntilde A^\ast(p(s))\,\psi+\ntilde B_u^\ast[u(s),p(s)]\,\psi\bigr\rangle_\cX\nonumber\\[2ex]
&&\mbox{}+\bigl\langle \bigl(R_{12}^\ast(t)+R_{21}(t)\bigr)\psi,\ntilde A_p^\ast[u(s),p(s)]\,\psi+\ntilde B_p^\ast[u(s),p(s)]\,\psi\bigr\rangle_\cE\\[2ex]
&=&\bigl\langle\frac{\partial R_{11}^\ast(t)}{\partial t}\psi,\psi\bigr\rangle_\cX\,,\nonumber
\end{eqnarray}
\begin{eqnarray}
\bigl\langle\frac{\partial R_{22}(t)}{\partial t}\Phi,\Phi\bigr\rangle_\mathcal{E}&=& \bigl\langle G(u(s),p(s))\Phi,\Phi\bigr\rangle_\cE + \bigl\langle \bigl(R_{12}(t)+R_{21}^\ast(t)\bigr)\Phi,\ntilde F_u^\ast[u(s),p(s)]\,\Phi\bigr\rangle_\cX\nonumber\\[2ex]
&&\mbox{}+ \bigl\langle\bigl(R_{22}(t)+R^\ast_{22}(t)\bigr)\Phi, \ntilde F_p^\ast[u(s),p(s)]\,\Phi\bigr\rangle_\mathcal{E}\\[2ex]
&=& \bigl\langle\frac{\partial R_{22}^\ast(t)}{\partial t}\Phi,\Phi\bigr\rangle_\mathcal{E},\nonumber
\end{eqnarray}
\begin{eqnarray}
\bigl\langle\frac{\partial R_{12}(t)}{\partial t}\Phi,\psi\bigr\rangle_\cX &=& \bigl\langle R^\ast_{11}(t)\psi, \ntilde F_u^\ast[u(s),p(s)]\,\Phi\bigr\rangle_\cX + \bigl\langle R_{12}(t)\Phi,
\ntilde A^\ast(p(s))\,\psi+\ntilde B_u^\ast[u(s),p(s)]\,\psi\bigr\rangle_\cX\nonumber\\[2ex]
&&\mbox{}+ \bigr\langle R_{22}(t)\Phi,\ntilde A_p^\ast[u(s),p(s)]\psi+\ntilde B_p^\ast[u(s),p(s)]\,\psi\bigr\rangle_\cE\phantom{xxxxxx}\\[2ex]
&&\mbox{}+\bigl\langle R_{12}^\ast(t)\psi,\ntilde F_p^\ast[u(s),p(s)]\,\Phi\bigr\rangle_\cE\nonumber\\[2ex]
&=&\bigl\langle\frac{\partial R_{21}^\ast(t)}{\partial t}\Phi,\psi\bigr\rangle_\cX \nonumber
\end{eqnarray}
and
\begin{eqnarray}
\bigl\langle\frac{\partial R_{21}(t)}{\partial t}\psi,\Phi\bigr\rangle_\cE&=& \bigl\langle R_{11}(t)\psi, \ntilde F_u^\ast[u(s),p(s)]\,\Phi\bigr\rangle_\cX\nonumber+ \bigl\langle R_{21}^\ast(t)\Phi,
\ntilde A^\ast(p(s))\,\psi+\ntilde B_u^\ast[u(s),p(s)]\,\psi\bigr\rangle_\cX\nonumber\\[2ex]
&&\mbox{}+ \bigr\langle R^\ast_{22}(t)\Phi,\ntilde A_p^\ast[u(s),p(s)]\psi+\ntilde B_p^\ast[u(s),p(s)]\,\psi\bigr\rangle_\cE\phantom{xxxxxx}\\[2ex]
&&\mbox{}+ \bigl\langle R_{21}(t)\psi,\ntilde F_p^\ast[u(s),p(s)]\,\Phi\bigr\rangle_\cE\nonumber\\[2ex]
&=&\bigl\langle\frac{\partial R_{12}^\ast(t)}{\partial t}\psi,\Phi\bigr\rangle_\cE\nonumber\,.
\end{eqnarray}
\end{subequations}
These real equations \eqref{op_valued_real_equations} are satisfied for all $\psi,\Phi$ if and only if the mean values satisfy the following system of Hilbert-valued differential equations
\begin{equation}\label{banach_ode_pt1}
\left.\begin{array}{rcl}
\dot m_1 &=& (\ntilde A+\ntilde B_u)\,m_1+(\ntilde A_p+\ntilde B_p)\,m_2,\\[2ex]
\dot m_2 &=& \ntilde F_u\,m_1+\ntilde F_p\,m_2
\end{array}\right.
\end{equation}
and the covariance operators satisfy the following system of operator-valued differential equations
\begin{equation}\label{banach_ode_pt2}
\left.\begin{array}{rcl}
\dot R_{11} &\!\!=&\!\! R_{11}\circ(\ntilde A^\ast+\ntilde B_u^\ast) +R_{12}\circ(\ntilde A_p^\ast+\ntilde B_p^\ast) + (\ntilde A+\ntilde B_u)\circ R_{11} +(\ntilde A_p+\ntilde B_p)\circ R_{21},\\[2ex]
\dot R_{22} &\!\!=&\!\! R_{22}\circ \ntilde F_p^\ast+R_{21}\circ \ntilde F_u^\ast + \ntilde F_p\circ R_{22} + \ntilde F_u\circ R_{12} + G,\\[2ex]
\dot R_{12} &\!\!=&\!\! R_{11}\circ\ntilde F_u^\ast+R_{12}\circ\ntilde F_p^\ast+(\ntilde A+\ntilde B_u)\circ R_{12} + (\ntilde A_p+\ntilde B_p)\circ R_{22},\\[2ex]
\dot R_{21} & \!\!=&\!\! R_{21}\circ (\ntilde A^\ast+\ntilde B^\ast_u) + R_{22}\circ (\ntilde A_p^\ast + \ntilde B_p^\ast) +\ntilde F_u\circ R_{11} + \ntilde F_p\circ R_{21}\,.
\end{array}\right.
\end{equation}
Note, that in the equations \eqref{banach_ode_pt1} and \eqref{banach_ode_pt2} we omitted the time-dependence of the derivatives in the right hand sides.
Thus the systems \eqref{banach_ode_pt1} and \eqref{banach_ode_pt2} decouple in two closed systems of differential equations, the first takes values in the space $\cX^\ast\times\cE^\ast$ and the second is operator-valued. 

We next show that these systems possess a unique global solution to any given initial condition as non-autonomous (the derivatives being evaluated along $(u(s),p(s))$ depend on time) differential equations in Banach spaces. First, this is straight forward for system \eqref{banach_ode_pt1}: by assumption the operators in its right hand side are linear and bounded. Hence, its right hand side is Lipschitz continuous. Moreover, the Lipschitz constant is bounded over $[0,T]$ for every $T>$0, since the derivatives of $\ntilde A,\, \ntilde B$ and $\ntilde F$ are continuous in $t$ due to assumption \assref{C}{ass_C_differentiability}. Thus the Picard-Lindel\"of Theorem for Banach-valued differential equations, see, e.g., \cite[Sec.~1.1]{Deimling}, implies existence of a unique global solution of \eqref{banach_ode_pt1} to any initial condition.

We proceed to system \eqref{banach_ode_pt2}. The composition of a trace class operator with a linear, bounded operator is again of trace class. Therefore we can interpret the composition of a linear mapping between spaces of trace class operators in the following sense. Let $Y_i$, $i=1,2,3$, be Hilbert spaces, $O_1\in L_1(Y_1,Y_2)$ be a trace class operator and $O_2\in L(Y_2,Y_3)$ a linear, bounded operator. Then $O_1\circ O_2\in L_1(Y_1,Y_3)$ is a trace class operator and $O_2\circ: L_1(Y_1,Y_2)\to L_1(Y_1,Y_3): O\mapsto O_2\circ O$ is a linear map between two Banach spaces. Moreover, this map is bounded:
\begin{eqnarray*}
\big\|O_2\circ\big\|_{L\bigl(L_1(Y_1,Y_2),L_1(Y_2,Y_3)\bigr)} &=& \sup_{\|O\|_{L_1(Y_1,Y_2)}\leq1} \|O_2\circ O\|_{L_1(Y_1,Y_3)} \\[2ex]
&\leq& \sup_{\|O\|_{L_1(Y_1,Y_2)}\leq1} \|O_2\|_{L(Y_2,Y_3)} \|O\|_{L_1(Y_1,Y_2)}\ = \ \|O_2\|_{L(Y_1,Y_2)}\,.
\end{eqnarray*}
Analogously, we obtain that for $O_1\in L_1(Y_2,Y_3)$ and $O_2\in L(Y_1,Y_2)$,  $O_1\mapsto O_1\circ O_2$ defines a linear, bounded mapping from $L_1(Y_2,Y_3)$ into $L_1(Y_1,Y_3)$. 
Therefore the set of operator-valued differential equations \eqref{banach_ode_pt2} is a linear, non-autonomous and inhomogeneous system (because of the term $G(u(s),p(s))$) in a vector Banach space of trace class operators. Thus it can also be analysed using the theory of ordinary differential equations in Banach spaces. 
The right hand side of the system \eqref{banach_ode_pt1}, \eqref{banach_ode_pt2} is linear, hence it is Lipschitz continuous with a uniform Lipschitz constant on any bounded interval $[0,T]$ due to assumption \assref{C}{ass_C_differentiability}. As before the Picard-Lindel\"of Theorem yields existence of a unique global solution of \eqref{banach_ode_pt2} to any initial condition. \medskip

Finally, in order for the trace class operators $\mathcal{R}(t)$, $t\geq 0$, --  the solution of the system \eqref{banach_ode_pt2} with initial condition $\mathcal{R}(0)=\mathcal{R}_0$ -- to define covariance operators of a Gaussian random variables we need additionally that they are (i) self-adjoint and (ii) that the map $(\psi,\Phi)\mapsto\bigl\langle\mathcal{R}(t)\,(\psi,\Phi),(\psi,\Phi)\bigr\rangle_{\cX\times\cE}$ is non-negative for all $t\geq 0$. By assumption the initial condition of the limiting process is Gaussian. Hence there exists a mean value $m_0\in\cX^\ast\times\cE^\ast$ and a non-negative, self-adjoint trace class operator $\mathcal{R}_0$ as its variance. Furthermore $\mathcal{R}_0$ is of the form \eqref{first_definition_full_exponent} for self-adjoint trace class operators $R_{11}^0\in L_1(\cX,\cX^\ast)$ and $R_{22}^0\in L_1(\cE,\cE^\ast)$ and operators $R_{12}(0)\in L_1(\cE,\cX^\ast)\in $ and $R_{21}(0)\in L_1(\cX,\cE^\ast)$ such that $R_{12}^\ast(0)=R_{21}(0)$. The mean values and covariance operators of the one-dimensional marginals are thus given by the solution to \eqref{banach_ode_pt1} and \eqref{banach_ode_pt2}, respectively, started at these initial conditions.
Using the adjoints in the ansatz \eqref{diffusion_char_funct_ansatz} we obtain the equations
\begin{equation}\label{banach_ode_pt2_b}
\left.\begin{array}{rcl}
\dot R_{11}^\ast &\!\!=&\!\! R_{11}\circ(\ntilde A^\ast+\ntilde B_u^\ast) +R_{12}\circ(\ntilde A_p^\ast+\ntilde B_p^\ast) + (\ntilde A+\ntilde B_u)\circ R_{11} +(\ntilde A_p+\ntilde B_p)\circ R_{21},\\[2ex]
\dot R_{22}^\ast &\!\!=&\!\! R_{22}\circ \ntilde F_p^\ast+R_{21}\circ \ntilde F_u^\ast + \ntilde F_p\circ R_{22} + \ntilde F_u\circ R_{12} + G,\\[2ex]
\dot R_{21}^\ast &\!\!=&\!\! R_{11}\circ\ntilde F_u^\ast+R_{12}\circ\ntilde F_p^\ast+(\ntilde A+\ntilde B_u)\circ R_{12} + (\ntilde A_p+\ntilde B_p)\circ R_{22},\\[2ex]
\dot R_{12}^\ast & \!\!=&\!\! R_{21}\circ (\ntilde A^\ast+\ntilde B^\ast_u) + R_{22}\circ (\ntilde A_p^\ast + \ntilde B_p^\ast) +\ntilde F_u\circ R_{11} + \ntilde F_p\circ R_{21}\,.
\end{array}\right.
\end{equation}
Comparing the right hand sides of these equations and of their respective adjoint operators in \eqref{banach_ode_pt2} we find that they coincide. Thus we infer that $R_{11}(t)$ and $R_{22}(t)$ are self-adjoint and $R_{12}(t)=R^\ast_{21}(t)$ for any time $t\geq 0$ as this holds for the initial conditions.
Thus it remains to establish the non-negativity of the operator $\mathcal{R}(t)$. But the map $\exp\{-\tfrac{1}{2}\bigl\langle\mathcal{R}(t)\,(\psi,\Phi),(\psi,\Phi)\bigr\rangle_{\cX\times\cE}\}$ is the unique solution to \eqref{the_complex_equality} with mean value function $m(t)\equiv 0$ and initial covariance $\mathcal{R}_0$. Therefore it is necessarily the characteristic function of a random variable and consequently $\exp\{-\tfrac{1}{2}\bigl\langle\mathcal{R}(t)\,(\psi,\Phi),(\psi,\Phi)\bigr\rangle_{\cX\times\cE}\}\leq 1$. From this we infer that $\mathcal{R}(t)$ is positive.

\subsubsection{Conclusion of the proof}\label{section_proof_conclusion}

Clearly, for a stochastic process to be Gaussian it is not sufficient that its one-dimensional distributions, i.e., \eqref{Gaussian_process_characteristic}, are Gaussian but all finite-dimensional distributions have to be jointly Gaussian. Analogously, in order to uniquely define an arbitrary limiting process we have to characterise all its finite-dimensional distributions. 
Thus, we now complete the proof of Theorem \ref{fluctuation_limit_theorem} showing that all the finite-dimensional distributions are characterised via the solution of \eqref{the_complex_equality}. We detail the step obtaining the two-dimensional from the one-dimensional characteristic function. The general finite-dimensional case follows by induction.

Due to the Markov property of the PDMPs $(U^n,\Theta^n,u,p)$ or the processes $(Z^n,u,p)$, respectively, the characteristic function of the two-dimensional distribution $\varphi^n_2$ of $Z^n$ at fixed times $t_1<t_2$ satisfies
\begin{eqnarray*}
\varphi^n_2(t_1,t_2,\psi_1,\psi_2,\Phi_1,\Phi_2) &=& \EX^n\Bigl[h(Z^n_{t_1},\psi_1,\Phi_1)h(Z^n_{t_2},\psi_2,\Phi_2)\Bigr]\\
&=& \EX^n\Bigl[h(Z^n_{t_1},\psi_1,\Phi_1)\,\EX^n_{Z_{t_1}^n,u(t_1),p(t_1)}\bigl[h(Z^n_{t_2-t_1},\psi_2,\Phi_2)\bigr]\Bigr]\\
&:=& \EX^n\Bigl[h(Z^n_{t_1},\psi_1,\Phi_1)\,g^n({Z_{t_1}^n},u(t_1),p(t_1),t_2-t_1,\psi_2,\Phi_2)\Bigr]\,,
\end{eqnarray*}
where $\EX^n_{\bullet}$ denotes the expectation with respect to the law of a PDMP $Z^n$ started almost surely in the point as indicated by the index. Clearly, this starting point has to be an attainable value for the random variable $Z^n_{t_1}$. These are not necessarily all elements of $\cX^\ast\times\cE^\ast$ and we denote this set of values by $\mathcal{K}_n$. 
We note that the function $g^n$ thus defined is bounded and measurable. Next, we denote by $g$ the characteristic function of $Z_{t_2-t_1}$ started almost surely at $z$ and thus given by the solution to \eqref{the_complex_equality} at time $t_2-t_1$ with respect to a degenerate Gaussian initial condition centred at $z$. 
The difficulty for taking the limit is that $g^n$ is usually not defined on the same set as $g$ but only on a set $\mathcal{K}_n\subseteq\sqrt{\alpha}(H\times z^n(K^n))$ possessing full measure, i.e., $\Pr^n[Z^n_{t_1}\in\mathcal{K}_n]=1$. We thus have to first extend $g^n$ to a function $\ntilde g^n$ on all of $X^\ast\times \cE^\ast$ via
\begin{equation*}
\ntilde g_n(z):=\left\{\begin{array}{cl} g_n(z)& \tn{if } z\in\mathcal{K}_n,\\[1ex]
g(z) & \tn{if } z\notin\mathcal{K}_n\,.
\end{array}\right.
\end{equation*}
With this extension it holds
\begin{equation*}
\varphi^n_2(t_1,t_2,\psi_1,\psi_2,\Phi_1,\Phi_2)=\EX^n\Bigl[h(Z^n_{t_1},\psi_1,\Phi_1)\,\ntilde g_n({Z_{t_1}^n},u(t_1),p(t_1),t_2-t_1,\psi_2,\Phi_2)\Bigr]\,.
\end{equation*}
Then we define the set
\begin{equation*}
\mathcal{S}:=\{z\in\cX^\ast\times\cE^\ast: \ntilde g_n(z_n)\to g(z) \ \forall\,z_n\to z\}\,.
\end{equation*}
Due to the foregoing convergence result and the continuity of $g$ we immediately find that $\Pr[Z_{t_1}\in \mathcal{S}]=1$. Here we have essentially two cases for $z^n\to z$. If all but finitely many elements of the sequence are such that $z^n\in\mathcal{K}_n$ then $g_n(z_n)\to g(z)$ holds due to the foregoing convergence result. On the other hand, if all but finitely many elements are such that $z^n\notin\mathcal{K}_n$ then $\ntilde g_n(z^n)=g(z^n)\to g(z)$ holds due to the continuity of $g$. All other possibilities are mixtures of such sequences and the convergence $\ntilde g_n(z^n)\to g(z)$ thus also holds.

Therefore, \cite[Thm.~5.5]{Billingsley} establishes that $\ntilde g_n(Z^n_{t_1})$ converges weakly to $g(Z_{t_1})$ which yields for the two-dimensional characteristic function the convergence 
\begin{eqnarray*}
\varphi_2(t_1,t_2,\psi_1,\psi_2,\Phi_1,\Phi_2)&=& \lim_{n\to\infty} \EX^n\bigl[h(Z^n_{t_1},\psi_1,\Phi_1)\,\ntilde g_n({Z_{t_1}^n},u(t_1),p(t_1),t_2-t_1,\psi_2,\Phi_2)\bigr]\\
&=&\EX\bigl[h(Z_{t_1},\psi_1,\Phi_1)\,g(Z_{t_1},u(t_1),p(t_1),t_2-t_1,\psi_2,\Phi_2)\bigr]\,.
\end{eqnarray*}
Here the left hand side $\varphi_2$ is the two-dimensional characteristic function of the limiting process $Z$. We conclude the proof showing that $\varphi_2$ is uniquely defined from \eqref{the_complex_equality}.
We have that $g(z,t_1,t_2-t_1)$ -- omitting the remaining arguments -- is the characteristic function given by the solution of \eqref{the_complex_equality} started at the characteristic function of the degenerate Gaussian law centered at $z$, i.e., the initial condition to \eqref{the_complex_equality} is given by $\exp( i\langle (\psi_2,\Phi_2),z\rangle)$, at time $t_1$ in $z$ after time $t_2-t_1$. Thus $g(z,t_1,t_2-t_1)$ is of the form \eqref{Gaussian_process_characteristic} where the mean value function at time $t_2-t_1$ is given by the solution to \eqref{banach_ode_pt1} started in $t_1$ at $z$. Due to the general existence theorem of linear equations in Banach spaces there exists a linear, bounded operator $M(t_1,t_2)$ such that this mean function is given by $M(t_1,t_2) z$, i.e., the first part of the exponent is $\langle M(t_1,t_2) Z_{t_1},(\psi_2,\Phi_2)\rangle_{\cX\times\cE}=\langle Z_{t_1},M(t_1,t_2)^\ast(\psi_2,\Phi_2)\rangle_{\cX\times\cE}$. Secondly, the covariance operator $\mathcal{R}$ at time $t_2-t_1$ is given by the solution to \eqref{banach_ode_pt2} started in the origin at $t_1$ which we denote by $\mathcal{R}_0(t_1,t_2)$. Note that it is independent of $z$. Therefore we obtain
\begin{eqnarray}\label{final_form_twodim_char_dist}
\varphi_2(t_1,t_2,\psi_1,\psi_2,\Phi_1,\Phi_2)&=&\EX\Bigl[\e^{i\langle Z_{t_1},(\psi_1,\Phi_1)+M^\ast(t_1,t_2)(\psi_2,\Phi_2)\rangle_{\cX\times\cE}}\Bigr] \\
&&\phantom{xxxxxxxxxxxxxx}\times\,\exp\Bigl\{-\frac{1}{2}\langle \mathcal{R}_0(t_1,t_2)(\psi_2,\Phi_2),(\psi_2,\Phi_2)\rangle_{\cX\times\cE}\Bigr\},\nonumber
\end{eqnarray}
where the remaining expectation in the right hand side is a one-dimensional characteristic function uniquely given by \eqref{the_complex_equality}. Hence, all two-dimensional characteristic functions are uniquely defined. The analogous result for all finite-dimensional equations follows by induction. We furthermore infer from \eqref{final_form_twodim_char_dist} that the limiting process is a Gaussian process, if the limit of the initial conditions is Gaussian.

\subsection{Proof of Corollary \ref{corollary_spde_representation}}

If the initial conditions of the PDMPs admit a Gaussian limit then the limiting process is under Theorem \ref{fluctuation_limit_theorem} a Gaussian process as is the mild solution to an stochastic evolution equation of the form \eqref{general_spde} started in a Gaussian initial condition (cf. \cite{DaPratoZabczyk}). Hence, it suffices to show that the finite-dimensional characteristic functions coincide. In the following we first show that the one-dimensional marginals coincide, i.e., we show that the mean and covariance of the solution of the SPDE satisfy \eqref{banach_ode_pt1} and \eqref{banach_ode_pt2}, respectively, and afterwards extend the result by induction to all finite-dimensional distributions. For future use we repeat these equations which we write in simplified matrix notation, i.e., for the equation of the mean values \eqref{banach_ode_pt1} we write
\begin{equation}\label{banach_ode_pt1_repeat}
\left(\!\!\begin{array}{c} \dot m_1 \\ \dot m_2 \end{array}\right)= \left(\!\!\begin{array}{cc} \ntilde A+\ntilde B_u & \ntilde A_p+\ntilde B_p\\ \ntilde F_u & \ntilde F_p\end{array}\right)\left(\!\!\begin{array}{c} m_1 \\ m_2 \end{array}\right),
\end{equation}
where the matrix-vector multiplication is understood as applications of operators, and similarly for the equation for the covariance operator \eqref{banach_ode_pt2}
\begin{eqnarray}\label{banach_ode_pt2_repeat}
\lefteqn{\left(\!\!\begin{array}{cc} \dot R_{11} & \dot R_{12} \\ \dot R_{21} & \dot R_{22} \end{array}\right)}\\
&=& \left(\!\!\begin{array}{cc} R_{11} & R_{12} \\ R_{21} & R_{22}\end{array}\right)\left(\!\!\begin{array}{cc} \ntilde A^\ast+\ntilde B^\ast_u & \ntilde F_u^\ast \\ \ntilde A^\ast_p+\ntilde B^\ast_p & \ntilde F_p^\ast\end{array}\right)+\left(\!\!\begin{array}{cc} \ntilde A+\ntilde B_u & \ntilde A_p+\ntilde B_p\\ \ntilde F_u & \ntilde F_p\end{array}\right)\left(\!\!\begin{array}{cc} R_{11} & R_{12} \\ R_{21} & R_{22}\end{array}\right)+\left(\!\!\begin{array}{cc} 0 & 0 \\ 0 & G\end{array}\right),\nonumber
\end{eqnarray}
where the matrix-matrix multiplication is understood as composition of operators. We denote by $S(s,t)$, $s\leq t$, the two-parameter semigroup of linear operators on $\cX^\ast\times\cE^\ast$ defined by the solution of the differential equation for the mean, i.e., $S(s,t) m$ is the solution to \eqref{banach_ode_pt1_repeat} at time $t$ started in $m$ in time $s$.
Then the unique mild solution to \eqref{general_spde} is given by
\begin{equation}\label{general_mild_solution}
\left(\!\!\begin{array}{c} U_t\\[2ex] P_t\end{array}\!\!\right) = S(0,t)\left(\!\!\begin{array}{c} U_0\\[2ex] P_0\end{array}\!\!\right) + \int_0^t S(s,t)\left(\!\!\begin{array}{c}  0\\[2ex]  g(s)\end{array}\!\!\right)\,\dif  W^Q_s\,, 
\end{equation}
where $W^Q$ is a Wiener process on a Hilbert space $H'$ with covariance $Q\in L(H')$, $g(t)\in L(H',\cE^\ast)$ and $0$ denotes the origin in $L(H',\cX^\ast)$.
Thus taking expectations on both sides of \eqref{general_mild_solution} and differentiating with respect to $t$ we immediately obtain that the mean value function of the mild solution \eqref{general_mild_solution} satisfies \eqref{banach_ode_pt1_repeat} with initial condition given by the expectation of $(U_0,P_0)$.

\medskip

It remains to consider the covariance operator. We denote by $Q_0$ the covariance operator of the initial condition and by $Q_Y(t)$ the covariance of the stochastic convolution 
\begin{equation*}
Y_t=\int_0^t S(s,t)(0,g(s))^T\,\dif W^Q_s
\end{equation*}
which is given by (cf. \cite{DaPratoZabczyk})
\begin{equation*}
Q_Y(t) = \int_0^t \Bigl(S(t,s)(0,g(s))^T\Bigr)Q\Bigl(S(t,s)(0,g(s))^T\Bigr)^\ast\,\dif s\,\in L(\cX^\ast\times\cE^\ast,\cX^\ast\times\cE^\ast)\,.
\end{equation*} 

Thus the covariance of $Z_t$ is given by $Q_Z(t)=S(0,t)Q_0S^\ast(0,t)+Q_Y(t)$. Here all the adjoints are understood as Hilbert space adjoints. Then $Q_Z(t)$ satisfies the initial condition $Q_Z(0)=Q_0$, and differentiating with respect to $t$ yields
\begin{equation}
\frac{\dif Q_Z(t)}{\dif t}
=\left(\!\!\begin{array}{cc} \ntilde A+\ntilde B_u & \ntilde A_p+\ntilde B_p\\ \ntilde F_u & \ntilde F_p\end{array}\right)Q_Z(t)+Q_Z(t)\left(\!\!\begin{array}{cc} \ntilde A^\ast+\ntilde B^\ast_u & \ntilde F_u^\ast \\ \ntilde A^\ast_p+\ntilde B^\ast_p & \ntilde F_p^\ast\end{array}\right)+\left(\!\!\begin{array}{cc} 0 & 0 \\ 0 & G\end{array}\right).
\end{equation}
Thus $Q_Z$ satisfies the differential equation for the covariance of the limit \eqref{banach_ode_pt2_repeat}. Thus the mild solution of the SPDE possesses the same one-dimensional marginals as the limiting process characterised by \eqref{the_complex_equality_in_theorem}. The Markov property of the SPDE solution yields that the two-dimensional characteristic functions are equal to \eqref{final_form_twodim_char_dist}. The general statement that the finite-dimensional characteristic functions agree with the finite-dimensional characteristic functions of the limit characterised by \eqref{the_complex_equality_in_theorem} follows by induction. The proof is complete.
%
% \marginpar{some details on paper}.

\section{Applications to Electrophysiology}\label{section_applications}

In this section we apply the main result, Theorem \ref{fluctuation_limit_theorem}, to obtain limits for the global fluctuations of stochastic microscopic models around their deterministic limit. These limits are classical deterministic spatio-temporal models in neuroscience. The first one is a simple version of the Hodgkin-Huxley model (cf.~\cite{Koch}) for an excitable membrane reduced to its essentials, which is used not only to model neuronal membranes but also intracellular tissues in calcium dynamics or cardiac tissues. Secondly, we consider a neural field model for the large scale activity in macroscopic brain areas. The corresponding deterministic limit in this case is the Wilson-Cowan equation \cite{Bressloff_review}. Internal fluctuations for such models have been considered in previous work applying the law of large numbers (cf.~Theorem \ref{theorem_lln}) and the martingale limit theorem \cite{RTW,BR_neural_fields}. For these models the spatial domain is a bounded domain $D\subset \rr^d$ with sufficiently regular boundary. The Hilbert spaces that thus occur as state spaces of the PDMPs are spaces of real functions on $D$, for instance, the Lebesgue space $L^2(D)$ or Sobolev Spaces $H^\alpha(D)$. Subsequently, with regards to brevity and simplicity of notation we simply write $L^2$ or $H^\alpha$ for these spaces, omitting the domain. Furthermore, products of elements of these spaces are always understood as pointwise products of real functions.

\begin{remark}[Spatial discretisation] In both applications we will study, we start with discretising the spatial domain $D$. For each $n\in\mathbb{N}$ we denote by $\pindex^n$ a decomposition of $D$ into $p(n)<\infty$ non-overlapping subdomains $D^{k,n}$, $k=1,\ldots,p(n)$. For technical reasons we assume that the subdomains $D^{k,n}$ are convex \cite{BR_neural_fields,RiedlerPhD,RTW}. Then, $\nu_{\pm}(n)$ and $\delta_{\pm}(n)$ denote the maximum / minimum Lebesgue measure and the maximum / minimum diameter of the subdomains in partition $\pindex^n$, respectively.  We refer to our previous work for a discussion of the existence of such partitions satisfying all the subsequent assumptions for a large class of domains containing all practically relevant geometries.
\end{remark}

\begin{remark}[Hilbert scales]\label{Hilbert_scales} Another important concept are Hilbert scales. A Hilbert scale is a family of Hilbert spaces $H_\alpha$ indexed by $\alpha\in\rr_+$ which satisfy $H_{\alpha_1}\hookrightarrow H_{\alpha_2}$ for $\alpha_1>\alpha_2$. This extends to Hilbert scale indexed on $\alpha\in\rr$ if we set $H_{-\alpha}$ the dual space of $H_\alpha$ with the following identification of inner product and duality pairing
\begin{equation}
\langle\phi,x\rangle_{H_{-\alpha}}=(\phi,x)_{H_0}\quad\forall\,\phi\in H_0,\,x\in H_\alpha\,. 
\end{equation}
As we observe from this identification the Hilbert space $H_0$ plays a distinguished role as each other space $H_\alpha$, $\alpha>0$, forms an evolution triplet with $H_0$. Examples of Hilbert scales are the standard Sobolev spaces $H^\alpha$, $\alpha\in\mathbb{N}$, on a bounded domain $D$ or Sobolev spaces $H^\alpha_0$ of functions vanishing at the boundary with the usual interpolation spaces for non-integer indices \cite{Adams}. Here the distinguished Hilbert space is given by $H^0=L^2$. A second example are Hilbert scales obtained by closing of operators, see, e.g., \cite{Blount1,Kotelenez1,Kotelenez2} in applications to limit theorems for spatio-temporal reaction-diffusion models. 

\noindent As an example consider the one-dimensional negative Laplacian $-\Delta$ with Dirichlet boundary conditions $D=[0,l]$, and the spaces defined by
\begin{equation}\label{definition_of_laplace_scale}
H_\alpha:=\Bigl\{g\in L^2:\,\sum_{i=1}^\infty (g,\varphi_i)_{L^2}^2(1+i^2)^\alpha\Bigr\}\,,
\end{equation}
where $\varphi_i(x)=\sqrt{2/l}\,\sin(\pi i x/l)$, $i\geq 1$, form a complete orthonormal basis in $L^2$ consisting of eigenvectors to the Laplacian. The inner product on this space is given by
\begin{equation*}
(g_1,g_2)_{H_\alpha}=\sum_{i=1}^\infty(g_1,\varphi_i)_{L^2}(g_2,\varphi_i)_{L^2}(1+\pi^2i^2/l^2)^\alpha\,.
\end{equation*}
Here, as before, the distinguished space satisfies $H_0=L^2$. As the Sobolev spaces above, the spaces \eqref{definition_of_laplace_scale} create a Hilbert scale. Moreover, these spaces are subspaces of the respective Sobolev spaces $H^\alpha$, where the norm induced by $(\cdot,\cdot)_{H_{-\alpha}}$ is equivalent to the restriction of the usual Sobolev norm \cite{Kotelenez2}. Thus in terms of embeddings and forming of evolution triplets this scale behaves as the standard Sobolev spaces. Another important result is the following \cite[Lemma 2.2]{Kotelenez2}: the Laplace operator $\Delta: H^{1}\to H^{-1}$ possesses an extension $\ntilde\Delta_{-\alpha}$ to an operator acting on $H_{-\alpha}$ for every $\alpha\geq 0$.
\end{remark}

Finally, we note that in the following we always use the notation $H_\alpha$ to denote the spaces defined  in \eqref{definition_of_laplace_scale}. \medskip

\subsection{Compartmental-type neuron models}

Models of excitable membranes describe the evolution of the transmembrane potential, that is the difference $V_{{\rm ext}}-V_{{\rm int}}$ of electrical potential between the outside and the inside of the cell. This evolution results from a complex electrochemical coupling between the membrane potential and the various ion channels located at some points of the membrane. The repartition of these channels is heterogeneous. In particular for neuronal membranes it is known that channels are present only at the Nodes of Ranvier. We first present a simple version of general excitable media models. However, the results extend immediately to more general versions (cf. \cite{RiedlerPhD,RTW}). For the spatial domain we restrict ourselves to a finite interval $D=[0,l]\subset\rr $. 
For this model the deterministic limit is given by the excitable membrane equation
\begin{equation}\label{excitable_membrane_system}
\left.\begin{array}{rcl}
 \dot u &=& \Delta u + p\cdot (\nbar v-u)\\[2ex]
\dot p &=& a(u)\cdot (1-p)-b(u)\cdot p
\end{array}\right.
\end{equation}
on $[0,l]$, with Dirichlet boundary conditions to conform with the exposition in Remark \ref{Hilbert_scales}. In \eqref{excitable_membrane_system} $\Delta$ denotes the Laplace operator, $\nbar v$ is some constant and $a,b:\rr\to\rr_+$ are non-negative smooth functions. The model \eqref{excitable_membrane_system} is stripped of any complicating constants and corresponds to an excitable membrane with currents due to a single family of ion channels which can be in two states, i.e., they are either open or closed. It was shown that systems of the form \eqref{excitable_membrane_system} possess unique global solutions which are componentwise in $C(\rr_+,H_2)$ and $C(\rr_+,H^2)$, respectively, for all initial condition $u_0\in H_2$ and $p_0\in H^2$ \cite{RiedlerPhD}.

\medskip

In the corresponding stochastic PDMP model, we get the family of abstract evolution equations
\begin{equation}
\dot u \,=\, \Delta u + z^n(\theta^n)\cdot(\nbar v-u) 
\end{equation}
equipped with the same boundary conditions as the deterministic model \eqref{excitable_membrane_system}. This yields that we choose for $X\subset H\subset X^\ast$ a standard evolution triplet in order to deal with reaction-diffusion equations, i.e.,
\begin{equation*}
H_1\subset L^2\subset H_{-1}\,.
\end{equation*}
Note that Maurin's Theorem implies that the embedding $H_1\hookrightarrow L^2$ is of Hilbert-Schmidt type as the spatial dimension is one. The piecewise constant component $\theta^n=(\theta^{1,n},\ldots,\theta^{p(n),n})\in K_n\subseteq\mathbb{N}^{p(n)}$ counts the number of open ion channels in the subdomains $D^{k,n}$. The Central Limit Theorem that we proved in the previous section enables us to specify the size of the intrinsic noise induced by the stochastic gating of the channels using parameters of the domain decomposition. One interest of our result is that we can, up to some extent, consider inhomogeneous channel repartition. More details are given in Theorem \ref{globalCLTcompartmental} below.

\bigskip 

\noindent Given a partition $\pindex^n$ of subdomains $D^{k,n}$, let us denote by $l(k,n)<\infty$ the total number of channels in domain $D^{k,n}$. We denote by $\ell_{\pm}(n)$ the maximum / minimum number of channels in the individual subdomains of $\pindex^n$. The transition dynamics of the piecewise constant components $\theta^n$ are given by the $(u,\theta^n)$-dependent rates
\begin{equation*}
\theta^n\to\nhat\theta^n \tn{ with rate } 
\left\{\begin{array}{cl}
b\displaystyle\bigl(\nbar u^{k,n}\bigr)\, \theta^{k,n}& \tn{if } \nhat\theta^{k,n}=\theta^{k,n}-1,\, \nhat\theta^{j,n}=\theta^{j,n},\, j\neq k,\\[2.5ex]
a\displaystyle\bigl(\nbar u^{k,n}\bigr)\, \bigl(l(k,n)-\theta^{k,n}\bigr)& \tn{if } \nhat\theta^{k,n}=\theta^{k,n}+1,\, \nhat\theta^{j,n}=\theta^{j,n},\, j\neq k,\\[2.5ex]
0 & \tn{otherwise}\,,
\end{array}\right.
\end{equation*}
where $\nbar u^{k,n}=|D^{k,n}|^{-1}\int_{D^{k,n}} u(x)\,\dif x$ is the average membrane potential over the $k$-th subdomain of $\pindex^n$. We see that the components of $\theta^n$ satisfy $\theta^{k,n}\in\{0,\ldots, l(k,n)\}$ and hence $K^n$ is finite. Finally, we repeat that the coordinate functions $z^n$ are given by \eqref{def_neuron_coordinate}, i.e.,
\begin{equation*}
z^n(\theta^n)=\sum_{k=1}^{p(n)} \frac{\theta^{k,n}}{l(k,n)}\,\mathbb{I}_{D^{k,n}} \in L^2\,
\end{equation*}
and thus $E=L^2$. It remains to fix the Hilbert spaces $\cE$ and $V$: we chose $V=H^\alpha$ and $\mathcal{E}=H^{2\alpha}$ for $\alpha\in(1/2,1]$ and Maurin's Theorem implies that the embedding $\cE \hookrightarrow V$ is of Hilbert-Schmidt type.

\medskip

If the initial conditions converge in probability and $\ell(n)_-\to\infty$, $\delta(n)_+\to 0$ the sequence of stochastic models $(U^n,z^n(\Theta^n))$ converges u.c.p.~to the solution $(u,p)$ of \eqref{excitable_membrane_system} in $L^2\times L^2$ \cite{RTW}. We now present the central limit theorem for the global fluctuations.

\begin{theorem}\label{globalCLTcompartmental} Let us assume that the rescaled difference in the macroscopic initial conditions converges to a Gaussian distribution $\mathcal{N}_0$ and that $\ell_-(n)\to\infty$, $\delta_+(n)\to 0$ with $\ell_-(n)\delta_+(n)\to 0$ and $\lim_{n\to\infty} \frac{\ell_-(n)\nu_-(n)}{\ell_+(n)\nu_+(n)}=1$. Let $\alpha\in(1/2,1]$. Then the sequence
\begin{equation*}
\sqrt{\tfrac{\ell_-(n)}{\nu_+(n)}}\Bigl(U^n-u,z^n(\Theta^n)-p\Bigr) 
\end{equation*}
converges weakly in $H_{-1}\times H^{-2\alpha}$ to the mild solution of the SPDE system
\begin{subequations}\label{compartmental_model_SPDE}
\begin{equation}
\left.\begin{array}{rcrl}
\dif U_t &=&  \displaystyle\Bigl[\bigl(\ntilde\Delta+p(t)\bigr)\, U_t+\bigl(\nbar v-u(s)\bigr)\, P_t\Bigr]\dif t &\\[3ex]
\dif P_t &=&  \ntilde F'\bigl[u(t),p(t)\bigr](U_t,P_t)\,\dif t & + g(t)\,\dif W_t^Q 
\end{array}\right.
\end{equation}
with 
\begin{equation}
\ntilde F'\bigl[u,p\bigr](U,P)\,=\,\bigl(a'(u)\,(1-p)-b'(u)\,p\bigr)\,U-\bigl(a(u)+b(u)\bigr)\, P
\end{equation}
and initial conditions $(U_0,P_0)$ distributed according to $\mathcal{N}_0$ on $H_{-1}\times H^{-2\alpha}$. The process $W^Q$ is a cylindrical Wiener process on a Hilbert space $H'$ with covariance operator $Q\in L(H')$ such that $g(t)\in L(H',H^{-2\alpha})$ satisfies $g(t)Qg^\ast(t)=G(u(t),p(t))\circ\iota^{-1}_{H^{2\alpha}}$, where $G$ is defined via the bilinear form
\begin{equation}
\langle G(u,p)\psi,\phi\rangle_{H^{2\alpha}}\,=\,\int_D \Bigl(a(u(x))\,(1-p(x))+b(u(x))\,p(x)\Bigr)\,\psi(x)\,\phi(x)\,\dif x\,.
\end{equation}
\end{subequations}
\end{theorem}

\noindent Note that the assumption that $\lim_{n\to\infty} \frac{\ell_-(n)\nu_-(n)}{\ell_+(n)\nu_+(n)}=1$ in the above theorem allows for heterogeneity in the spatial repartition of channels in the membrane. 

\bigskip

\begin{proof} We have to establish conditions \assref{C}{ass_C_lipschitz_conds} -- \assref{C}{ass_C_differentiability} for the choice of Hilbert spaces $H=E=L^2$, $X=H_1$, $\cE=H^{2\alpha}$ and $V=H^\alpha$. Then the convergence will follow from Theorem \ref{fluctuation_limit_theorem} and Corollary \ref{corollary_spde_representation}. 

Some of these conditions are easier to check than others. Note that \assref{C}{ass_C_IC_bounds} is satisfied by assumption. Moreover as the embeddings $H_1\hookrightarrow L^2$ and $H^{2\alpha}\hookrightarrow H^{\alpha}$ are of Hilbert-Schmidt type, condition  \assref{C}{ass_C_tightness} follows from \assref{C}{ass_C_lipschitz_conds} -- \assref{C}{ass_C_generator_diff_bound} as discussed in Section \ref{discussion_of_assumptions}. Conditions \eqref{ass_C_covariance_convergence} and \assref{C}{ass_B_large_fluctuations} were established in \cite{RTW} to which we refer the reader. The PDMPs as well as the deterministic limit are pointwise bounded over $D$ independently of $n$ which ensures condition \assref{C}{ass_C_uniform_bound} holds. 

\noindent Hence, we are left with \assref{C}{ass_C_lipschitz_conds}, \assref{C}{ass_C_generator_diff_bound}, \eqref{ass_C_generator_convergence} and \assref{C}{ass_C_differentiability} which we now establish.\medskip

\quad(a)\ Lipschitz conditions in \assref{C}{ass_C_lipschitz_conds}: From the coercivity of the Laplace operator we deduce
\begin{equation*}
\langle \Delta (U^n-u),U^n-u\rangle_{H^1}\,\leq\,-\gamma_1\|U^n-u\|^2_{H^1}+\gamma_2\|U^n-u\|_{L^2}^2
\end{equation*}
for some constants $\gamma_1,\gamma_2>0$. Next, we obtain for $1/2<\alpha\leq 1$ that
\begin{eqnarray*}
\langle z^n(\Theta^n)\,(\nbar v-U^n)-p\,(\nbar v-u),U^n-u\rangle_{H^1}&=&-\bigl(z^n(\Theta^n),(U^n-u)^2\bigr)_{L^2}+\langle z^n(\Theta^n)-p,(\nbar v-u)(U^n-u) \rangle_{H^1}\\[1ex]
&\leq& \|z^n(\Theta^n)-p\|_{H^{-\alpha}}\,\|\nbar v-u\|_{H^\alpha}\,\|U^n-u\|_{H^\alpha}\\[1ex]
&\leq& \frac{1}{4\eps}\,\|z^n(\Theta^n)-p\|_{H^{-\alpha}}^2+\eps\, C\,\|\nbar v-u\|_{H^1}^2\,\|U^n-u\|_{H^1}^2,
\end{eqnarray*}
where $C$ is a constant resulting from the continuous embedding $H^1\hookrightarrow H^\alpha$. Note here, that for these estimates to hold we have to choose $\alpha\leq 1$. This restriction combined with the further condition $\alpha>d/2$ where $d$ is the dimension of the domain $D$, impose that we take $d=1$. Integrating the above two estimates over $[0,T]$ we find that if
\begin{equation*}
S\,:=\,-\gamma_1\int_0^T\|U^n-u\|_{H^1}^2\,\dif t + \eps\, C\int_0^T \|\nbar v-u\|_{H^1}^2\,\|U^n-u\|_{H^1}^2\,\dif t\,\leq 0
\end{equation*}
then the first Lipschitz condition in \assref{C}{ass_C_lipschitz_conds} is satisfied. Standard estimates in the theory of linear partial differential equations, see \cite{Evans}, yield the bound
\begin{equation*}
\|u\|_{L^\infty((0,T),H^1)}\ \leq\ C\bigl(\|u_0\|_{H^1}^2+\|p\cdot(\nbar v-u)\|_{L^2((0,T),L^2)}\bigr)\ \leq\ C\bigl(\|u_0\|_{H^1}^2+T\,\nbar v\bigr)\,.
\end{equation*}
Hence $\|\nbar v-u\|_{L^\infty((0,T),H^1)}$ is finite and we estimate $S$ by
\begin {equation*}
 S\, \leq\, -\gamma_1\int_0^T\|U^n-u\|_{H^1}^2\,\dif t + \eps\, C\,\|\nbar v-u\|_{L^\infty((0,T),H^1)}^2\,\int_0^T\|U^n-u\|_{H^1}^2\,\dif t\,
\end {equation*}
Now, choosing $\eps$ sufficiently small we find that $S\leq0$ and thus there exists a constant $L_3$ such that
\begin{eqnarray*}
\lefteqn{\int_0^T\langle \Delta (U^n-u)+z^n(\Theta^n)\,(\nbar v-U^n)-p\cdot(\nbar v-u),U^n-u\rangle_{H^1}\, \dif t}\\
&&\phantom{xxxxxxxxxxxxxxxxxxxxxxx}\leq\ L_3\int_0^T \|U^n-u\|_{L^2}^2 + \|z^n(\Theta^n)-p\|_{H^{-\alpha}}^2\,\dif t\,. 
\end{eqnarray*}
Hence the first Lipschitz condition in \assref{C}{ass_C_lipschitz_conds} is satisfied. In order to derive the second one  we use that $\|z^n(\Theta^n)\|_{L^\infty}\leq 1$ and obtain
\begin{eqnarray*}
\lefteqn{\|F(u,p)-F(U^n,z^n)\|_{H^{-\alpha}}}\\[1ex]
&\leq&  \|(a(u)+b(u))\, p-(a(U^n)+b(U^n))\, z^n(\theta^n)\|_{H^{-\alpha}}+\|a(u)-a(U^n)\|_{H^{-\alpha}}\\[1ex]
&\leq&  \|(a(u)+b(u))\, (p-z^n(\Theta^n))\|_{H^{-\alpha}}+ C\,\|(a(u)-a(U^n)+b(u)-b(U^n))\, z^n(\Theta^n)\|_{L^2}+C\,\|a(u)-a(U^n)\|_{L^2}\\[1ex]
&\leq&  \|(a(u)+b(u))\cdot (p-z^n(\Theta^n))\|_{H^{-\alpha}}+ 3CL_{ab}\,\|u-U^n\|_{L^2},
\end{eqnarray*}
where $C$ is a constant resulting from the continuous embedding of $L^2$ into $H^{-\alpha}$ and $L_{ab}$ is a common Lipschitz constant for $a$ and $b$. Finally, since
\begin{eqnarray*}
 \|(a(u)+b(u))\,(p-z^n(\Theta^n))\|_{H^{-\alpha}}&=&\sup_{\|v\|_{H^\alpha}=1}\big|\bigl(v, (a(u)+b(u))\, (p-z^n)\bigr)_{L^2}\big|\\[1ex]
&\leq& \sup_{\|v\|_{H^\alpha}=1}\Bigl(\|v\, (a(u)+b(u))\|_{H^\alpha}\,\|p-z^n\|_{H^{-\alpha}}\Bigr)\\[1ex]
&\leq& C\,\|a(u)+b(u)\|_{H^1}\, \|p-z^n(\Theta^n)\|_{H^{-\alpha}},
\end{eqnarray*}
where $C$ is a constant resulting from the continuous embedding $H^2\hookrightarrow H^\alpha$, we find the second Lipschitz condition in \assref{C}{ass_C_lipschitz_conds} is satisfied since $\|a(u)+b(u)\|_{L^\infty((0,T),H^1)}<\infty$. We establish the third Lipschitz condition in \assref{C}{ass_C_lipschitz_conds} by the same estimates with $\alpha$ replaced by $2\alpha$.

\medskip

\quad(b)\ We now establish \assref{C}{ass_C_generator_diff_bound} and \eqref{ass_C_generator_convergence} using that (see \cite{RTW}) for $\alpha>1/2$
\begin{equation}\label{comp_model_asymp_behav_1}
\EX^n\int_0^t\Bigl[\Lambda^n(U^n_s,\Theta^n_s)\int_{K_n}\|z^n(\xi)-z^n(\Theta^n_s)\|_{H^{-\alpha}}^2\,\mu^n\bigl((U^n_s,\Theta^n_s),\dif\xi\bigr)\,\dif s\Bigr]\,=\,\landau\bigl(\nu_+(n)/\ell_-(n)\bigr)\,,
\end{equation}
and 
\begin{equation}\label{comp_model_asymp_behav_2}
\int_0^T\EX^n\big\|\bigl[\mathcal{A}^n\langle z^n_j(\cdot),\,\cdot\,\rangle_{L^2}\bigr](U^n,\Theta^n) -  F(U^n,z^n(\Theta^n))\big\|_{L^2}^2\,\dif s = \landau(\delta^2_+(n))\,.
\end{equation}
The asymptotic behaviour \eqref{comp_model_asymp_behav_1} implies \eqref{ass_C_generator_diff_bound_1}. The asymptotic behaviour \eqref{comp_model_asymp_behav_2} combined with the continuous embedding $L^2\hookrightarrow H^{-\alpha}$ and the assumption $\ell_-(n)\delta_+(n)\to 0$ imply \eqref{ass_C_generator_diff_bound_2}. The same reasoning implies condition \eqref{ass_C_generator_convergence}.

\medskip

\quad(c)\ It remains the consider the differentiability conditions \assref{C}{ass_C_differentiability}. Remember that in this section the operators $A,\,B$ and $F$ are given by 
\begin{equation*}
A(p)=\Delta,\quad B(u,p)=p\,(\nbar v-u),\quad F(u,p)=-(a(u)+b(u))\,p + a(u)\,.
\end{equation*} 
We have already stated that Laplacian can be extended to an operator $\ntilde\Delta_{-\alpha}\in L(H_{-\alpha},H_{-\alpha})$ (cf.~Remark \ref{Hilbert_scales}). As it is independent of $p$, $A=\ntilde\Delta_{-\alpha}$ trivially satisfies the differentiability condition \assref{C}{ass_C_differentiability}. It is easy to see that the operator $B$ is partially differentiable with derivatives $B_u[u,p]\in L(H_1,H_{-1})$ and $B_p[u,p]\in L(L^2,H_{-1})$:
\begin{equation*}
h\mapsto B_u[u,p]h= -ph,\qquad h\mapsto B_p[u,p]h= (\nbar v-u)h\,.
\end{equation*}
As $p$ and $\nbar v-u$ evaluated along the deterministic solution are elements of $H^1$ at each time $t$, these operators immediately extend to operators $\ntilde B_u[u,p]\in L(H_{-1},H_{-1})$ and $\ntilde B_p[u,p]\in L(H^{-2\alpha},H_{-1})$ being densely defined and bounded. For instance, the derivative $\ntilde B_u[u,p]\phi=-p\phi\in H_{-1}$ for $\phi\in H_{-1}$ is defined as
\begin{equation*}
\langle-p\phi,v\rangle_{H_1}:= \langle\phi,-p v\rangle_{H_1}\qquad\forall\,v\in H_1\,.
\end{equation*}
This is well-defined as the pointwise product $pv$ is an element of $H_1$ due to the Banach algebra property of this space on one-dimensional domains. 
Like the operator $B$, the operator $F$ is linear in $p$ and the derivative $F_p[u,p]\in L(L^2,L^2)$ is given by $h\mapsto F_p[u,p] h=-(a(u)+b(u))\, h$ which extends in the same way as $B_p$ above to an operator $\ntilde F_p[u,p]\in L(H^{-2\alpha},H^{-2\alpha})$. Dependence on $u$ is nonlinear, thus the derivative $F_u[u,p]\in L(H^1,L^2)$ is
\begin{equation*}
h\mapsto F_u[u,p]h = -(a'(u)+b'(u))\,p\,h + a'(u)\,h
\end{equation*}
which completely is extended to an operator $\ntilde F_u[u,p]$ in $L(H_{-1},H^{-2\alpha})$ analogously.
\end{proof}

\subsection{Neural field models}

As in the previous compartmental model, one chooses a sequence of partitions $\pindex^n$ of a spatial domain $D\subset\rr^d$. The sequence of PDMPs is given by a sequence of jump processes connected to this sequence of partitions. For this model there is no restriction on the dimension of the domain. The main difference currently is that here we only consider jumping components $\Theta^n$ and no continuous component $U^n$. The deterministic limit is here given by the solution to the Wilson-Cowan model which describes the mean level of activity of populations of neurons and reads as follows:
\begin{equation}\label{wilson_cowan_equation}
\dot p(t,x) = -p(t,x)+f\Bigl(\int_D w(x,y)\,p(t,y)\,\dif y\Bigr)\,,
\end{equation}
The gain function $f$ and the connectivity kernel $w$ are assumed to be smooth and bounded. As previously mentioned only finite number of neurons are accessible so it is natural to consider stochastic models in which the piecewise constant components $\theta^n=(\theta^{1,n},\ldots,\theta^{p(n),n})\in K_n\subseteq\mathbb{N}^{p(n)}$ count the number of active neurons in the individual subdomains. Here $l(k,n)$ is a parameter which is related to the number of neurons in the subdomains in the sense that $l(k,n)\geq \EX^n\Theta^{k,n}_0$. The transition dynamics of the components $\theta^n$ is given by
\begin{equation*}
\theta^n\to\nhat\theta^n \tn{ with rate } 
\left\{\begin{array}{cl}
\theta^{k,n}& \tn{if } \nhat\theta^{k,n}=\theta^{k,n}-1,\, \nhat\theta^{j,n}=\theta^{j,n},\, j\neq k,\\[2.5ex]
l(k,n)\,f\displaystyle\Bigl(\sum_{j=1}^{p(n)} \nbar W^n_{kj} \frac{\theta^{j,n}}{l(j,n)}\Bigr)& \tn{if } \nhat\theta^{k,n}=\theta^{k,n}+1,\, \nhat\theta^{j,n}=\theta^{j,n},\, j\neq k,\\[4.5ex]
0 & \tn{otherwise}\,,
\end{array}\right.
\end{equation*}
where
\begin{equation*}
\nbar W^n_{kj} =\frac{1}{|D^{k,n}|}\int_{D^{k,n}}\Bigl(\int_{D^{j,n}} w(x,y)\,\dif y\Bigr)\,\dif x\,.
\end{equation*}
Finally, as in the previous applications the coordinate functions are given by
\begin{equation}\label{def_field_coordinate}
z^n(\theta^n)=\sum_{k=1}^{p(n)} \frac{\theta^{k,n}}{l(k,n)}\,\mathbb{I}_{D^{k,n}} \in L^2\,.
\end{equation}
The Hilbert spaces are chosen as before, i.e., $E=L^2$, $V=H^\alpha$ and $\cE=H^{2\alpha}$ with $\alpha>d/2$. If the initial conditions converge in probability and $\ell(n)_-\to\infty$, $\delta(n)_+\to 0$, the sequence of stochastic models $z^n(\Theta^n)$ converges u.c.p.~to the solution $p$ of the the Wilson-Cowan equation \eqref{wilson_cowan_equation} in $L^2$. This was proved in \cite{BR_neural_fields} building on \cite{RTW}. We now present the central limit theorem for the global fluctuations.

\begin{theorem} Let $\alpha> d/2$ and assume that the rescaled difference in the macroscopic initial conditions converges to a Gaussian distribution $\mathcal{N}_0$ and that $\ell_-(n)\to\infty$, $\delta_+(n)\to 0$ with $\ell_-(n)\delta_+(n)\to 0$ and $\lim_{n\to\infty} \frac{\ell_-(n)\nu_-(n)}{\ell_+(n)\nu_+(n)}=1$. Then the sequence
\begin{equation*}
\sqrt{\tfrac{\ell_-(n)}{\nu_+(n)}}\,(z^n(\Theta^n)-p)
\end{equation*}
converges weakly to an $H^{-2\alpha}$-valued diffusion process which is a version of the mild solution to the SPDE
\begin{subequations}\label{neural_field_spde}
\begin{equation}
\dif Y_t = \ntilde F_p[p(t)] Y_t\,\dif t + g(t)\,\dif W_t,
\end{equation}
where $\ntilde F_p[p(t)]\in L(H^{-2\alpha},H^{-2\alpha})$ is defined by
\begin{equation}
\langle\ntilde F_p[p]\,h,v\rangle_{H^{2\alpha}} \,=\, -\langle h,v\rangle_{H^{2\alpha}}+\int_D\Bigl(f'\bigl((w(x,\cdot),p)_{L^2}\bigr)\,\langle h,w(x,\cdot)\rangle_{H^{2\alpha}}\Bigr)\, v(x)\,\dif x\qquad\forall\, h,v\in H^{-2\alpha}
\end{equation}
and initial conditions $P_0$ are distributed according to $\mathcal{N}_0$ on $H^{-2\alpha}$. The process $W^Q$ is a cylindrical Wiener process on a Hilbert space $H'$ with covariance operator $Q\in L(H')$ such that $g(t)\in L(H',H^{-2\alpha})$ satisfies $g(t)Q g^\ast(t)=G(p(t))\circ\iota^{-1}_{H^{2\alpha}}$, with $G(p)\in L_1(H^{2\alpha},H^{-2\alpha})$ defined by
\begin{equation}
\langle G(p)h,v\rangle_{H^{2\alpha}} \,=\, \int_D \Bigl(p(x)+f\Bigl(\int_Dw(x,y)\,p(y)\,\dif y\Bigr)\Bigr)\,h(x)\,v(x)\,\dif x\qquad\forall\, h,v\in H^{2\alpha}
\end{equation}
\end{subequations}
\end{theorem}

\begin{proof} As for the compartmental models we only need to establish \assref{C}{ass_C_lipschitz_conds}, \assref{C}{ass_C_generator_diff_bound}, \eqref{ass_C_generator_convergence} and \assref{C}{ass_C_differentiability}. For the other conditions see \cite{BR_neural_fields}. 
We start with the Lipschitz conditions in \assref{C}{ass_C_lipschitz_conds}. It is shown in \cite{BR_neural_fields} that the function $F(p)$ given by the right hand side of \eqref{wilson_cowan_equation} satisfies a Lipschitz condition in every space $H^{-\alpha}$, $\alpha>0$, i.e., there exist constants $L_{-\alpha}$ such that
\begin{equation*}
\|F(g_1,t)-F(g_2,t)\|_{H^{-\alpha}}\leq L_{-\alpha}\|g_1-g_2\|_{H^{-\alpha}}\qquad\forall\,t\geq 0, g_1,g_2\in L^2\,. 
\end{equation*}
Therefore condition \assref{C}{ass_C_lipschitz_conds} holds. The same asymptotic behaviours as \eqref{comp_model_asymp_behav_1} and \eqref{comp_model_asymp_behav_2} for the neural field case are proven in \cite{BR_neural_fields}. Analogous arguments yield again conditions \assref{C}{ass_C_generator_diff_bound} and \eqref{ass_C_generator_convergence} are satisfied. 
We now consider the differentiability of the right hand side $F$ of (\ref{wilson_cowan_equation}) to establish condition \assref{C}{ass_C_differentiability}. It is straightforward to see that $F$ is strongly Fr\'echet-differentiable in $L^2$ for $f\in C^2_b(\rr)$ with derivative $F_p[p] \in L(L^2,L^2)$ given by
\begin{equation*}
h\mapsto F_p[p]\,h:=-h+f'\Bigl(\int_D w(\cdot,y)p(y)\,\dif y\Bigr)\,\int_D w(\cdot,y)h(y)\,\dif y\,. 
\end{equation*}
For sufficiently smooth $y\mapsto w(x,y)$ this map extends to an operator in $L(H^{-2\alpha},H^{-2\alpha})$ via the definition
\begin{equation}
\langle\ntilde F_p[p]\,h,v\rangle_{H^{2\alpha}} \,=\, -\langle h,v\rangle_{H^{2\alpha}}+\int_D\Bigl(f'\bigl((w(x,\cdot),p)_{L^2}\bigr)\,\langle h,w(x,\cdot)\rangle_{H^{2\alpha}}\Bigr)\, v(x)\,\dif x\,.
\end{equation}
Therefore condition \assref{C}{ass_C_differentiability} is satisfied. 
\end{proof}

%\textbf{Acknowledgements:}

\begin{appendix}

\section{Appendix}

In this Appendix we have gathered detailed results and remarks mentioned without details in the main text.  

\subsection {A function satisfying the assumptions of Theorem \ref{PDMP_gen_theorem}}\label{appendix_function_example}

We show here that the function $f$ given by (\ref{def_of_a_special_f}) satisfies the assumptions of Theorem \ref{PDMP_gen_theorem} (b). This function is complex which however, does not pose any difficulties as $\mathbb{C}$ can be identified with $\rr^2$ and $f$ considered componentwise in which case it possesses a bounded real and imaginary part as $|f(u,\theta)|_\mathbb{C}=1$ for all $(u,\theta)\in H\times K$. As $f$ is bounded the integrability with respect to the compensating measure is immediate. Moreover, its integral with respect to the associated fundamental martingale measure is a martingale due to \cite[Thm.~4.6.1]{Jacobsen} and thus Dynkin's formula \ref{Dynkin_formula} holds. Furthermore, $f$ is partially continuously Fr\'echet differentiable with respect to the first argument due to the chain rule and the Riesz Representation $f_u$ of the partial derivative is given by
\begin{equation*}
f_u(u,\theta)= i\,\exp\Bigl(i\,\langle \psi,u \rangle_H+i\,\langle\Phi,z(\theta)\rangle_E\Bigr)\,\psi \in H^\ast\,,
\end{equation*}
where $H$ is considered a complex Hilbert space. As $\psi \in X$ it holds that $f_u(u,\theta)\in X$ and finally it is easy to see that $f_u$ is a locally bounded composition operator as
\begin{equation*}
\Big\|i\,\exp\Bigl(i\,\langle \psi,u \rangle_H+i\,\langle\Phi,z(\theta)\rangle_E\Bigr)\psi\Big\|_{L^2((0,T),X)}^2 
\,\leq\, T \|\psi\|_X^2
\end{equation*}
for all $(u,\theta)\in L^2((0,T),X\times K)$. Hence, formula \eqref{infindimGen} holds for $f$ as defined in \eqref{def_of_a_special_f}

\subsection{An auxiliary convergence result}\label{app_trace_convergence}

Let us prove a convergence result for the trace of the quadratic variations of the PDMPs. This result was previously stated without proof in \cite{RTW}. 

\setcounter{ass}{001}
\begin{assumptions}\label{assumptions_trace_conv} 
\begin{enumerate_ass}
\item \label{ass_B_weaker_conditions_1}For all $T>0$,
\begin{eqnarray*}
\sup_{n\in\mathbb{N}}\alpha_n\,\EX^n\int_0^T \trace G^n(U^n_s,\Theta^n_s)\,\dif s < \infty\,,
\end{eqnarray*}
where $G^n(u,\theta^n)$ is the operator from $\cE$ to $\cE^\ast$ defined in \eqref{definition_quad_var_operator}.

\item \label{ass_B_weaker_conditions_2}There exists an orthonormal basis $(\varphi_k)_{k\in\mathbb{N}}$ of $\cE$ such that for all $k\in\mathbb{N}$ and all $T>0$
\begin{equation*}
\alpha_n\,\EX^n\Bigl[\int_0^T\langle G^n(U^n_s,\Theta^n_s)\varphi_k,\varphi_k\rangle_\cE\,\dif s\,\Bigr]\ \leq\  \gamma_k\,C(T).
\end{equation*}
The constants $\gamma_k>0$ are independent of $n,\,T$ and satisfy $\sum_{k\in\mathbb{N}}\gamma_k<\infty$. The constant $C(T)>0$ is independent of $n,\,k$.
\end{enumerate_ass}
\end{assumptions}

Note that assumption \assref{B}{ass_B_weaker_conditions_1} is implied  by \eqref{ass_C_generator_diff_bound_1} and assumption \assref{B}{ass_B_weaker_conditions_2} is in applications usually an intermediate result in establishing tightness of the PDMP sequence. We can then prove the following result.

\begin{proposition}\label{app_trace_con_prop} The \tn{Assumptions \ref{assumptions_trace_conv}} and the convergence \eqref{ass_C_covariance_convergence} imply
\begin{equation*}
\lim_{n\to\infty} \alpha_n\int_0^T\EX^n\trace G^n(U^n_s,\Theta^n_s)\,\dif s \ =\ \int_0^T \trace G(u(s),p(s))\,\dif s\,.
\end{equation*}
\end{proposition}

\begin{proof} %
Let $(\Phi_k)_{k\in\mathbb{N}}$ be an orthonormal basis of $\mathcal{E}$ then it holds by Jensen's inequality that
\begin{eqnarray*}
\lefteqn{\int_0^T\EX^n\big|\bigl\langle G(u(s),p(s))\,\Phi_k,\Phi_k\bigr\rangle_\cE-\alpha_n\bigl\langle G^n(U^n_s,\Theta^n_s)\,\Phi_k,\Phi_k\bigr\rangle_\cE\big|\,\dif s}\\
&&\phantom{xxxxxxxxxxxxxxx}\geq\ \Big|\int_0^T\bigl\langle G(u(s),p(s))\,\Phi_k,\Phi_k\bigr\rangle_\cE\,\dif s-\alpha_n\int_0^T\EX^n\bigl\langle G^n(U^n_s,\Theta^n_s)\,\Phi_k,\Phi_k\bigr\rangle_\cE\,\dif s\Big|\,.
\end{eqnarray*}
As the left hand side converges to zero due to \eqref{ass_C_covariance_convergence} it follows that for all $\Phi_k,\,k\in\mathbb{N}$,
\begin{equation*}
\lim_{n\to\infty} \alpha_n\int_0^T\EX^n\bigl\langle G^n(U^n_s,\Theta^n_s)\,\Phi_k,\Phi_k\bigr\rangle_\cE\,\dif s \ = \ \int_0^T\bigl\langle G(u(s),p(s))\,\Phi_k,\Phi_k\bigr\rangle_\cE\,\dif s\,.
\end{equation*}
This implies that any finite sum of terms in the right hand side converges to the corresponding finite sum of the limits in the left hand side. Moreover due to dominated convergence,
\begin{eqnarray}
\int_0^T\!\!\EX^n\,\textnormal{Tr}\,G^n(U^n_s,\Theta^n_s)\,\dif s& =& \int_0^T\!\!\EX^n\sum_{k\in\mathbb{N}}\bigl\langle G^n(U^n_s,\Theta^n_s)\,\Phi_k,\Phi_k\bigr\rangle_\cE\,\dif s\nonumber\\ 
&=& \sum_{k\in\mathbb{N}}\int_0^T\!\!\EX^n\bigl\langle G^n(U^n_s,\Theta^n_s)\,\Phi_k,\Phi_k\bigr\rangle_\cE\,\dif s \label{dom_conv_seq_trace}
\end{eqnarray}
and analogously
\begin{equation}\label{dom_conv_limit_trace}
\int_0^T \textnormal{Tr}\,G(u(s),p(s))\,\dif s = \sum_{k\in\mathbb{N}}\int_0^T\bigl\langle G(u(s),p(s))\,\Phi_k,\Phi_k\bigr\rangle_\cE\,\dif s\,.
\end{equation}

We show now that the partial sums for the sequence of PDMP models converge to the trace {\it uniformly} in $n$:
\begin{eqnarray*}
\lefteqn{\forall\,\eps>0\ \exists\, M:\ \forall\,m >M:}\\[3ex]
&\Big|\displaystyle\sum_{k\leq m} \alpha_n\int_0^T\EX^n\bigl\langle G^n(U^n_s,\Theta^n_s)\,\Phi_k,\Phi_k\bigr\rangle_\cE\,\dif s - \alpha_n\int_0^T\EX^n\trace G^n(U^n_s,\Theta^n_s)\,\dif s \Big| < \eps\quad \forall\,n\in\mathbb{N}\,.&
\end{eqnarray*}
Note that \eqref{dom_conv_seq_trace} implies
\begin{eqnarray*}
\lefteqn{\Big|\sum_{k\leq m} \alpha_n\int_0^T\EX^n\bigl\langle G^n(U^n_s,\Theta^n_s)\,\Phi_k,\Phi_k\bigr\rangle_\cE\,\dif s - \alpha_n\int_0^T\EX^n\trace G^n(U^n_s,\Theta^n_s)\,\dif s \Big|}\\
&&\phantom{xxxxxxxxxxxxxxxxxxxxxxxxxxxxxxxxxxx} = \ \sum_{k>m} \alpha_n\int_0^T\EX^n\bigl\langle G^n(U^n_s,\Theta^n_s)\,\Phi_k,\Phi_k\bigr\rangle_\cE\,\dif s\,.
\end{eqnarray*}
Due to condition \assref{B}{ass_B_weaker_conditions_2} there exists a sequence of $\gamma_k>0$ independent of $t,n$ with $\sum_{k\in\mathbb{N}}\gamma_k<\infty$ and a constant $C(T)>0$ independent of $n$ such that
\begin{equation*}
\sum_{k>m} \alpha_n\int_0^T\EX^n\bigl\langle G^n(U^n_s,\Theta^n_s)\,\Phi_k,\Phi_k\bigr\rangle_\cE\,\dif s\ \leq \ C(T)\sum_{k>m}\gamma_k\,.
\end{equation*}
As $\sum_{k\in\mathbb{N}}\gamma_k<\infty$ for every $\eps>0$ there exists an $M$ such that
\begin{equation*}
C(T)\sum_{k>m}\gamma_k\ < \ \eps\quad\forall\, m>M\,,
\end{equation*}
where $M$ does not depend on $n\in\mathbb{N}$. This yields the uniform convergence.\medskip

Uniform convergence is important because it allows to change the order of taking limits. Thus we obtain
\begin{eqnarray*}
\lefteqn{\lim_{n\to\infty}\lim_{m\to\infty}\sum_{k\leq m}\alpha_n\int_0^T\EX^n\bigl\langle G^n(U^n_s,\Theta^n_s)\,\Phi_k,\Phi_k\bigr\rangle_\cE\,\dif s}\\&&\phantom{xxxxxxxxxxxxx}=\ \lim_{m\to\infty}\lim_{n\to\infty}\sum_{k\leq m}\alpha_n\int_0^T\EX^n\bigl\langle G^n(U^n_s,\Theta^n_s)\,\Phi_k,\Phi_k\bigr\rangle_\cE\,\dif s\\
&&\phantom{xxxxxxxxxxxxx}=\ \lim_{m\to\infty} \sum_{k\leq m} \int_0^T\bigl\langle G(u(s),p(s))\,\Phi_k,\Phi_k\bigr\rangle_\cE\,\dif s\\
&&\phantom{xxxxxxxxxxxxx}=\ \int_0^T \textnormal{Tr}\,G(u(s),p(s))\,\dif s\,,
\end{eqnarray*}
where we have used \eqref{dom_conv_limit_trace} for the last equality. The proof is completed.
\end{proof}

\subsection{Bounds for the application of the Dominated Convergence Theorem}\label{app_dom_conv_appl}

We now turn to some detailed estimation procedures that we did not include in the proof of the Central Limit Theorem: bounds for the application of the Dominated Convergence Theorem, for, on the one hand, interchanging differentiation and taking expectation and, on the other hand, for concluding from pointwise convergence to convergence of the integrals. In the following $\mathcal{H}$ denotes appropriately either the space $X$ or $\cE$. Analogously $\phi$ denotes either $\psi\in X$ or $\Phi\in\cE$, respectively, and so does $Z^{\mathcal{H}^\ast}$ denote either the $X^\ast$-- or the $\cE^\ast$--component of the limiting process $Z$.\medskip

We start presenting an argument which allows to interchange Fr\'echet differentiation and taking the expectation. Assume that 
\begin{equation}\label{app_swapping_what_we_need}
 \lim_{\|x\|_\mathcal{H}\to 0}\Big\|\EX\Bigl(\frac{h(Z_s,\phi+x)-h(Z_s,\phi)-\bigl\langle\frac{\partial h(Z_s)}{\partial\phi}[\phi],x\bigr\rangle_B}{\|x\|_\mathcal{H}}\Bigr)\Bigr\|_\mathbb{C}=0
\end{equation}
which is equivalent to
\begin{equation*}
 \lim_{\|x\|_\mathcal{H}\to 0}\Big\|\frac{\varphi(s,\phi+x)-\varphi(,s\phi)-\bigl\langle\EX\frac{\partial h(Z_s)}{\partial\phi}[\phi],x\bigr\rangle_\mathcal{H}}{\|x\|_\mathcal{H}}\Bigr\|_\mathbb{C}=0\,.
\end{equation*}
Thus the uniqueness of the derivative implies that $\frac{\partial\varphi}{\partial\phi}[\phi]=\EX\frac{\partial h(Z_s)}{\partial\phi}[\phi]$ which is what we aim for. Hence, in order to obtain \eqref{app_swapping_what_we_need} it suffices to show that
\begin{equation*}
 \lim_{\|x\|_\mathcal{H}\to 0}\EX\frac{\big\|h(Z_s,\phi+x)-h(Z_s,\phi)-\bigl\langle\frac{\partial h(Z_s)}{\partial\phi}[\phi],x\bigr\rangle_\mathcal{H}\big\|_\mathbb{C}}{\|x\|_\mathcal{H}}=0\,.
\end{equation*}
However, here the term inside the expectation converges to zero almost surely, hence we can employ the Dominated Convergence Theorem to conclude from pointwise limits to limits of the expectation. That is, we are left to find an integrable upper bound to the terms
\begin{equation*}
 \frac{\|h(Z_s,\phi+x)-h(Z_s,\psi)\|_\mathbb{C}}{\|x\|_\mathcal{H}}\quad\tn{and}\quad \Big\|\frac{\partial h(Z_s)}{\partial\phi}[\phi]\Big\|_{L(\mathcal{H},\mathbb{C})}.
\end{equation*}
In order to bound the first term we employ the Mean Value Theorem and obtain
\begin{eqnarray*}
\frac{\big\|h(Z_s,\phi+x)-h(Z_s,\phi)\big\|_\mathbb{C}}{\|h\|_\mathcal{H}} &=&\frac{1}{\|x\|_\mathcal{H}}\Big\|\Bigl(\int_0^1 i Z^{\mathcal{H}^\ast}_s h(Z_s,\phi+rx)\dif r\Bigr)x\Big\|_\mathbb{C}\\
&\leq& \int_0^1 \big\|i Z^{\mathcal{H}^\ast}_s h(Z_s,\psi+rx)\big\|_{\mathcal{H}^\ast}\dif r\\
&\leq&\|Z^{\mathcal{H}^\ast}_s\|_{\mathcal{H}^\ast}\,.
\end{eqnarray*}
For the second term we obtain
\begin{eqnarray*}
\Big\|\frac{\partial h(Z_s)}{\partial\phi}[\phi]\Bigr\|_{L(\mathcal{H},\mathbb{C})}=\|i\,h(Z_s,\phi)\,Z_s^{\mathcal{H}^\ast}\|_{\mathcal{H}^\ast} \leq \|Z_s^{\mathcal{H}^\ast}\|_{\mathcal{H}^\ast}\,.
\end{eqnarray*}
Clearly, the upper bound is integrable, thus we have completed the argument. \medskip

For the second application of the Dominated Convergence Theorem in order to conclude from pointwise convergence of $\lim_{n\to\infty}\EX^n \bigl\langle\frac{\partial h(Z^n_s,\psi)}{\partial\psi},g(s)\bigr\rangle_B=\EX\bigl\langle\frac{\partial h(Z_s,\psi)}{\partial\psi},g(s)\bigr\rangle_\mathcal{H}$ for almost all $s\leq t$ to
\begin{equation*}
\lim_{n\to\infty}\int_0^t\EX^n\bigl\langle\frac{\partial h(Z^n_s,\psi)}{\partial \psi},g(s)\bigr\rangle_\mathcal{H}\,\dif s =  \int_0^t\EX^n\bigl\langle\frac{\partial h(Z_s,\psi)}{\partial \psi},g(s)\bigr\rangle_\mathcal{H}\,\dif s
\end{equation*}
we need the integrands in the left hand side to be bounded in $n$ almost everywhere by a function integrable over $(0,t)$. Here $g(s)$ represents the points where the derivatives are evaluated such as, e.g., $\ntilde B^\ast_u[u(s),p(s)]\phi$. As these derivatives are continuous -- as is the deterministic solution $(u(s),p(s))$ -- it holds that $\sup_{s\in[0,t]}\|g(s)\|_\mathcal{H}<\infty$. Simple estimates yield
\begin{eqnarray*}
 \EX^n\bigl\langle\frac{\partial h(Z^n_s,\psi)}{\partial \psi},g(s)\bigr\rangle_\mathcal{H} &\leq& \EX^n\Big\|\frac{\partial h(Z^n_s,\psi)}{\partial \psi}\|_{\mathcal{H}^\ast}\|g(s)\|_{\mathcal{H}}\\
&\leq& \sup_{s\in[0,t]}\|g(s)\|_\mathcal{H}\,\EX^n \|i Z_s^{n,B^\ast} h(Z_s,\psi)\|_{\mathcal{H}^\ast}\\
&\leq& \sup_{s\in[0,t]}\|g(s)\|_\mathcal{H}\,\EX^n \|Z_s^{n,\mathcal{H}^\ast}\|_{\mathcal{H}^\ast}.
\end{eqnarray*}
Here the last expectation in the right hand side possesses a uniform bound independent of $n\in\mathbb{N}$ and $s\leq t$ due to Lemma \ref{lemma_uniform_integrable}. Therefore there exists an integrable dominating function.

\section{The vanishing remainder term}\label{app_vanishing_remainder}

As announced in Remark \ref{remainders} we now study the remainder terms resulting from the expansions employed in Section \ref{proof_section_characterisation}.  Some calculus in normed spaces is needed.

\subsection{Mean Value Theorem}

The key tool is a Mean Value Theorem for Gateaux derivatives. To introduce the central objects let $Y_1$ be a normed vector space (not necessarily complete) and $Y_2$ be a Banach space. Then the space $L(Y_1,Y_2)$ of all bounded, linear functions from $Y_1$ to $Y_2$ equipped with the operator norm is a Banach space \cite[Thm.~II.1.4]{Werner}. If for a function $F:Y_1\to Y_2$ the strong limit
\begin{equation}
\lim_{t\to 0}\frac{F(x+th)-F(x)}{t}=:F_u[x]h 
\end{equation}
exists, it is called the \emph{Gateaux differential at $x$ in direction $h$}. The Gateaux differential is unique, if it exists. We say $F$ is \emph{Gateaux differentiable at $x$} if the Gateux differential at $x$ exists in direction of all $h\in Y_1$ and $F$ is \emph{Gateaux differentiable}, if it is Gateaux differentiable at all $x\in Y_1$. If $F$ is Gateaux differentiable, then we call $F_u$ the \emph{Gateaux derivative} of $F$ if the map $h\mapsto F_u[x]h$ is in $L(Y_1,Y_2)$ for every $x\in Y_1$. It is not necessarily the case that a Gateaux differentiable $F$ possesses a Gateaux derivative. Finally, we say that $F$ is \emph{continuously Gateux differentiable} when the Gateux derivative exists and depends continuously on $x$, i.e., $x\mapsto F_u[x]$ is a continuous map from $Y_1$ into the Banach space $L(Y_1,Y_2)$. Then for continuously Gateaux differentiable $F$ the following Mean Value Theorem holds. 

\begin{proposition}\label{prop_mean_value_theorem} Let $F$ be continuously Gateux differentiable, then for every $x,h\in Y_1$,
\begin{equation}\label{gateaux_mean_value_theorem}
F(x+h)-F(x)\,=\,\Bigl(\int_0^1 F_u[x+\zeta h]\,\dif\zeta\Bigr)h, 
\end{equation}
where the integral in the right hand side is a Bochner integral in the Banach space $L(Y_1,Y_2)$.
\end{proposition}

\begin{proof} We set $f(t):=F(x+th)$ which is Fr\'echet differentiable in $t$ with derivative $f'[t]=F_u[x+th]h\in Y_2\triangleq L((0,1),Y_2)$ depending continuously on $t$. The Fundamental Theorem of Calculus for Fr\'echet derivatives implies that the left hand side of \eqref{gateaux_mean_value_theorem} equals
\begin{equation*}
 F(x+h)-F(x)\,=\,f_1(1)-f_1(0)\,=\,\int_0^1 F_u[u+\zeta h]h\,\dif\zeta\,.
\end{equation*}
The map $\zeta\mapsto F_u[x+\zeta h]$ is continuous from $[0,1]$ to $L(Y_1,Y_2)$ due to the continuous differentiability of $F$ for given $x,h\in Y_1$. Hence, the integral in the right hand side of \eqref{gateaux_mean_value_theorem} exists in the sense of Bochner and its properties allow that
\begin{equation*}
\Bigl(\int_0^1 F_u[x+\zeta h]\,\dif\zeta\Bigr)h \,=\, \int_0^1 F_u[x+\zeta h]h\,\dif\zeta. 
\end{equation*}
Thus the left and right hand sides in \eqref{gateaux_mean_value_theorem} are equal and the proposition is proven.
\end{proof}

\subsection{Remainder terms}

We study the remainder term $\eps_{B,n}$ in more detail. The vanishing of $\eps_{A,n}$ and $\eps_{F,n}$ follows analogously. Thus the aim is to show that for all $T>0$
\begin{equation}\label{app_vanishing_result}
\lim_{n\to\infty} \int_0^T \EX\bigl[i\sqrt{\alpha_n}\,h(Z^n_t)\,\eps_{B,n}(t)\bigr]\,\dif t\ = \ 0\,,
\end{equation}
where $\eps_{B_n}$ is the term resulting from the expansion \eqref{expansion_in_B}.\medskip

The Mean Value Theorem in Proposition \ref{prop_mean_value_theorem} yields an expression for $\eps_{B,n}$ resulting from an expansion of the map $B:H\times E\to X^\ast$ which is partially continuously Gateaux differentiable due to assumption \assref{C}{ass_C_differentiability}. We obtain the equality 
\begin{eqnarray*}
\lefteqn{B(u+(U^n-u),p+(z^n-p))}\\[2ex]
&=&B(u,p+(z^n-p))+\Bigl(\int_0^1 B_u[u+\zeta(U^n-u),p+(z^n-p)]\,\dif\zeta\Bigr)(U^n-u)\\[2ex]
&=&B(u,p)+\Bigl(\int_0^1 \underbrace{B_u[u+\zeta(U^n-u),p+(z^n-p)]}_{\in L(H,X^\ast)}\,\dif\zeta\Bigr)(U^n-u)\\[2ex]
&&\mbox{} + \Bigl(\int_0^1 \underbrace{B_p[u,p+\zeta(z^n-p)]}_{\in L(E,X^\ast)}\,\dif\zeta\Bigr)(z^n-p)\,,
\end{eqnarray*}
where for the sake of simplicity we omit the time argument. By definition \eqref{expansion_in_B} the remainder term satisfies
\begin{eqnarray*}
\eps_{B,n}& =& \langle B(U^n,z^n(\Theta^n))-B(u,p),\psi\bigr\rangle_X
- \langle B_u[u,p](U^n-u),\psi\rangle_X - \langle B_p[u,p](z^n(\Theta^n)-p),\psi\rangle_X\\[2ex]
&=:& \langle \nhat \eps_{B,n},\psi \rangle_X\,.
\end{eqnarray*}
Hence we get 
\begin{eqnarray*}
X^\ast\ni\nhat \eps_{B,n}&=& \Bigl(\int_0^1 B_u[u+\zeta(U^n-u),p+(z^n-p)]\,\dif\zeta\Bigr)(U^n-u)-B_u[u,p](U^n-u)\\[2ex]
&&\mbox{} + \Bigl(\int_0^1 B_p[u,p+\zeta(z^n-p)]\,\dif\zeta\Bigr)(z^n-p)-B_p[u,p](z^n-p)\,.
\end{eqnarray*}
Therefore, the remainder term is given by
\begin{eqnarray*}
i\sqrt{\alpha_n}\,h(Z^n)\,\eps_{B,n}&=&i\bigl\langle\Bigl(\int_0^1 B_u[u+\zeta(U^n-u),p+(z^n-p)]-B_u[u,p]\,\dif\zeta\Bigr)\Bigl(\sqrt{\alpha_n}(U^n-u)\Bigr),\psi\bigr\rangle_X\,h(Z^n)\\[2ex]
&&\mbox{} + i\bigl\langle\Bigl(\int_0^1 B_p[u,p+\zeta(z^n-p)]-B_p[u,p]\,\dif\zeta\Bigr)\Bigl(\sqrt{\alpha_n}(z^n-p)\Bigr),\psi\bigr\rangle_X\,h(Z^n)\,.
\end{eqnarray*}
This is a continuous function of the random vector $\bigl(U^n-u,z^n-p,\sqrt{\alpha_n}(U^n-u),\sqrt{\alpha_n}(z^n-u)\bigr)$ which converges weakly to the random vector $(0,0,Z)$.\footnote{The last two components being $Z^n$ converge weakly jointly to $Z$ and the first two components converge due to Theorem \ref{theorem_lln} (law of large numbers) weakly to the constant $0$. Hence joint weak convergence of the whole vector holds.} Due to the Continuous Mapping Theorem the error thus converges weakly to the functions evaluated at the limit which gives zero. We are interested in the convergence of the expectations hence due to \cite[Appendix, Prop.~2.3]{EthierKurtz} it remains to show the uniform integrability of the errors $i\sqrt{\alpha_n}\,h(Z^n)\,\eps_{B,n}$.\medskip

Due to the de la Vallee-Poussin Theorem it is sufficient to show that the $q$-th moments are uniformly bounded for some $q>1$ which we infer from the Lipschitz continuity of the partial derivatives in the following. We choose some $q\in(0,2)$. First we obtain the estimate
\begin{equation*}
\EX^n\|i\sqrt{\alpha_n} h(Z^n)\,\eps_{B,n}\|^q_\mathbb{C}\,\leq\, \EX^n\|\sqrt{\alpha_n}\,\nhat\eps_{B,n}\|^q_{X^\ast}\,\|\psi\|^q_X
\end{equation*}
and further estimating the stochastic term in the right hand side yields
\begin{eqnarray*}
\lefteqn{\EX^n\|\sqrt{\alpha_n}\,\nhat\eps_{B,n}\|^q_{X^\ast}}\\[2ex]
&\leq& 2 \EX^n\Bigl[\Big\|\int_0^1 \ntilde B_u[u+\zeta(U^n-u),p+(z^n-p)]-\ntilde B_u[u,p]\,\dif\zeta \Big\|_{L(X^\ast,X^\ast)} \bigl(\sqrt{\alpha_n}\,\|U^n-u\|_{X^\ast}\bigr)\Bigr]^q \\
&&\mbox{} +2 \EX^n\Bigl[\Big\|\int_0^1 \ntilde B_p[u,p+\zeta(z^n-p)]-\ntilde B_p[u,p]\,\dif\zeta\Big\|_{L(\cE^\ast,X^\ast)} \bigl(\sqrt{\alpha_n}\,\|z^n-p\|_{\cE^\ast}\bigr)\Bigr]^q\\
&\leq& 2\EX^n\Bigl[\int_0^1 \big\|\ntilde B_u[u+\zeta(U^n-u),p+(z^n-p)]-\ntilde B_u[u,p] \big\|_{L(X^\ast,X^\ast)}\,\dif\zeta \bigl(\sqrt{\alpha_n}\,\|U^n-u\|_{X^\ast}\bigr)\Bigr]^q \\
&&\mbox{} +2 \EX^n\Bigl[\int_0^1 \big\|\ntilde B_p[u,p+\zeta(z^n-p)]-\ntilde B_p[u,p]\big\|_{L(\cE^\ast,X^\ast)}\,\dif\zeta \bigl(\sqrt{\alpha_n}\,\|z^n_t-p\|_{\cE^\ast}\bigr)\Bigr]^q\\
&\leq& 2L^2\,\EX^n\Bigl[\bigl(\|U^n-u\|_X+\|z^n-p\|_E\bigr) \bigl(\sqrt{\alpha_n}\,\|U^n-u\|_{X^\ast}^2\bigr)\Bigr]^q \\
&&\mbox{} + 2L^2\,\EX^n\Bigl[\|z^n-p\|_E \bigl(\sqrt{\alpha_n}\,\|z^n-p\|_{\cE^\ast}\bigr)\Bigr]^q\,.
\end{eqnarray*}
Here $L>1$ is a common Lipschitz constant on the derivatives $\ntilde B_u$ and $\ntilde B_p$. Next using Young's inequality we obtain
\begin{eqnarray*}
\lefteqn{\EX^n\|\sqrt{\alpha_n}\,\nhat\eps_{B,n}\|^q_{X^\ast}}\\[2ex]
&\leq& C\EX^n\Bigl[\tfrac{2-q}{2}\bigl(\|U^n-u\|_X+\|z^n-p\|_E\bigr)^{\tfrac{2}{2-q}}+\tfrac{q}{2}\,\bigl(\sqrt{\alpha_n}\,\|U^n-u\|_{X^\ast}\bigr)^{\tfrac{2}{q}}\Bigr]^q \\[1ex]
&& \mbox{}+ C\EX^n\Bigl[\tfrac{2-q}{2}\,\|z^n-p\|_E^{\tfrac{2}{2-q}}+\tfrac{q}{2}\,\bigl(\sqrt{\alpha_n}\,\|z^n-p\|_{\cE^\ast}\bigr)^{\tfrac{2}{q}}\Bigr]^q\\[1ex]
&\leq& C\Bigl(\EX^n\|U^n-u\|_X^{\tfrac{2q}{2-q}}+\EX\|z^n-p\|_E^{\tfrac{2q}{2-q}}+\alpha_n\EX^n\|U^n-u\|_{X^\ast}^2 + \alpha_n\EX^n\|z^n-p\|^2_{\cE^\ast}\Bigr)
\end{eqnarray*}
where $\tfrac{2q}{2-q}\downarrow 2$ for $q\to 1$. Moreover the terms $\alpha_n \|U^n-u\|_{\cX^\ast}^2$ and $\alpha_n\,\|z^n-p\|_{\cE^\ast}^2$ are uniformly bounded due to Lemma \ref{lemma_uniform_integrable} and the terms $\EX\|U^n-u\|_X^{\tfrac{2q}{2-q}}$ and $\EX\|z^n-p\|_E^{\tfrac{2q}{2-q}}$ are uniformly bounded by assumption \assref{C}{ass_C_uniform_bound} for some $q\in(1,2)$. Thus we obtain
\begin{equation*}
 \lim_{n\to\infty} \EX^n\bigl[ i\sqrt{\alpha_n}\,h(Z^n_t)\,\eps_{B,n}(t)\bigr]\,=\, 0\quad\forall\,t\geq 0
\end{equation*}
and the just derived uniform boundedness over compact intervals implies \eqref{app_vanishing_result} due to the Dominated Convergence Theorem.

\end{appendix}

\bibliographystyle{plain}
\bibliography{bibliography}

\begin{thebibliography}{10}

\bibitem{Adams}
R.~A. Adams and J.~F.~J. Fournier.
\newblock {\em Sobolev Spaces, 2nd Ed.}
\newblock Academic Press, 2003.

\bibitem{Applebaum}
D.~Applebaum.
\newblock {\em L\'evy Processes and Stochastic Calculus, 2nd Ed.}
\newblock Cambridge University Press, Cambridge, 2009.

\bibitem{Austin}
T.~D. Austin.
\newblock The emergence of the deterministic {H}odgkin-{H}uxley equations as a
  limit from the underlying stochastic ion-channel mechanism.
\newblock {\em Ann.~Appl.~Prob.}, 18(4):1279--1325, 2006.

\bibitem{Billingsley}
P.~Billingsley.
\newblock {\em Convergence of Probability Measures, 2nd Ed.}
\newblock Wiley, New York, 1999.

\bibitem{Blount1}
D.~Blount.
\newblock Comparison of stochastic and deterministic models of a linear
  chemical reaction with diffusion.
\newblock {\em Ann.~Probab.}, 19(4):1440--1462, 1991.

\bibitem{Bressloff1}
P.~C. Bressloff.
\newblock Stochastic neural field theory and the system-size expansion.
\newblock {\em SIAM J.~Appl.~Math.}, 70:1488--1521, 2009.

\bibitem{Bressloff_review}
P.~C. Bressloff.
\newblock Spatiotemporal dynamics of continuum neural fields.
\newblock {\em J.~Phys.~A: Math.~Theor.}, 45:033001, 2012.

\bibitem{BuckwarRiedler}
E.~Buckwar and M.~G. Riedler.
\newblock Exact modelling of neuronal membranes including spatio-temportal
  evolution.
\newblock {\em J.~Math.~Bio.}, 63(6):1051--1093, 2011.

\bibitem{DaPratoZabczyk}
G.~Da~Prato and J.~Zabczyk.
\newblock {\em Stochastic Equations in Infinite Dimensions}.
\newblock Cambridge University Press, Cambridge, 1992.

\bibitem{Davis2}
M.~H.~A. Davis.
\newblock {\em Markov Models and Optimisation}.
\newblock Chapman and Hall, London, 1993.

\bibitem{Deimling}
K.~Deimling.
\newblock {\em Ordinary Differential Equations in Banach Spaces}.
\newblock Springer, Berlin, 1977.

\bibitem{EthierKurtz}
S.~N. Ethier and Kurtz~T. G.
\newblock {\em Markov Processes: Characterization and Convergence}.
\newblock Wiley, New York, 1986.

\bibitem{Evans}
L.~C. Evans.
\newblock {\em Partial Differential Equations, 2nd Ed.}
\newblock AMS, 2010.

\bibitem{FaisalLaughlinWhite}
A.~A. Faisal, J.~A. White, and S.~B. Laughlin.
\newblock Ion-channel noise places limits on the miniaturization of the brain's
  wiring.
\newblock {\em Current Biol.}, 15:1143--1149, 2005.

\bibitem{GenadotThieullen1}
A.~Genadot and M.~Thieullen.
\newblock Averaging for a fully coupled {P}iecewise {D}eterministic {M}arkov
  {P}rocess in infinite dimension.
\newblock {\em Adv.~in Appl.~Probab.}, 44(3):749--773, 2012.

\bibitem{GenadotThieullen2}
A.~Genadot and M.~Thieullen.
\newblock Multiscale {P}iecewise {D}eterministic {M}arkov {P}rocess in
  {I}nfinite {D}imension: {C}entral {L}imit {T}heorem and {L}angevin
  {A}pproximation.
\newblock arXiv:1211.1894, 2012.

\bibitem{Jacobsen}
M.~Jacobsen.
\newblock {\em Point Process Theory and Applications: Marked Point and
  Piecewise Deterministic Processes}.
\newblock Birkh\"auser, Boston, 2006.

\bibitem{Koch}
C.~Koch.
\newblock {\em Biophysics of Computation}.
\newblock Oxford University Press, New York, 1999.

\bibitem{Kotelenez1}
P.~Kotelenez.
\newblock Law of large numbers and central limit theorem for linear chemical
  reactions with diffusion.
\newblock {\em Ann.~Prob.}, 14(1):173--193, 1986.

\bibitem{Kotelenez2}
P.~Kotelenez.
\newblock Linear parabolic differential equations as limits of space-time jump
  {M}arkov processes.
\newblock {\em J.~Math.~Anal.~Appl.}, 116(1):42--76, 1986.

\bibitem{Kurtz1}
T.~G. Kurtz.
\newblock Solutions of ordinary differential equations as limits of pure jump
  {M}arkov processes.
\newblock {\em J.~Appl.~Prob.}, 7:49--58, 1970.

\bibitem{Kurtz2}
T.~G. Kurtz.
\newblock Limit theorems for a sequence of jump {M}arkov processes
  approximating ordinary differential equations.
\newblock {\em J.~Appl.~Prob.}, 8:344--356, 1971.

\bibitem{Wainrib1}
K.~Pakdaman, M.~Thieullen, and G.~Wainrib.
\newblock Fluid limit theorems for stochastic hybrid systems with application
  to neuron models.
\newblock {\em Adv.~in Appl.~Probab.}, 42(3):761--794, 2010.

\bibitem{Wainrib2}
K.~Pakdaman, M.~Thieullen, and G.~Wainrib.
\newblock Asymptotic expansion and central limit theorem for multiscale
  piecewise-deterministic markov processes.
\newblock {\em Stoch. Proc. and their Appl.}, 122(6):2292Ð--2318, 2012.

\bibitem{RiedlerPhD}
M.~G. Riedler.
\newblock {\em Spatio-temporal Stochastic Hybrid Models of Excitable Biological
  Membranes}.
\newblock PhD thesis, Heriot-Watt University, 2011.

\bibitem{BR_neural_fields}
M.~G. Riedler and E.~Buckwar.
\newblock Laws of large numbers and {L}angevin approximations for stochastic
  neural field equations.
\newblock {\em J.~Math.~Neuroscience}, 3(1):54 p., 2013.

\bibitem{RTW}
M.~G. Riedler, M.~Thieullen, and Wainrib G.
\newblock Limit theorems for infinite-dimensional piecewise deterministic
  processes and applications to stochastic neuron models.
\newblock {\em Elect.~J.~Prob.}, 17(55):48, 2012.

\bibitem{Werner}
D.~Werner.
\newblock {\em Funktionalanalysis, 2.~Aufl.}
\newblock Springer, Berlin, 1997.

\bibitem{YinZhang}
G.~G. Yin and I.~Zhang.
\newblock {\em Continuous-Time Markov Chains and Applications: A Singular
  Perturbation Approach}.
\newblock Springer, New York, 1998.

\end{thebibliography}

\end{document}